\documentclass[11pt]{amsart}

\usepackage{amssymb}

\usepackage{amsmath,mathrsfs,amsthm}
\usepackage{enumerate,amsfonts,stmaryrd,graphicx,psfrag}
\usepackage[all]{xy}

\usepackage{color}

\usepackage{amscd}

\usepackage[margin=1in]{geometry}

\usepackage[bookmarksnumbered=true, bookmarksopen=true]{hyperref}

\newcommand{\bC}{\mathbb{C}}
\newcommand{\bE}{\mathbb{E}}
\newcommand{\bF}{\mathbb{F}}

\newcommand{\bP}{\mathbb{P}}
\newcommand{\bQ}{\mathbb{Q}}

\newcommand{\bZ}{\mathbb{Z}}

\newcommand{\bV}{\mathbb{V}}

\newcommand{\cA}{\mathcal{A}}

\newcommand{\cC}{\mathcal{C}}
\newcommand{\cD}{\mathcal{D}}

\newcommand{\cF}{\mathcal{F}}

\newcommand{\cM}{\mathcal{M}}
\newcommand{\cN}{\mathcal{N}}
\newcommand{\cO}{\mathcal{O}}
\newcommand{\cP}{\mathcal{P}}

\newcommand{\cS}{\mathcal{S}}
\newcommand{\cT}{\mathcal{T}}
\newcommand{\cU}{\mathcal{U}}

\newcommand{\cX}{\mathcal{X}}
\newcommand{\cY}{\mathcal{Y}}
\newcommand{\cZ}{\mathcal{Z}}

\newcommand{\fr}{\mathfrak{r}}
\newcommand{\fC}{\mathfrak{C}}
\newcommand{\fD}{\mathfrak{D}}
\newcommand{\fM}{\mathfrak{M}}
\newcommand{\fT}{\mathfrak{T}}
\newcommand{\fX}{\mathfrak{X}}
\newcommand{\fY}{\mathfrak{Y}}

\newcommand{\orb}{\mathrm{orb}}
\newcommand{{\inv} }{\mathrm{inv}}
\newcommand{\ev}{\mathrm{ev}}
\newcommand{\Aut}{\mathrm{Aut}}
\newcommand{\pt}{\mathrm{pt}}

\newcommand{\rk}{\operatorname{rk}}

\newcommand{\vir}{{\mathrm{vir}}}

\newcommand{\Spec}{\operatorname{Spec}}
\newcommand{\Hilb}{\operatorname{Hilb}}
\newcommand{\reg}{\mathrm{reg}}
\newcommand{\DT}{\mathrm{DT}}
\newcommand{\GW}{\mathrm{GW}}
\newcommand{\Sym}{\operatorname{Sym}}

\newcommand{\QH}{\operatorname{QH}}
\newcommand{\DR}{\operatorname{DR}}

\newcommand{\pr}{\operatorname{pr}}
\newcommand{\vdim}{\operatorname{vdim}}

\newcommand{\irreg}{\mathrm{irreg}}
\newcommand{\rel}{\mathrm{rel}}

\newcommand{\sh}{\mathrm{sh}}

\newcommand{\one}{\mathbf{1}}

\newcommand{\tC}{{\widetilde{C}}}

\newcommand{\tP}{ {\widetilde{P}} }

\newcommand{\tpi}{ {\tilde{\pi}} }

\newcommand{\Mbar}{\overline{\cM}}
\newcommand{\Nbar}{\overline{\cN}}

\newcommand{\mubar}{\overline{\mu}}
\newcommand{\nubar}{\overline{\nu}}
\newcommand{\rhobar}{\overline{\rho}}
\newcommand{\psibar}{\overline{\psi}}

\newcommand{\etabar}{\overline{\eta}}
\newcommand{\lambdabar}{\overline{\lambda}}

\newtheorem{dummy}{dummy}[section]
\newtheorem{lemma}[dummy]{Lemma}
\newtheorem{theorem}[dummy]{Theorem}
\newtheorem{conjecture}[dummy]{Conjecture}

\newtheorem{corollary*}[dummy]{Corollary*}
\newtheorem{proposition}[dummy]{Proposition}

\theoremstyle{definition}
\newtheorem{definition}[dummy]{Definition}
\newtheorem{remark}[dummy]{Remark}

\numberwithin{equation}{section}

\allowdisplaybreaks

\begin{document}

\title{Gromov--Witten theory of $[\bC^2/\bZ_{n+1}]\times \bP^1$}

\author{Zijun Zhou}
\address{Zijun Zhou, Kavli Institute for the Physics and Mathematics of the Universe (WPI), The University of Tokyo Institutes for Advanced Study, The University of Tokyo, Kashiwa, Chiba 277-8583, Japan}
\email{zijun.zhou@ipmu.jp}

\author{Zhengyu Zong}
\address{Zhengyu Zong, Yau Mathematical Sciences Center, Tsinghua University, Jin Chun Yuan West Building,
Tsinghua University, Haidian District, Beijing 100084, China}
\email{zyzong@mail.tsinghua.edu.cn}

\maketitle

\begin{abstract}

We compute the relative orbifold Gromov--Witten invariants of $[\bC^2/\bZ_{n+1}]\times \bP^1$, with respect to vertical fibers. Via a vanishing property of the Hurwitz--Hodge bundle, 2-point rubber invariants are calculated explicitly using Pixton's formula for the double ramification cycle, and the orbifold quantum Riemann--Roch. As a result parallel to its crepant resolution counterpart for $\cA_n$, the GW/DT/Hilb/Sym correspondence is established for $[\bC^2/\bZ_{n+1}]$. The computation also implies the crepant resolution conjecture for relative orbifold Gromov--Witten theory of $[\bC^2/\bZ_{n+1}]\times \bP^1$.

\end{abstract}

\tableofcontents

\section{Introduction}

\subsection{Overview}

Upon its emergence, the GW/DT correspondence has aroused plenty of interest in mathematical physics. The story begins with the technique of topological vertex invented in \cite{AKMV} to compute Gromov--Witten invariants for toric Calabi--Yau 3-folds, where generating functions for GW invariants are expressed as a summation over partitions. As observed by \cite{Oko-Res-Vaf, MNOP1, MNOP2}, this combinatorial feature can be interpreted in terms of another enumerative theory --- the Donaldson--Thomas theory. Lots of work has been done after this discovery, including the GW/DT correspondence for local curves \cite{Bry-Pan, Oko-Pan} and its generalization to $\cA_n\times \bP^1$ \cite{Mau, Mau-Ob2}. Here $\cA_n$ is defined as the minimal resolution of the singular quotient $\bC^2/\bZ_{n+1}$, where the cyclic group
$$\bZ_{n+1}:= \bZ/(n+1)\bZ = \{\zeta\in \bC | \zeta^{n+1}=1 \}$$
acts on $\bC^2$ in the anti-diagonal manner:
$$\zeta\cdot (x,y):= (\zeta x, \zeta^{-1} y).$$

The resolution $\cA_n \rightarrow \bC^2/\bZ_{n+1}$ is a crepant resolution, meaning that it preserves the canonical class. On the other hand, there is an obvious resolution of the same singularity in the category of orbifolds, the stacky quotient $[\bC^2/\bZ_{n+1}]$. In the spirit of the crepant resolution conjecture \cite{Rua, Bry-Gra}, one expects a GW/DT correspondence for $[\bC^2/\bZ_{n+1}]\times \bP^1$, which should be closely related to that for $\cA_n\times \bP^1$.

\subsection{Summary of results}

Let $\cX:= [\bC^2/\bZ_{n+1}]\times \bP^1$ be our target, and $\cD = \coprod_{i=1}^r [\bC^2/\bZ_{n+1}]\times \{z_i\}$ be a disjoint union of vertical fibers, where $z_1,\cdots,z_r$ are distinct points on $\bP^1$. A relative stable map from an orbifold nodal curve $C$ to $\cX$, relative to $\cD$, is a map from $C$ to a modified target $\cX[k]$, for some $k$. $\cX[k]$ is defined by gluing $\cX$ along $\cD$ with $k$ copies of ``bubbles", constructed by the projective completion of the normal bundle of $\cD$ in $\cX$. The map $C\rightarrow \cX[k]$ is required to be stable, and satisfy certain transversality conditions. For the precise definition and detailed discussions on orbifold relative GW theory, we refer the readers to \cite{Abr-Fan}.

Let $m> 0$ be a fixed integer. Consider the moduli space of such relative stable maps,
$$\Mbar_{g, \gamma} (\cX, \mubar^1, \cdots, \mubar^r),$$
where $g$ is the genus of domains, $\gamma = (\gamma_1, \cdots, \gamma_p)$ is a tuple of elements in $\bZ_{n+1}$ indicating the monodromies of non-relative marked points, and $\mubar^1, \cdots, \mubar^r$ are $\bZ_{n+1}$-weighted partitions of $m$. The partition $\mubar^i$ records the ramification profile of the stable map with the $i$-th divisor, where the decoration of each part remembers the monodromy of the corresponding relative marked point.

Let $T$ be the 2-dimensional torus acting on the fiber. Note that $\cX$, and hence the moduli space, are noncompact, but admit a $T$-action with compact fixed loci. The moduli space of relative stable maps is equipped with a $T$-equivariant perfect obstruction theory. Hence by $T$-localization one can define the GW invariants, written as correlation functions
$$\langle \mubar^1, \cdots, \mubar^r \rangle^{\cX, \circ}_{g,\gamma}.$$
Here the circle $\circ$ denotes the \emph{connected} theory which means that the domain curve is required to be connected. The above GW invariants can be packaged into a generating function
$$Z'_\GW(\cX)^{\circ}_{\mubar^1, \cdots, \mubar^r}:= \sum_{g\geq 0} \sum_\gamma  z^{2g-2} \frac{x^\gamma}{\gamma!} \langle \mubar^1, \cdots, \mubar^r \rangle^{\cX, \circ}_{g, \gamma}.$$
Similarly, one can also consider the \textit{disconnected} GW invariants,
$$\langle \mubar^1, \cdots, \mubar^r \rangle^{\cX, \bullet}_{\chi,\gamma}.$$
Here the dot $\bullet$ denotes the \emph{disconnected} theory which means that the domain curve can be disconnected and $\chi$ is Euler characteristic of the domain curve. Here we \emph{do not allow contracted connected components}. One can similarly define the generating function
$$Z'_\GW(\cX)_{\mubar^1, \cdots, \mubar^r}:= \sum_{\chi,\gamma}   z^{-\chi} \frac{x^\gamma}{\gamma!} \langle \mubar^1, \cdots, \mubar^r \rangle^{\cX, \bullet}_{\chi, \gamma}.$$

Another equivalent way to think about this is to identify the moduli with that of relative stable maps to $\cY:= B\bZ_{n+1}\times \bP^1$, and consider the GW theory twisted by the obstruction bundle associated with the normal bundle of $\cY\subset \cX$.

By the orbifold GW degeneration formula \cite{Abr-Fan}, the computation of $r$-point functions reduces to that of $3$-point functions. If we assume the generation conjecture (see Section 6), one can further reduce it to the case when one of the three partitions is of the form $(1,0)^m, (2,0)(1,0)^{m-2},$ or $(1,k)(1,0)^{m-1}$. In this case, one can reduce the $3$-point functions to $2$-point \emph{rubber} invariants by a rigidification argument. It turns out that the obstruction bundle is only nontrivial on certain simple strata of the moduli, and an application of Pixton's formula for the double ramification cycle leads to an explicit formula for the rubber invariants. We are able to compare the result with Maulik \cite{Mau} after a change of variables.

The first main result of this paper is about the crepant resolution conjecture for the GW theory of $\cX$. More explicitly, one can define a certain generating function $Z'_\GW(\cA_n\times \bP^1)_{\vec\mu^1, \cdots, \vec\mu^r}$ of GW invariants of $\cA_n\times \bP^1$. The definition of $Z'_\GW(\cA_n\times \bP^1)_{\vec\mu^1, \cdots, \vec\mu^r}$ is similar to that of $Z'_\GW(\cX)_{\mubar^1, \cdots, \mubar^r}$. The main difference is that there are certain variables $s_j, 1\leq j\leq n$ in $Z'_\GW(\cA_n\times \bP^1)_{\vec\mu^1, \cdots, \vec\mu^r}$ encoding the degree of the map from the domain curve to the $\cA_n$ component. The crepant resolution conjecture states that $Z'_\GW(\cA_n\times \bP^1)_{\vec\mu^1, \cdots, \vec\mu^r}$ and $Z'_\GW(\cX)_{\mubar^1, \cdots, \mubar^r}$ are equal under certain change of variables between $s_1,\cdots,s_n$ and $x_1,\cdots,x_n$. Theorem \ref{Thm-GW-CRC} below is about the crepant resolution conjecture of the three fundamental cases. From the discussion above, we know that Theorem \ref{Thm-GW-CRC} implies the general crepant resolution conjecture if we assume the generation conjecture (see Section 6).

Let $S$ be an smooth orbifold surface. Denote by $\cF_S$ the vector space spanned by $H^*_\orb(S)$-weighted partitions of $m$, which we call the Fock space.

\begin{theorem}[GW crepant resolution, Theorem \ref{GW-CRC}] \label{Thm-GW-CRC}
	Given $\mubar, \nubar, \rhobar\in \cF_{[\bC^2/\bZ_{n+1}]}$, with
	$$\rhobar = (1,0)^m, \qquad (2,0)(1,0)^{m-2}, \qquad \textrm{or} \qquad (1,k)(1,0)^{m-1},$$
	let $\vec\mu, \vec\nu, \vec\rho\in \cF_{\cA_n}$ be their correspondents. Then under the change of variables
	$$s_j = \zeta \exp \left( \frac{1}{n+1} \sum_{a=1}^n (\zeta^{a/2}- \zeta^{-a/2}) \zeta^{ja} x_a \right), \qquad 1\leq j\leq n,$$
	we have:
	\begin{enumerate}[1)]
		\item When $l(\mubar)+ l(\nubar)=2$, and $\rhobar = (1,0)^m$ or $(2,0)(1,0)^{m-2}$,
		$$Z'_\GW(\cA_n\times \bP^1)_{\vec\mu, \vec\nu, \vec\rho} = Z'_\GW([\bC^2/\bZ_{n+1}]\times \bP^1)_{\mubar, \nubar, \rhobar};$$
		\item When $l(\mubar)+ l(\nubar)\geq 3$, or $\rhobar =(1,k)(1,0)^{m-1}$,
		$$Z'_{\GW,\beta\neq 0}(\cA_n\times \bP^1)_{\vec\mu, \vec\nu, \vec\rho} = Z'_\GW([\bC^2/\bZ_{n+1}]\times \bP^1)_{\mubar, \nubar, \rhobar}.$$
	\end{enumerate}
\end{theorem}

Here $\vec\mu$, $\vec\nu$ and $\mubar$, $\nubar$ are identified via the explicit isomorphism
$$\Phi: H^*_\orb([\bC^2/\bZ_{n+1}]) \cong H^*(\cA_n),$$
\begin{equation} \label{correspondence}
e_0 \mapsto 1, \qquad e_i \mapsto \frac{\zeta^{i/2} - \zeta^{-i/2}}{n+1} \sum_{j=1}^n \zeta^{ij} \omega_j, \qquad 1\leq i\leq n,
\end{equation}
where $\omega_1, \cdots, \omega_n \in H^2(\cA_n, \bQ)$ is the dual basis to the exceptional curves in $\cA_n$. The discrepancy between Case 1) and 2) in the theorem is due to the fact that $\Phi$ only preserves the Poincar\'e pairing, but not the ring structures of the cohomology.

As a byproduct, we observe that the moduli space of genus-0 stable maps to $\Sym^m([\bC^2/\bZ_{n+1}])$ shares a common open substack with the moduli of relative stable maps to $\cX$. Moreover, the obstruction bundles coincide and vanish outside of this open substack. Thus our computation also leads to a formula for the orbifold quantum cohomology of the symmetric product $\Sym^m([\bC^2/\bZ_{n+1}])$.

\begin{theorem}[GW/Sym correspondence, Theorem \ref{GW-Sym}]
  Given $\mubar, \nubar, \rhobar\in \cF_{[\bC^2/\bZ_{n+1}]}$, with
  $$\rhobar = (1,0)^m, \qquad (2,0)(1,0)^{m-2}, \qquad \textrm{or} \qquad (1,k)(1,0)^{m-1},k\neq 0,$$
  we have
  $$z^{l(\mu) + l(\nu) + l(\rho) - m} Z'_\GW([\bC^2/\bZ_{n+1}]\times \bP^1)_{\mubar, \nubar, \rhobar} =  \langle \mubar, \nubar, \rhobar \rangle_{\Sym^m ([\bC^2/\bZ_{n+1}])},$$
  where the right hand side is the 3-point genus-zero orbifold GW invariants of $\Sym^m ([\bC^2/\bZ_{n+1}])$.
\end{theorem}

In \cite{Zhou2}, the first named author proved the crepant resolution conjecture for relative DT invariants of $\cX$, via a further DT/Hilb correspondence to the quantum cohomology of $\Hilb^m([\bC^2/\bZ_{n+1}])$. Combining Theorem \ref{Thm-GW-CRC} and these results with the GW/DT correspondence for $\cA_n\times \bP^1$, we obtain the following.

\begin{theorem}[GW/DT correspondence, Theorem \ref{GW-DT}]
	Given $\mubar, \nubar, \rhobar\in \cF_{[\bC^2/\bZ_{n+1}]}$, with
	$$
	\rhobar = (1,0)^m, \qquad (2,0)(1,0)^{m-2}, \qquad \textrm{or} \qquad (1,k)(1,0)^{m-1},
	$$
	let $\vec\mu$, $\vec\nu$, $\vec\rho$ be their correspondents in $\cF_{\cA_n}$. Then
	under the change of variables
	$$
	Q =q_0 q_1 \cdots q_n = -e^{iz}, \qquad q_j = \zeta \exp \left( \frac{1}{n+1} \sum_{a=1}^n (\zeta^{a/2}- \zeta^{-a/2}) \zeta^{ja} x_a \right), \qquad 1\leq j\leq n,
	$$
	we have
	\begin{enumerate}[1)]
		\item when $l(\mubar)+ l(\nubar)=2$,  and $\rhobar = (1,0)^m$ or $(2,0)(1,0)^{m-2}$,
		$$
		(-iz)^{l(\mu)+l(\nu)+l(\rho)-m} Z'_\GW([\bC^2/\bZ_{n+1}]\times \bP^1)_{\mubar, \nubar, \rhobar} = (-1)^m Z'_\DT([\bC^2/\bZ_{n+1}]\times\bP^1)_{\vec\mu, \vec\nu, \vec\rho};
		$$
		\item when $l(\mubar)+ l(\nubar)\geq 3$, or $\rhobar =(1,k)(1,0)^{m-1}$,
		$$
		(-iz)^{l(\mu)+l(\nu)+l(\rho)-m} Z'_\GW([\bC^2/\bZ_{n+1}]\times \bP^1)_{\mubar, \nubar, \rhobar} = (-1)^m Z'_{\DT, \varepsilon-irreg} ([\bC^2/\bZ_{n+1}]\times\bP^1)_{\vec\mu, \vec\nu, \vec\rho}.
		$$
	\end{enumerate}
\end{theorem}

In conclusion, we obtain a GW/DT/Hilb/Sym correspondence on the $[\bC^2/\bZ_{n+1}]$ level, which can be viewed as a crepant resolution/transformation correspondent to its parallel picture on the $\cA_n$ level. The relationship among these theories can be summarized in the following diagram.

\begin{figure}[h]
\begin{center}
\psfrag{A}{$\QH(\Sym(\cA_n))$}
\psfrag{B}{$\GW(\cA_n\times \bP^1)$}
\psfrag{C}{$\DT(\cA_n\times \bP^1)$}
\psfrag{D}{$\QH(\Hilb(\cA_n))$}
\psfrag{E}{$\QH(\Sym([\bC^2/\bZ_{n+1}]))$}\psfrag{F}{$\GW([\bC^2/\bZ_{n+1}]\times \bP^1)$}
\psfrag{G}{$\DT([\bC^2/\bZ_{n+1}]\times \bP^1)$}
\psfrag{H}{$\QH(\Hilb([\bC^2/\bZ_{n+1}]))$}
\hspace*{-2.8cm}\includegraphics[scale=0.5]{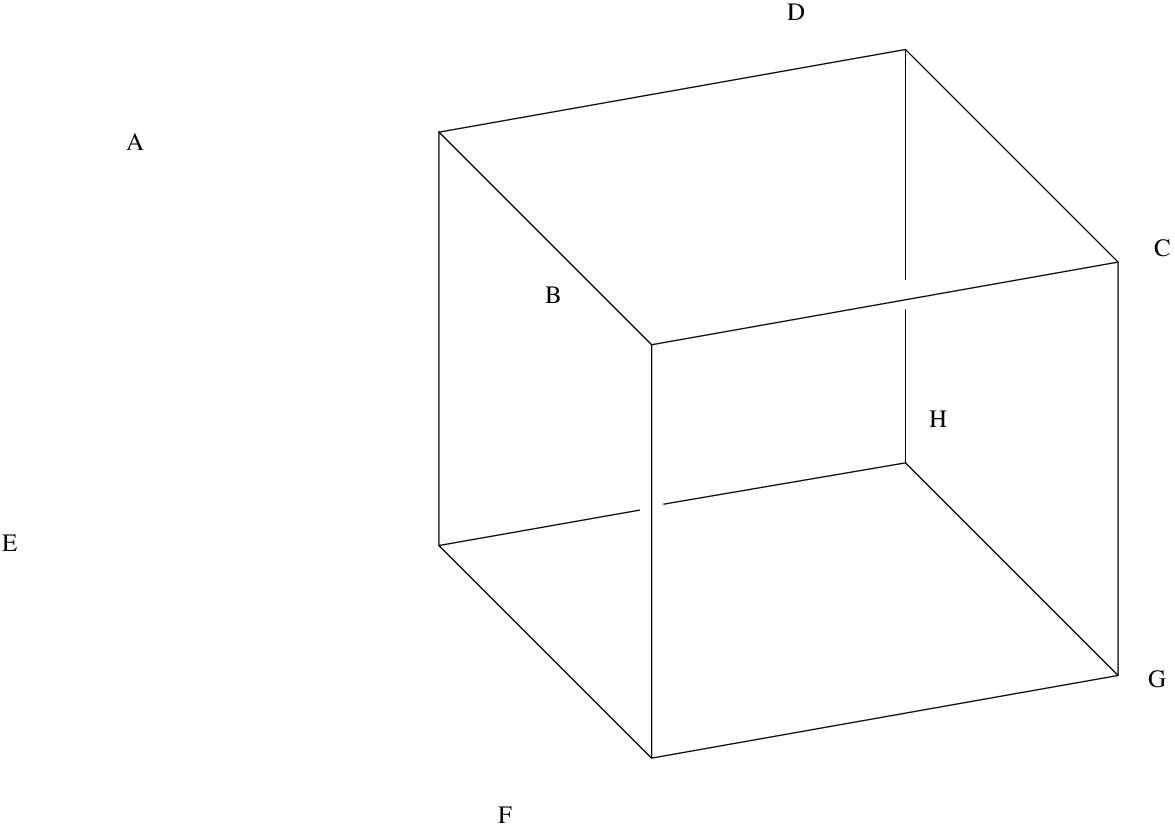}
\end{center}
\end{figure}


The paper is organized as follows. In Section 2 and 3, we explain in detail the definition of relative GW invariants, and how one can reduce the calculation of 3-point functions with one divisor insertion to 2-point rubber invariants. In Section 4, we prove the vanishing property of the obstruction bundle, and use Pixton's formula for the double ramification cycle to calculate the 2-point rubber invariants. Following the calculation of J. Zhou in \cite{Zhou}, we apply the change of variables and prove Theorem \ref{Thm-GW-CRC}. In Section 5, we discuss the orbifold quantum cohomology of symmetric products and obtain the GW/Sym correspondence. Finally, Section 6 is a summary of all existing results, where we prove the GW/DT correspondence for $[\bC^2/\bZ_{n+1}]\times \bP^1$.

\subsection{Acknowledgements}

The authors would like to thank Emily Clader, Felix Janda, and Dustin Ross for their wonderful talks at Columbia University on the double ramification cycle, which brought our attention to Pixton's formula. Moreover, the authors would like to express their acknowledgements to Professor Chiu-Chu Melissa Liu, for suggesting the whole project, and many useful conversations and discussions. The first author was supported by the National Science Foundation under Grant No. DMS-1440140 while the author was in residence at the Mathematical Sciences Research Institute in Berkeley, California, during the Spring 2018 semester; the first author is also supported by World Premier International Research Center Initiative (WPI), MEXT, Japan. The second author is supported by NSFC grant No. 11701315 and the start up grant of Tsinghua University.

\section{Geometry}

\subsection{Relative stable maps} \label{rel-moduli}

In this subsection, we briefly recall the definition of relative stable maps, mainly in the sense of Abramovich--Fantech \cite{Abr-Fan}. For the detailed definitions and constructions, we refer the reader to the original reference.

Let $X$ be a smooth DM stack, and $D\subset X$ be a smooth effective divisor. An \emph{expanded pair} of length $k$ is obtained by attaching to $X$ along $D$ a chain of $k$ ``bubble components":
$$
X[k]:= X \cup_D \Delta \cup_D \cdots \cup_D \Delta.
$$
Here $\Delta = \bP_D (\cO \oplus N_{D / X})$, and the gluings are along $D\subset X$, and the $0$- and $\infty$-sections of $\Delta$, such that at each singular divisor, which is isomorphic to $D$, the normal bundles of $D$ in the two $\Delta$'s containing it are inverse to each other. We call the smooth divisor $D$ in the last $\Delta$ the distinguished divisor. Expanded pairs can form families, where they are allowed to degenerate and produce new bubbles in a nice way. Dating back to J. Li \cite{Li, GV}, there is a smooth Artin stack $\fT$, parameterizing all possible expanded pairs, with a universal family $\fY \to \fT$, and the universal distinguished divisor $\fD \subset \fY$.

The idea is further developed in \cite{ACFW, Abr-Fan}, where twisting structures are introduced to treat the orbifold case. For the expanded pair $X[k]$, let $\fr_0, \cdots, \fr_k$ be a tuple of non-negative integers. The twisted expanded pair is defined as
$$
X[k] (\fr_1, \cdots, \fr_k) := X ( \sqrt[\fr_1]{D} ) \cup \Delta_1 \cup \cdots \cup \Delta_k,
$$
where $\Delta_i := \Delta (\sqrt[\fr_i]{D_0}, \sqrt[\fr_{i+1}]{D_\infty} )$, and the gluing is along the root gerbes in a balanced manner. There is also a notion of families of twisted expanded pairs, given a \emph{twisting choice} $\fr$, which is a compatible choice of the twisting indices for singular and distinguished divisors in families. Similarly those form a smooth Artin stack $\fT^\fr$, together with a universal family $\fY^\fr$, parameterizing all $\fr$-twisted expanded pairs.

\begin{definition}
A \emph{relative stable map} \footnote{These are called predeformable maps in \cite{Abr-Fan}.} is a representable map $f$ from a nodal prestable orbifold curve $\cC$ to an expanded pair $X[k]$, for some $k\geq 0$, such that

(i) no irreducible components map to any singular divisor or the distinguished divisor;

(ii) only nodes map to singular divisors, and locally at a node of $\cC$, $f$ is of the form 
$$
\Spec \left( \bC[u, v] / (uv) \right)^\sh \to \Spec \left( \bC [x,y, z_1, \cdots, z_m] / (xy) \right)^\sh, \qquad (x,y) \mapsto (u^c, v^c), 
$$
for some positive integer $c$; 

(iii) $f$ admits finitely many automorphisms.

In particular, $f$ is called \emph{transversal} if the contact orders $c$ at all nodes and the distinguished divisor are $c=1$. Similarly, one can define the twisted version, with $X[k]$ replaced by $X[k] (\fr)$.
\end{definition}

Let $\Mbar_\Gamma (X, D)$ be the moduli space of all relative stable maps, with topological data $\Gamma$. Here $\Gamma$ denotes all topological information such as genus, degree, contact orders along the disinguish divisor, which makes the moduli space of finite type. Similarly, given a twisting choice $\fr$, we define $\Mbar_\Gamma^\fr (X, D)$ as the moduli space of all \emph{transversal} relative stable maps, with topological data $\Gamma$. By results in Section 3 of \cite{Abr-Fan}, both $\Mbar_\Gamma (X, D)$ and $\Mbar_\Gamma^\fr (X, D)$ are proper DM stacks. There is a proper map $\Mbar^\fr_\Gamma (X, D) \to \Mbar_\Gamma (X, D)$ forgetting all the twisting structures.

The crucial idea of the approach by Abramovich--Fantechi to the relative GW theory is that, compared to the predeformable condition that defines $\Mbar_\Gamma (X, D)$, which is a locally closed condition, the transversal condition defining $\Mbar^\fr_\Gamma (X, D)$ is an open condition, which makes its deformation-obstruction theory much easier to work with. Therefore, a perfect obstruction theory exists for $\Mbar_\Gamma^\fr (X, D)$, rendering a virtual cycle.

Moreover, the obstruction theory is independent of the twisting choice $\fr$, in the following sense. There is a partial order $\prec$ in the set of all twisting choices. Given $\fr, \fr'$ such that $\fr \prec \fr'$, there is a Cartesian diagram
$$
\xymatrix{
\Mbar_\Gamma^{\fr'} (X, D) \ar[r]^\phi \ar[d] & \Mbar_\Gamma^\fr (X, D) \ar[d] \\
\fT^{\fr'} \ar[r] & \fT^\fr ,
}
$$
with $\phi$ flat and proper \footnote{So that the pull-back $\phi^*$ and push-forward $\phi_*$ are well-defined for cycles. }, such that the obstruction theory pulls back along $\phi$, and hence $\phi^* [\Mbar_\Gamma^{\fr} (X, D)]^\vir = [\Mbar_\Gamma^{\fr'} (X, D)]^\vir$. The fact that the map $\fT^{\fr'} \to \fT^\fr$ is generically of degree one implies that $\phi_* [\Mbar_\Gamma^{\fr'} (X, D)]^\vir = [\Mbar_\Gamma^{\fr} (X, D)]^\vir$, due to Costello's pushforward formula \cite{Cos}. Given furthermore the compatibility of the maps with the partial ordering $\prec$, we can define the virtual class on $\Mbar_\Gamma (X, D)$ in the following way.

\begin{definition} \label{vc}
Take a sufficiently large twisting choice $\fr$ (which always exists), with the map $\pi: \Mbar_\Gamma^\fr (X, D) \to \Mbar_\Gamma (X, D)$. Define $[\Mbar_\Gamma (X, D)]^\vir := \pi_* [\Mbar_\Gamma^\fr (X, D)]^\vir$.
\end{definition}

When $X$ is a scheme, by the comparison result (Theorem 1.1 in \cite{AMW}), this definition agrees with the one defined by J. Li \cite{Li2}.

\subsection{Geometry of $[\bC^{2}/\bZ_{n+1}]\times \bP^1$}
Fix an integer $n\geq 0$. Let $\bP^1$ be the projective line and $\cO_{\bP^1}$ be the trivial line bundle on it. Let $\cY$ be the trivial $\bZ_{n+1}$-gerbe over $\bP^1$, coming from the root construction \cite{Abr-Gra-Vis, Cad} of order $n+1$ on $\cO_{\bP^1}$. In other words, $\cY$ is defined by the following Cartesian diagram.
\[\begin{CD}
\cY   	@>>>          B\bC^*       \\
@VVV                             @VVV\lambda\mapsto\lambda^{n+1}\\
\bP^1                  @>\cO_{\bP^1}>>      B\bC^*.
\end{CD}\]

There is an orbifold line bundle $L$ on $\cY$ associated with the top map in the above diagram. The degree of $L$ is zero, but there is a nontrivial action by $\bZ_{n+1}$ on the fibers of $L$, for which the generator acts by multiplication with $\zeta:=e^{\frac{2\pi\sqrt{-1}}{n+1}}$. We will be interested in the relative GW theory of the total space $L\oplus L^{-1}\to\cY$, which is isomorphic to
$$\cX:=[\bC^{2}/\bZ_{n+1}]\times \bP^1,$$
where the generator $\zeta\in \bZ_{n+1}$ acts on $(x,y)\in \bC^{2}$ by $(\zeta x,\zeta^{-1} y)$.

Let $T=(\bC^*)^2$ be the 2-dimensional algebraic torus. Consider the standard action of $T$ on $\bC^2$, with $T$-characters $(n+1)t_1$, $(n+1)t_2$, which induces a $T$-action on $[\bC^{2}/\bZ_{n+1}]$ with characters $t_1$, $t_2$. Identifying $\cX= [\bC^{2}/\bZ_{n+1}]\times \bP^1$ with the total space of $L\oplus L^{-1}\to\cY$, we have a $T$-action on the fibers of $L\oplus L^{-1}\to\cY$ via characters $t_1$, $t_2$.

\subsection{Relative stable maps to $\cY = B \bZ_{n+1} \times \bP^1$} \label{moduli}

Let $z_1, \cdots, z_r$ be $r$ points on $\cY$. We apply the construction in Subsection \ref{rel-moduli} to the case $X = \cY$, and $D = \{z_1, \cdots, z_r\}$. The divisor $D$ here is disconnected, but the theory can be easily generalized to this case. For example, an expanded pair would be the gluing of $\cY$ with $r$ chains (of different lengths) of $B\bZ_{n+1} \times \bP^1$'s, along the points $z_1, \cdots, z_r$.

We will consider relative stable maps to $\cY$, with marked points on the domain, which can be \emph{non-relative}, which do not map into any singular or distinguished divisors of the expanded target, or \emph{relative}, which map into the distinguished divisors. The topological datum will be described by the following notion.

\begin{definition}
Let $m$ be a positive integer. A \emph{$\bZ_{n+1}$-weighted partition} of $m$
$$
\mubar=\{(\mu_1,k_1),\cdots , (\mu_{l(\mu)}, k_{l(\mu)})\}
$$
means the following: $\mu:=\{\mu_1,\cdots, \mu_{l(\mu)}\}$ is an ordinary partition of $m$, and each part is decorated with an element $k_i\in\bZ_{n+1}$, $i=1,\cdots,l(\mu)$.

Define the subset $A'(\bar\mu)$ and $A''(\bar\mu)$, such that
$$A'(\mubar)\sqcup A''(\mubar) = \{1,\cdots, l(\mubar)\},$$
where $k_i=0$ if and only if $i\in A'(\mubar)$. Denote $l'(\mubar)=|A'(\mubar)|$ and $l''(\mubar)=|A''(\mubar)|$.

For any $\mubar$, we use the notation $-\mubar$ to denote $\{(\mu_1,-k_1),\cdots , (\mu_{l(\mu)}, -k_{l(\mu)})\}$.
\end{definition}

Let $f$ be a relative stable map of degree $m$. For each distinguished divisor $z_i$, the contact order of $f$ along the relative marked points can be recorded by a partition of $m$, which together with the monodromy data, form a $\bZ_{n+1}$-weighted partition $\mubar^i$ of $m$. The contact order of the $j$-th relative marked point at the $z_i$ is recorded as $\mubar^i_j$, whose monodromy recorded by the decoration. We call this $\mubar^i$ the ramification profile of $f$ at $z_i$.

Suppose there are $p$ non-relative marked points, whose monodromies are denoted by $\gamma_1, \cdots, \gamma_p \in \bZ_{n+1}$. We only consider the case where all $\gamma_i$'s are nontrivial, i.e. $\gamma_i \neq 0$.

\begin{definition}
Let $\mubar^1, \cdots, \mubar^r$ be $\bZ_{n+1}$-weighted partitions of $m$, and $\gamma = (\gamma_1, \cdots, \gamma_p) \in (\bZ_{n+1}^* )^p$. Let
$$
\overline{\cM}_{g, \gamma} (\cY, \mubar^1, \cdots, \mubar^r)
$$
be the moduli stack of relative stable maps to $(\cY, z_1, \cdots, z_r)$ with ramification profiles $\mubar^1, \cdots, \mubar^r$, and $p$ non-relative marked points, with (nontrivial) monodromies $\gamma_1, \cdots, \gamma_p$.
\end{definition}

In order for the moduli space $\Mbar_{g, \gamma}(\cY, \mubar^1,\cdots,\mubar^r)$ to be non-empty, we must have the condition
$$
\sum_{i=1}^{p}\gamma_i+\sum_{i=1}^{r}\sum_{j=1}^{l(\mu^i)}k^i_j=0\in\bZ_{n+1},
$$

The virtual dimension of $\Mbar_{g, \gamma}(\cY, \mubar^1,\cdots,\mubar^r)$ is equal to
$$
\vdim (\Mbar_{g, \gamma}(\cY, \mubar^1, \cdots, \mubar^r)) := 2g-2+p+\sum_{i=1}^{r}l(\mubar^i)-(r-2)m.
$$
We will also consider the disconnected version  $\Mbar^\bullet_{\chi, \gamma}(\cY, \mubar^1,\cdots,\mubar^r)$, where the domain curve $C$ is allowed to be disconnected, and $\chi := 2(h^0(\cO_{C})-h^1(\cO_{C}))$.

\subsection{Rubber moduli space}

We will also consider relative maps to the nonrigid target $\cY$. Let $\mubar$, $\nubar$ be two $\bZ_{n+1}$-weighted partitions of $m$. The \emph{rubber} moduli space
$$
\Mbar_{g, \gamma} (\cY, \mubar, \nubar)^\sim
$$
is defined to parameterize relative stable maps to $\cY$, relative to $0$ and $\infty$, but up to the $\bC^*$-scaling on $\cY$. In other words, two relative stable maps will be regarded as equivalent, if they differ by a $\bC^*$-action on $\cY$.

Alternatively, the rubber moduli space can be defined more precesly as follows. Consider the moduli space $\Mbar_{g, \gamma} (\cY, \nubar)$ with one relative divisor $D = \infty$. Its $\bC^*$-fixed loci consists of relative stable maps which are $\bC^*$-fixed over the rigid component $\cY$, but still admits continuous moduli over the bubble components. The rubber moduli space $\Mbar_{g, \gamma} (\cY, \mubar, \nubar)^\sim$ can then be defined as the open and closed substack in $\Mbar_{g, \gamma} (\cY, \nubar)^{\bC^*}$, consisting of relative stable maps whose ramification profile, at the first node connecting the rigid component $\cY$ and the bubbles, is $\mubar$.

Similarly, the disconnected version is denoted as $\Mbar^\bullet_{g, \gamma} (\cY, \mubar, \nubar)^\sim$.

\subsection{Obstruction bundle}\label{obs bundle}
Let $\pi:\cU\to \Mbar_{g, \gamma}(\cY, \mubar^1,\cdots,\mubar^r)$ be the universal domain curve and $\cT$ be the universal target. There is a universal map $F:\cU\to \cT$ and a contraction map $\tpi: \cT\to \cY$. Define
$$V_1 = R^1\pi_*\tilde{F}^*L, \qquad V_2 = R^1\pi_* \tilde{F}^*L^{-1},$$
where $\tilde{F}=\tilde\pi\circ F:\mathcal{U}\to \cY$. The rank of $V_1$ is
$$
g-1+\sum_{i=1}^{p}\frac{\gamma_i}{n+1}+\sum_{i=1}^{r}\sum_{j=1}^{l(\mu^i)}
\frac{k^i_j}{n+1}+\delta,
$$
where we identify $\bZ_{n+1}$ with the set $\{0,\cdots,n\}$ and $\delta$ is defined to be
$$\delta=\left\{\begin{array}{ll}1, & \qquad \textrm{if all monodromies on the domain curve are trivial,}\\
0, & \qquad \textrm{otherwise}.\end{array} \right.$$
Similarly, the rank of $V_2$ is
$$
g-1+\sum_{i=1}^{p}\frac{n+1-\gamma_i}{n+1}+\sum_{i=1}^{r}\sum_{j=1}^{l(\mu^i)}
\frac{n+1-k^i_j-(n+1)\delta_{0,k^i_j}}{n+1}+\delta,
$$
where
$$\delta_{0,x}=\left\{\begin{array}{ll}1, & \qquad x=0,\\
0, & \qquad x\neq 0.\end{array} \right.$$
So the rank of the bundle $V:=V_1\oplus V_2$
is equal to
$$
\rk(V) = 2g-2+p+\sum_{i=1}^{r}l''(\mubar^i)+2\delta.
$$

\section{Relative Gromov--Witten theory of $[\bC^{2}/\bZ_{n+1}] \times \bP^1$}\label{relative}

\subsection{Relative GW invariants} \label{relative-GW}
Recall that $\cX=[\bC^{2}/\bZ_{n+1}]\times \bP^1$, $\cY = B\bZ_{n+1} \times \bP^1$. Let $z_1,\cdots,z_r$ be $r$ points on $\cY$ and $\mubar^1,\cdots,\mubar^r$ be $\bZ_{n+1}$-weighted partitions of $m$. We are interested in the GW theory of $\cX$, relative to the $r$ fibers $[\bC^{2}/\bZ_{n+1}]\times \{z_i\}$.

Let $\gamma=(\gamma_{1}, \cdots, \gamma_{p})$ be a vector of nontrivial elements in $\bZ_{n+1}$, with $l(\gamma)=p$. Define
$$\Mbar_{g, \gamma}(\cX, \mubar^1,\cdots,\mubar^r)$$
to be the moduli space of relative stable maps to $(\cX, [\bC^{2}/\bZ_{n+1}]\times \{z_1\},\cdots, [\bC^{2}/\bZ_{n+1}]\times \{z_r\})$, with ramification profiles $\mubar^1,\cdots,\mubar^r$, and $p$ non-relative marked points $x_1,\cdots,x_p$ on the domain curve with monodromies $\gamma_1,\cdots,\gamma_p\in\bZ_{n+1}$. The moduli space $\Mbar_{g, \gamma}(\cX, \mubar^1,\cdots,\mubar^r)$ is noncompact.

Recall that we have a $T$-action on the fibers of $\cX$ with weights $t_1$, $t_2$, which induces a $T$-action on $\Mbar_{g, \gamma}(\cX, \mubar^1,\cdots,\mubar^r)$. The fixed loci of this action is
$$
\Mbar_{g, \gamma}(\cX, \mubar^1,\cdots,\mubar^r)^T=\Mbar_{g, \gamma}(\cY, \mubar^1,\cdots,\mubar^r).
$$
Therefore, although $\Mbar_{g, \gamma}(\cX, \mubar^1,\cdots,\mubar^r)$ is noncompact, the fixed loci $\Mbar_{g, \gamma}(\cX, \mubar^1,\cdots,\mubar^r)^T$ is compact. The relative GW invariants can be defined $T$-equivariantly as
$$
\langle \mubar^1,\cdots,\mubar^r\rangle^{\cX,\circ}_{g,\gamma}:= \frac{1}{|\Aut(\mu^1)|\cdots|\Aut(\mu^r)|} \int_{\left[ \Mbar_{g, \gamma}(\cX, \mubar^1,\cdots,\mubar^r)^T \right]^\vir}\frac{1}{e_T(N^\vir)},
$$
where $N^\vir$ is the virtual normal bundle and $e_T(-)$ is the $T$-equivairant Euler class. Here the factors $|\Aut(\mu^i)|$ come from the convention that we treat the relative marked points as \emph{unordered}.

In other words,
\begin{eqnarray*}
\langle \mubar^1,\cdots,\mubar^r\rangle^{\cX,\circ}_{g,\gamma} &=&
\frac{1}{|\Aut(\mu^1)|\cdots|\Aut(\mu^r)|}\int_{\left[ \Mbar_{g, \gamma}(\cY, \mubar^1,\cdots,\mubar^r) \right]^\vir}\frac{e_T(R^1\pi_*\tilde{F}^*(L\oplus L^{-1}))}{e_T(R^0\pi_*\tilde{F}^*(L\oplus L^{-1}))}\\
&=&\frac{1}{|\Aut(\mu^1)|\cdots|\Aut(\mu^r)|}\int_{\left[ \Mbar_{g, \gamma}(\cY, \mubar^1,\cdots,\mubar^r) \right]^\vir}\frac{e_T(V)}{e_T(R^0\pi_*\tilde{F}^*(L\oplus L^{-1}))},
\end{eqnarray*}
where $V$ is the obstruction bundle defined in the last section. The rank of $R^0\pi_*\tilde{F}^*(L\oplus L^{-1})$ is equal to $2\delta$ and $e_T(R^0\pi_*\tilde{F}^*(L\oplus L^{-1}))=\big(t_1t_2\big)^\delta$.

In later sections, invariants with $r$ relative insertions will be called \emph{$r$-point functions}. The notation
$$
\langle \mubar^1,\cdots,\mubar^r\rangle^{\cX, \bullet}_{\chi,\gamma},
$$
will denote the \emph{disconnected} $r$-point correlation functions, where the domain curve $C$ is allowed to be disconnected, and $\chi := 2(h^0(\cO_{C})-h^1(\cO_{C}))$.

\subsection{$2$-point Rubber invariants and $(t_1 + t_2)$-divisibility} \label{rubber-section}

We are also interested in the case when the target is nonrigid. The rubber invariants $\langle \mubar,\nubar \rangle^{\cX,\circ,\backsim }_{g,\gamma}$ can be similarly defined:
$$
\langle \mubar,\nubar\rangle^{\cX,\circ,\sim }_{g,\gamma}=
\frac{1}{|\Aut(\mu)||\Aut(\nu)|}\int_{\left[ \Mbar_{g, \gamma}(\cY, \mubar,\nubar)^\sim \right]^\vir}\frac{e_T(V)}{e_T(R^0\pi_*\tilde{F}^*(L\oplus L^{-1}))}.
$$
Here we abuse the notations $V$, $\pi$ and $\tilde{F}$ for their counterparts in the nonrigid case.

The virtual dimension $d$ of $\Mbar_{g, \gamma}(\cY, \mubar,\nubar)^\sim$ is
$$
\vdim = 2g-3+p+l(\mubar)+l(\nubar).
$$
On the other hand, the rank of the obstruction bundle $V=V_1\oplus V_2$ is
$$
\rk (V) = 2g-2+p+l''(\mubar)+l''(\nubar)+2\delta.
$$

Recall that in orbifold GW theory, the Hodge bundle is generalized to the so-called Hurwitz--Hodge bundle. For each character $\chi:\bZ_{n+1}\to\bC^*$, there is an associated Hurwitz--Hodge bundle $\bE_\chi$. Let $\lambda^\chi_i=c_i(\bE_\chi)$ be the $i$-th Chern class of the Hurwitz--Hodge bundle $\bE_\chi$, called the Hurwitz--Hodge class. In our case, the vector bundles $V_1$ and $V_2$ are dual to the Hurwitz--Hodge bundles $\bE_U$ and $\bE_{U^\vee}$, where $U$ and $U^\vee$ denote respectively the fundamental representation of $\bZ_{n+1}$ and its dual. Let $r_1$ and $r_2$ be the rank of $V_1$ and $V_2$ respectively. Then we have
\begin{eqnarray*}
	e_T(V_1) &=& t_1^{r_1}-t_1^{r_1-1}\lambda^U_1+\cdots+(-1)^{r_1}\lambda^U_{r_1} \\
	e_T(V_2) &=& t_2^{r_2}-t_2^{r_2-1}\lambda^{U^\vee}_1+
	\cdots+(-1)^{r_2}\lambda^{U^\vee}_{r_2}.
\end{eqnarray*}

There is an orbifold version of the Mumford relation:
$$
e_T(V_1)e_T(V_2) \bigg|_{t_1+t_2=0}=t_1^{r_1}t_2^{r_2}.
$$
In particular,
$$
\lambda^U_{r_1}\lambda^{U^\vee}_{r_2}=0, \qquad \text{if } r_1+r_2>0.
$$

\begin{lemma}
	The rubber invariants $\langle \mubar,\nubar\rangle^{\cX,\circ,\sim }_{g,\gamma}$ vanish unless either
	$$\rk(V) = \vdim = 0, \qquad \text{or} \qquad \rk(V) = \vdim +1 > 0.$$
	In particular, the latter case happens only if $\delta = 1$ or $\delta = 0$ and $l(\mubar) = l''(\mubar)$, $l(\nubar) = l''(\nubar)$.
\end{lemma}

\begin{proof}
	By dimensional reason, for the integral not to vanish, one must have $\rk(V) \geq \vdim$. On the other hand, direct comparison shows that $\rk(V) \leq \vdim +1$. It suffices to show that the invariants vanish when $\rk(V) = \vdim >0$, which follows from $\lambda^U_{r_1}\lambda^{U^\vee}_{r_2}=0$ and dimension counting in the integral. The latter case happens only if $l' (\mubar) + l'(\nubar) = 2 \delta$, which implies either $\delta = 1$ or $\delta = l'(\mubar) = l'(\nubar) = 0$.
\end{proof}

We now analyze the rubber invariants in different contexts.

\begin{enumerate}[a)]
	
\setlength{\parskip}{1ex}
	
\item $\delta = 1$, i.e.  all monodromies around loops on the domain curve are zero. In other words, $p=l''(\mubar)=l''(\nubar)=0$, and $\rk(V) = 2g$. The invariants simply reduce to the smooth case in \cite{Bry-Pan}.

\item $\delta = 0$ and $\rk(V) = \vdim = 0$. There are only a few possibilities in this case and one can compute the invariants directly by naive counting.

\begin{enumerate}[$\bullet$]

\item $g=p=0$, $l(\mubar)=2$, $l''(\mubar)=l''(\nubar)=l(\nubar)=1$.
$$\langle \mubar,\nubar\rangle^{\cX,\circ,\sim }_{0,\emptyset} = \frac{1}{(n+1)|\Aut(\mu)|}.$$

\item $g=p=0$, $l(\mubar) = l''(\mubar) =2$, $l(\nubar)=1$, $l''(\nubar)=0$.
$$\langle \mubar,\nubar\rangle^{\cX,\circ,\sim }_{0,\emptyset} = \frac{1}{(n+1)|\Aut(\mu)|}.$$

\item $g=0$, $p=1$, $l(\mubar)=1$, $l''(\mubar)=0$, $l(\nubar)=l''(\nubar)=1$.
$$\langle \mubar,\nubar\rangle^{\cX,\circ,\sim }_{0,\gamma} = \frac{1}{n+1}.$$

\end{enumerate}

\item $\delta = 0$ and $\rk (V) = \vdim +1 > 0$, in which case $l(\mubar) = l''(\mubar)$, $l(\nubar) = l''(\nubar)$, i.e. all monodromies are nontrivial. Mumford's relation $\lambda^U_{r_1-1}\lambda^{U^\vee}_{r_2}=\lambda^U_{r_1}\lambda^{U^\vee}_{r_2-1}$ implies
\begin{eqnarray*}
\langle \mubar,\nubar\rangle^{\cX,\circ,\sim }_{g,\gamma} &=&
	\frac{(-1)^{r_1+r_2-1}}{|\Aut(\mu)||\Aut(\nu)|}\int_{\left[ \Mbar_{g, \gamma}(\cY, \mubar,\nubar) \sslash \bC^* \right]^\vir} \left( t_1\lambda^U_{r_1-1}\lambda^{U^\vee}_{r_2}+t_2
	\lambda^U_{r_1}\lambda^{U^\vee}_{r_2-1} \right) \\
&=& \frac{(-1)^{r_1+r_2-1} (t_1 + t_2)}{|\Aut(\mu)||\Aut(\nu)|}\int_{\left[ \Mbar_{g, \gamma}(\cY, \mubar,\nubar) \sslash \bC^* \right]^\vir}
\lambda^U_{r_1}\lambda^{U^\vee}_{r_2-1}.
\end{eqnarray*}
In particular, it is divisible by $(t_1 + t_2)$. This is the main case we will treat in the following sections.

\end{enumerate}

The argument for the $(t_1+t_2)$-divisibility is valid in more general contexts. We summarize this feature in the following lemma, whose proof is exactly the same as above.

\begin{lemma} \label{divisibility}
If $\rk(V)>0$, then the invariants $\langle \mubar^1,\cdots,\mubar^r\rangle^{\cX,\circ}_{g,\gamma}$ and $\langle \mubar, \nubar \rangle^{\cX,\circ,\sim }_{g,\gamma}$ are divisible by $(t_1 + t_2)$.
\end{lemma}

\subsection{$3$-point functions from rubber}

According to the degeneration formula \cite{Abr-Fan}, $r$-point functions $\langle \mubar^1,\cdots,\mubar^r\rangle^{\cX,\circ}_{g,\gamma}$ can be determined by $3$-point functions. Hence we are particularly interested  in the case $r=3$. Moreover, under the generation conjecture (see Section \ref{sec-gen-conj}), it suffices to consider the following three special cases.

Let $\mubar,\nubar,\rhobar$ be three $\bZ_{n+1}$-weighted partitions of $m$. In the following three subsections, we will study the relative GW invariants
$$
\langle \mubar,\nubar,\rhobar\rangle^{\cX,\circ}_{g,\gamma}
$$
for
$$
\rhobar=\{(1,0),\cdots,(1,0)\}, \qquad \{(2,0),(1,0),\cdots,(1,0)\}, \qquad \textrm{or} \qquad \{(1,k),(1,0),\cdots,(1,0)\},
$$
where $k\neq0$. For simplicity, we abbreviate the notations as
$$(1,0)^m, \qquad (2,0)(1,0)^{m-2}, \qquad (1,k)(1,0)^{m-1}.$$

In this subsection, apart from several exceptional cases, the $3$-point functions above can be reduced to the rubber invariants of the previous section. For the exceptional cases, $3$-point functions can be easily computed.

\subsubsection{Case $\rhobar=(1,0)^m$} \label{1,0}

By computations in Section \ref{moduli}, the virtual dimension of the moduli space $\Mbar_{g, \gamma}(\cY, \mubar,\nubar,\rhobar)$ is
$$
\vdim (\Mbar_{g, \gamma}(\cY, \mubar, \nubar, \rhobar)) = 2g-2+p+l(\mubar)+l(\nubar).
$$
On the other hand, by the computation in Section \ref{obs bundle}, the rank of the obstruction bundle $V=V_1\oplus V_2$ is
$$
\rk(V) = 2g-2+p+l''(\mubar)+l''(\nubar)+2\delta.
$$

We consider the two possibilities $\delta = 1$ and $\delta = 0$.

\begin{enumerate}[a)]

\setlength{\parskip}{1ex}

\item $\delta=1$, i.e. all the monodromies around loops on the domain curve are trivial.

In this case $l''(\mubar)=l''(\nubar)=p=0$. We have
$$\rk(V) \leq \vdim (\Mbar_{g, \gamma}(\cY, \mubar, \nubar, \rhobar)).$$
By dimensional reason, for the invariants to be nontrivial, the equality needs to hold, which happens only if $l(\mubar)=l(\nubar)=1$, and hence $\vdim = \rk (V) = 2g$. Either by a $(t_1+ t_2)$-divisibility argument or by the smooth case \cite{Bry-Pan}, the only nontrivial case is $g=0$, and $\langle \mubar,\nubar,\rhobar\rangle^{\cX,\circ}_{0,\emptyset}=\frac{1}{m(n+1)t_1t_2}$.

\item $\delta=0$. The rank of $V$ is $\rk(V) = 2g-2+p+l''(\mubar)+l''(\nubar)$. We also have
$$\rk (V) \leq \vdim(\Mbar_{g, \gamma}(\cY, \mubar, \nubar, \rhobar)) = 2g-2+p+l(\mubar)+l(\nubar).$$
Again, dimension counting forces the equality to hold.

If $\vdim = \rk(V) >0$, then by Lemma \ref{divisibility}, the invariant is a polynomial in $t_1$, $t_2$ divisible by $(t_1+t_2)$, and hence has to be zero by dimensional constraints.

If $\vdim = \rk(V)=0$, we must have $p=0$, $g=0$, and $l(\mubar)=l(\nubar)=l''(\mubar)=l''(\nubar)=1$. The only nontrivial invariant is for $\rhobar = (1,0)^m, \mubar=(m,k), \nubar=(m,-k),k\neq 0$, which is $\frac{1}{m(n+1)}$.

In general, this case will contribute to disconnected invariants with $\rhobar = (1,0)^m,\mubar=-\nubar,l(\mubar)=l''(\mubar)$. The partition function is
$$
Z'_\GW([\bC^2/\bZ_{n+1}]\times \bP^1)_{\mubar, \nubar, \rhobar}= Z_{\mubar}^{-1},
$$
where $Z_{\mubar}=|\Aut(\mubar)|(n+1)^{l(\mubar)}\prod_{i=1}
^{l(\mubar)}\mu_i$.

\end{enumerate}

\subsubsection{Case $\rhobar= (2,0)(1,0)^{m-2}$}

In this case, we will reduce the relative GW invariants $\langle \mubar,\nubar,\rhobar\rangle^{\cX,\circ}_{g,\gamma}$ to the rubber invariants $\langle \mubar,\nubar\rangle^{\cX,\circ,\sim}_{g,\gamma}$. The key point is that one can replace a nonrigid invariant with a rigid invariant by imposing the condition that a marked point on the domain curve lies on a fixed fiber of $L\oplus L^{-1}\to \cY$. Consider the following descendent $2$-point relative invariants.
\begin{eqnarray*}
\langle \mubar|\tau_1[F]|\nubar\rangle^{\cX,\circ}_{g,\gamma} \ =\ \langle \mubar|\tau_1(\one)|\nubar\rangle^{\cX,\circ,\sim}_{g,\gamma}
\ =\ (2g-2+p+l(\mubar)+l(\nubar))\langle \mubar,\nubar\rangle^{\cX,\circ,\sim}_{g,\gamma},
\end{eqnarray*}
where $[F]$ is the fiber class and the second equality is the dilaton equation.

On the other hand, $\langle \mubar|\tau_1[F]|\nubar\rangle^{\cX,\circ}_{g,\gamma}$ can be computed by the degeneration formula. Let the base $\cY$ degenerate into two components, such that the two relative marked points lie on one component and the fiber insertion lies on the other. Degeneration formula implies
$$
\langle \mubar|\tau_1[F]|\nubar\rangle^{\cX,\circ}_{g,\gamma}
=\sum_{\etabar,\gamma'\sqcup\gamma''=\gamma,\Gamma_1,\Gamma_2}
\langle \mubar,\nubar,\etabar\rangle^{\cX,\bullet}_{\Gamma_1,\gamma'}Z_{\etabar}(t_1t_2)^{l'(\etabar)}
\langle -\etabar|\tau_1[F] \rangle^{\cX,\bullet}_{\Gamma_2,\gamma''},
$$
where $Z_{\etabar}=|\Aut(\etabar)|(n+1)^{l(\etabar)}\prod_{i=1}^{l(\etabar)}\eta_i$, and the summation is over all domain curve configurations $\Gamma_1$, $\Gamma_2$, such that the glued curve over $\Gamma_1$, $\Gamma_2$ is connected.

The second factor, which is a priori an integral over the moduli of relative stable maps with disconnected domains, can be written as that over a product of moduli spaces with connected domains. Each such moduli space is either of the form
$$\Mbar_{g_i, \gamma^i\sqcup(\one)}(\cY, -\etabar^i), \qquad \text{or} \qquad \Mbar_{g_i, \gamma^i}(\cY, -\etabar^i),  $$
depending on whether the insertion $\tau_1[F]$ is on the particular connected component or not.

The virtual dimensions of $\Mbar_{g_i, \gamma^i\sqcup(\one)}(\cY, -\etabar^i)$ and  $\Mbar_{g_i, \gamma^i}(\cY, -\etabar^i)$ are respectively
$$
2g_i-1+p_i+m_i+l(\etabar^i), \qquad 2g_i-2+p_i+m_i+l(\etabar^i),
$$
where $p_i, m_i, \etabar^i$ are the corresponding data associated to the component. The rank of the obstruction bundle $V$ over both moduli spaces is equal to
$$
\rk(V) = 2g_i-2+p_i+l''(\etabar^i)+2\delta.
$$

\begin{enumerate}[a)]
	
\setlength{\parskip}{1ex}

\item $\delta=0$. We have
$$\vdim (\Mbar_{g_i, \gamma^i\sqcup(\one)}(\cY, -\etabar^i)) \geq \rk (V) + 2, \qquad \vdim(\Mbar_{g_i, \gamma^i}(\cY, -\etabar^i)) > \rk (V).$$
The only nontrivial invariants come from the first type of components, and since $\deg \tau_1 [F] = 2$, the equality holds, i.e. $m_i=1$ and $l''(\etabar^i)=l(\etabar^i)=1$. Now $\gamma^i\neq\emptyset$ in order for the sum of monodromies at all marked points to vanish ($l''(\etabar^i)=l(\etabar^i)=1$ implies the only relative marked point has nontrivial monodromy).

However if $g_i>0$, together with $\gamma^i \neq \emptyset$ it would also imply $\rk(V) > 0$ and the invariant is divisible by $(t_1+t_2)$, which forces it to vanish by dimensional reasons. In short, the only invariant that survives is when $g_i=0$, and $p_i = l''(\etabar^i)=m_i = 1$; so $-\etabar^i = (1,-k)$, for some $k\neq 0$. The restriction of $\gamma''$ on this component is $(k)$, and hence $\gamma' = \gamma \backslash (k)$. The contribution of this component to the invariants $\langle -\etabar|\tau_1[F] \rangle^{\cX,\bullet}_{\Gamma_2,\gamma''}$ is $\frac{1}{n+1}$.

\item $\delta=1$. Counting dimensions, for $\Mbar_{g_i, \gamma^i\sqcup(\one)}(\cY, -\etabar^i)$, there are two possibilities. The first one is that $m_i = 2$, $p_i = 0$, $g_i=0$, $l(\etabar^i) =1$, and hence $\etabar^i = (2,0)$. The contribution of this component to the invariant $\langle - \etabar^i | \tau_1 [F] \rangle_{\Gamma_2, \gamma''}^{\cX, \circ}$ is $\frac{1}{2(n+1)t_1t_2}$. The second one is that $m_i = 1$, $p_i = 0$, $g_i>0$, $l(\etabar^i) =1$, and hence $\etabar^i = (1,0)$. For $\Mbar_{g_i, \gamma^i}(\cY, -\etabar^i)$, we must have $m_i = l(\etabar^i) = 1$, $p_i = 0$, $g_i=0$, and hence $\etabar^i = (1,0)$.

\end{enumerate}

Combining a) and b), for invariants $\langle -\etabar|\tau_1[F] \rangle^{\cX,\bullet}_{\Gamma_2,\gamma''}$, relative insertions that could appear in the gluing formula are $\etabar = (2,0)(1,0)^{m-2}$, $\etabar = (1,k)(1,0)^{m-1}$, $k\neq 0$, or $\etabar =(1,0)^{m}$ . In the first two cases, $\langle -\etabar|\tau_1[F] \rangle^{\cX,\bullet}_{\Gamma_2,\gamma''}$ are exactly canceled by the gluing factor $Z_\eta(t_1t_2)^{l'(\etabar)}$. In the third case, the factor $\langle \mubar,\nubar,\etabar\rangle^{\cX,\bullet}_{\Gamma_1,\gamma'}$ must be an integral over moduli space of genus 0 stable maps, $l(\mubar)=l(\nubar)=1, \mubar=-\nubar$, and $\gamma'=\emptyset$ by Section \ref{1,0} and this factor is equal to $\frac{1}{m(n+1)(t_1t_2)^\delta}$. So when $l(\mubar)=l(\nubar)=1,p=0$ we have
\begin{equation}
\langle \mubar|\tau_1[F]|\nubar\rangle^{\cX,\circ}_{g,\gamma}=\langle \mubar,\nubar,(2,0)(1,0)^{m-2} \rangle^{\cX,\circ}_{g,\gamma} +(t_1t_2)^{1-\delta}\langle (1,0)|\tau_1[F]|\rangle^{\cX,\circ}_{g,\emptyset}.
\end{equation}
When $l(\mubar)+l(\nubar)\geq 3$ or when $p>0$
\begin{equation}
\langle \mubar|\tau_1[F]|\nubar\rangle^{\cX,\circ}_{g,\gamma}=\langle \mubar,\nubar,(2,0)(1,0)^{m-2} \rangle^{\cX,\circ}_{g,\gamma} + \sum_{k=1}^n \langle \mubar,\nubar,(1,k)(1,0)^{m-1} \rangle^{\cX,\circ}_{g,\gamma \backslash (k) }.
\end{equation}

Recall the rigidification result obtained at the beginning of this subsection, and we conclude the following.
\begin{lemma}
\begin{enumerate}[(i)]

\item If $l(\mubar)=l(\nubar)=1$ and $p=0$, then
$$
\langle \mubar,\nubar,(2,0)(1,0)^{m-2} \rangle^{\cX,\circ}_{g,\gamma} +(t_1t_2)^{1-\delta}\langle (1,0)|\tau_1[F]|\rangle^{\cX,\circ}_{g,\emptyset} = 2g\langle \mubar,\nubar\rangle^{\cX,\circ,\sim}_{g,\gamma}.
$$

\item If $l(\mubar)+l(\nubar)\geq 3$ or $p>0$, then
$$
\langle \mubar,\nubar,(2,0)(1,0)^{m-2} \rangle^{\cX,\circ}_{g,\gamma} + \sum_{k=1}^n \langle \mubar,\nubar,(1,k)(1,0)^{m-1} \rangle^{\cX,\circ}_{g,\gamma \backslash (k)}
= (2g-2+p+l(\mubar)+l(\nubar))\langle \mubar,\nubar\rangle^{\cX,\circ,\sim}_{g,\gamma},
$$
\end{enumerate}
\end{lemma}

\begin{lemma} \label{cot}
	$$\sum_{g\geq 1}(-1)^g z^{2g-1}\langle (1,0)|\tau_1[F]|\rangle^{\cX,\circ}_{g,\emptyset}
=\frac{(t_1+t_2)}{2(n+1)t_1t_2}\left(\frac{\cosh \frac{z}{2}}{\sinh \frac{z}{2}}-\frac{2}{z}\right)$$
\end{lemma}

The proof of Lemma \ref{cot} will be given in Section \ref{pflemma}

\subsubsection{Case $\rhobar= (1,k)(1,0)^{m-1}$, $k\neq 0$}
In this case, similar argument still works to reduce $\langle \mubar,\nubar,\rhobar\rangle^{\cX,\circ}_{g,\gamma}$ to the rubber integral. First there is a similar rigidification argument
$$
\langle \mubar|\tau_0[\iota_*k]|\nubar\rangle^{\cX,\circ}_{g,\gamma} \ =\  \langle \mubar|\tau_0(k)|\nubar\rangle^{\cX,\circ,\sim}_{g,\gamma} \ =\ \langle \mubar,\nubar\rangle^{\cX,\circ,\sim}_{g,\gamma\sqcup(k)},
$$
where $\iota:F\to \cX$ is the inclusion of the fiber and we view $k$ as a twisted sector of $[\bC^2/\bZ_{n+1}]$.

On the other hand, we can still use the degeneration formula. Degenerate the base $\cY$ into two components, such that the two relative marked points lie on one component and the fiber insertion lies on the other. We have the following degeneration formula
$$
\langle \mubar|\tau_0[\iota_*k]|\nubar\rangle^{\cX,\circ}_{g,\gamma}
=\sum_{\etabar,\gamma'\sqcup\gamma''=\gamma,\Gamma_1,\Gamma_2}
\langle \mubar,\nubar,\etabar\rangle^{\cX,\bullet}_{\Gamma_1,\gamma'}Z_{\etabar}(t_1t_2)^{l'(\etabar)}
\langle -\etabar|\tau_0[\iota_*k] \rangle^{\cX,\bullet}_{\Gamma_2,\gamma''},
$$
where $Z_{\etabar}=|\Aut(\etabar)|(n+1)^{l(\etabar)}\prod_{i=1}^{l(\etabar)}\eta_i$, and we are summing over all domain curve configurations $\Gamma_1$, $\Gamma_2$ such that the glued curve over $\Gamma_1$, $\Gamma_2$ is connected.

As before, the second factor is an integral over a product of moduli spaces of relative stable maps with connected domains, each of the form
$$ \Mbar_{g_i, \gamma^i\sqcup(k)}(\cY, -\etabar^i) \qquad \text{or} \qquad \Mbar_{g_i, \gamma^i}(\cY, -\etabar^i), $$
depending on whether the insertion $\tau_0[\iota_*k]$ is on the particular connected component or not.

The virtual dimensions of $\Mbar_{g_i, \gamma^i\sqcup(k)}(\cY, -\etabar^i)$ and  $\Mbar_{g_i, \gamma^i}(\cY, -\etabar^i)$ are respectively
$$
2g_i-1+p_i+m_i+l(\etabar^i), \qquad 2g_i-2+p_i+m_i+l(\etabar^i),
$$
where $p_i$, $m_i$, $\etabar^i$ are the corresponding data associated to the component. The ranks of the obstruction bundle $V$ are respectively
$$
 2g_i-1+p_i+l''(\etabar^i)+2\delta, \qquad 2g_i-2+p_i+l''(\etabar^i)+2\delta.
$$

\begin{enumerate}[a)]
	
\setlength{\parskip}{1ex}

\item $\delta = 1$, which only happens for the second type of components, since the first type already has a nontrivial marking $k$. To get nontrivial invariants, one must have $\rk(V) \geq \vdim$, which implies $m_i = l(\etabar^i) = 1$. Hence $\etabar^i = (1,0)$.

This also forces that $\rk(V) = \vdim$, and hence $\rk(V) = \vdim = 0$ since otherwise the invariants vanish by $(t_1+t_2)$-divisibility. Hence $g_i = p_i = 0$. The invariant $\langle -\etabar \rangle^{\cX,\circ}_{\Gamma_2,\gamma''}$ contributed by the component is $\frac{1}{(n+1)t_1t_2}$.

\item $\delta = 0$. For the two types of components, we always have respectively
$$\vdim \geq \rk(V)+1, \qquad \vdim > \rk(V).$$
Only the first type contributes nontrivially, and the equality must hold. However, in this case we must also need $\rk(V) =0$, since otherwise the invariants will be divisible by $(t_1+ t_2)$ and therefore vanish by dimensional constraint. The only possibility is $g_i = p_i = 0$, $m_i = l(\etabar^i) = l''(\etabar^i) = 1$; so $\etabar^i = (1,k)$, $k\neq 0$. The invariant $\langle -\etabar|\tau_0[\iota_*k] \rangle^{\cX,\circ}_{\Gamma_2,\gamma''}$ contributed by this component is $\frac{1}{n+1}$.

\end{enumerate}

Combining a) and b), for invariants $\langle -\etabar|\tau_0[\iota_*k] \rangle^{\cX,\bullet}_{\Gamma_2,\gamma''}$, relative insertions that could appear are $\etabar = \rhobar = (1,k)(1,0)^{m-1}$, $k\neq 0$, and we compute
$$\langle \mubar|\tau_0[\iota_*k]|\nubar\rangle^{\cX,\circ}_{g,\gamma}=\langle \mubar,\nubar,(1,k)(1,0)^{m-1} \rangle^{\cX,\circ}_{g,\gamma}.$$
Recall the rigidification result obtained at the beginning of this subsection, and we conclude the following
\begin{lemma}
For $k\neq 0$,
$$
\langle \mubar,\nubar,(1,k)(1,0)^{m-1} \rangle^{\cX,\circ}_{g,\gamma} = \langle \mubar, \nubar \rangle^{\cX, \circ, \sim}_{g, \gamma\sqcup (k)}.
$$
\end{lemma}

\section{Relative Gromov--Witten theory of $\cA_n \times \bP^1$}
\subsection{Geometry of $\cA_n \times \bP^1$}
We define a $\bZ_{n+1}$ action on $\bC^2$ by
$$
\zeta\cdot(z_1,z_2):=(\zeta z_1,\zeta^{-1} z_2),
$$
where $\zeta$ is a primitive $(n+1)-$th root of unity viewed as a generator of $\bZ_{n+1}$.
Let $\cA_n \to \bC^2 / \bZ_{n+1}$ be the minimal resolution. Consider an action of the torus $T=(\bC^*)^2$ on $\bC^2$ via the standard diagonal action with torus weights $t_1,t_2$. The $T$ action on $\bC^2$ induces a $T$ action on $\cA_n$ with fixed points $p_1,\cdots,p_{n+1}$. The tangent weights at the fixed point $p_i$ are
$$w_i^- := (n+2-i)t_1-(i-1)t_2, \qquad w_i^+ := (i-n-1)t_1+it_2.$$

Let $\cF_{[\bC^2/\bZ_{n+1}]}$ denote the space of $\bZ_{n+1}$-weighted partitions, and $\cF_{\cA_n}$ denote the space of $H^*(\cA_n)$-weighted partitions.

A curve class $(\beta, m)\in H_2(\cA_n\times \bP^1, \bZ)$ is specified by the datum $(\beta, m):= m[\bP^1] + \beta$, where $m$ is the fixed integer as before, and $\beta\in H_2(\cA_n, \bZ)$. Let $E_1,\cdots,E_n$ be the $n$ exceptional curves in $\cA_n$. Then the intersection matrix of $E_1,\cdots,E_n$ is given by $E_i\cdot E_i=-2$, $E_i\cdot E_{i+1}=1$. We view $E_1,\cdots,E_n$ as a basis of $H^2(\cA_n, \bQ)$ and we let $\omega_1, \cdots, \omega_n \in H^2(\cA_n, \bQ)$ be the dual basis to $E_1,\cdots,E_n$ with respect to the above intersection pairing. For any $\beta\in H_2(\cA_n, \bQ)$ and $\omega\in H^2(\cA_n, \bQ)$, let $\beta\cdot\omega\in\bQ$ be the natural pairing.

\subsection{2-point rubber invariants} \label{sec:Anrubber}
Let $\vec\mu^1, \cdots, \vec\mu^r\in \cF_{\cA_n}$ with
$$
\vec\mu^i=\{(\mu^i_1,\gamma^i_1),\cdots,(\mu^i_{l(\vec\mu^i)},\gamma^i_{l(\vec\mu^i)})\}.
$$
Define the Gromov-Witten invariant
$$
\langle \vec\mu^1, \cdots, \vec\mu^r \rangle^{\cA_n\times \bP^1, \circ}_{g, (\beta, m)}:=\frac{1}{\prod_{i=1}^r|\Aut(\mu^i)|}
\int_{[\Mbar_g(\cA_n\times \bP^1,(\beta, m),\mu^1,\cdots,\mu^r)]^\vir}\prod_{i=1}^r\prod_{j=1}^{l(\vec\mu^i)}\ev^*_{ij}\gamma^i_j.
$$
We can also define the disconnect Gromov-Witten invariant $\langle \vec\mu^1, \cdots, \vec\mu^r \rangle^{\cA_n\times \bP^1, \bullet}_{\chi, (\beta, m)}$ and the rubber Gromov-Witten invariant $\langle \vec\mu, \vec\nu \rangle_{g, (\beta, m)}^{\cA_n\times \bP^1,\circ, \sim}$ in a similar matter as we did in Section \ref{relative-GW} and Section \ref{rubber-section}.

The generating functions for the relative GW theory of $\cA_n\times \bP^1$ and the rubber theory are defined in \cite{Mau} as
\begin{eqnarray*}
Z'_\GW(\cA_n\times \bP^1)_{\vec\mu, \vec\nu}^{\circ, \sim} &:=& \sum_{g\geq 0}\sum_{\beta\geq 0} z^{2g} s_1^{\beta\cdot\omega_1} \cdots s_n^{\beta\cdot\omega_n} \langle \vec\mu, \vec\nu \rangle_{g, (\beta, m)}^{\cA_n\times \bP^1,\circ, \sim}, \\
Z'_\GW(\cA_n\times \bP^1)^\circ_{\vec\mu^1, \cdots, \vec\mu^r} &:=& \sum_{g\geq 0}\sum_{\beta\geq 0} z^{2g-2} s_1^{\beta\cdot \omega_1} \cdots s_n^{\beta\cdot \omega_n} \langle \vec\mu^1, \cdots, \vec\mu^r \rangle^{\cA_n\times \bP^1, \circ}_{g, (\beta, m)},\\
Z'_\GW(\cA_n\times \bP^1)_{\vec\mu^1, \cdots, \vec\mu^r} &:=& \sum_{\chi}\sum_{\beta\geq 0} z^{-\chi} s_1^{\beta\cdot \omega_1} \cdots s_n^{\beta\cdot \omega_n} \langle \vec\mu^1, \cdots, \vec\mu^r \rangle^{\cA_n\times \bP^1, \bullet}_{\chi, (\beta, m)},
\end{eqnarray*}
where $\vec\mu^i\in \cF_{\cA_n}$. Let
\begin{eqnarray*}
Z'_{\GW,\beta=0}(\cA_n\times \bP^1)_{\vec\mu, \vec\nu}^{\circ, \sim}&:=&\sum_{g\geq 0}z^{2g}\langle \vec\mu, \vec\nu \rangle_{g, (0, m)}^{\cA_n\times \bP^1,\circ, \sim}\\
Z'_{\GW,\beta=0}(\cA_n\times \bP^1)^\circ_{\vec\mu^1, \cdots, \vec\mu^r} &:=& \sum_{g\geq 0}\sum_{\beta\geq 0} z^{2g-2} \langle \vec\mu^1, \cdots, \vec\mu^r \rangle^{\cA_n\times \bP^1, \circ}_{g, (0, m)},\\
Z'_{\GW,\beta=0}(\cA_n\times \bP^1)_{\vec\mu^1, \cdots, \vec\mu^r} &:=& \sum_{\chi} z^{-\chi} \langle \vec\mu^1, \cdots, \vec\mu^r \rangle^{\cA_n\times \bP^1, \bullet}_{\chi, (0, m)}
\end{eqnarray*}
denote the $\beta = 0$ parts of $Z'_\GW(\cA_n\times \bP^1)_{\vec\mu, \vec\nu}^{\circ, \sim}$, $Z'_\GW(\cA_n\times \bP^1)^\circ_{\vec\mu^1, \cdots, \vec\mu^r}$ and $Z'_\GW(\cA_n\times \bP^1)_{\vec\mu^1, \cdots, \vec\mu^r}$ respectively. Similarly, let $Z'_{\GW,\beta\neq0}(\cA_n\times \bP^1)_{\vec\mu, \vec\nu}^{\circ, \sim}$, $Z'_{\GW,\beta\neq0}(\cA_n\times \bP^1)^\circ_{\vec\mu^1, \cdots, \vec\mu^r}$ and $Z'_{\GW,\beta\neq0}(\cA_n\times \bP^1)_{\vec\mu^1, \cdots, \vec\mu^r}$ denote the nonzero degree parts of $Z'_\GW(\cA_n\times \bP^1)_{\vec\mu, \vec\nu}^{\circ, \sim}$, $Z'_\GW(\cA_n\times \bP^1)^\circ_{\vec\mu^1, \cdots, \vec\mu^r}$ and $Z'_\GW(\cA_n\times \bP^1)_{\vec\mu^1, \cdots, \vec\mu^r}$ respectively.

Recall the explicit isomorphism between the two Fock spaces. If we identify $\bZ_{n+1}$ with the orbifold cohomology $H^*_\orb([\bC^2/\bZ_{n+1}])$, the isomorphism is given by
$$\Phi: H^*_\orb([\bC^2/\bZ_{n+1}]) \cong H^*(\cA_n),$$
\begin{equation*}
e_0 \mapsto 1, \qquad e_i \mapsto \frac{\zeta^{i/2} - \zeta^{-i/2}}{n+1} \sum_{j=1}^n \zeta^{ij} \omega_j, \qquad 1\leq i\leq n,
\end{equation*}

We have the following results on the $\beta = 0$ part of the generating function.

\begin{proposition} \label{Anrubber}
Let $1\leq j, j' \leq n+1$.
$$
Z'_{\GW,\beta=0}(\cA_n\times \bP^1)_{\big(m,\Phi(e_j)\big), \big(m,\Phi(e_{j'})\big)}^{\circ, \sim} =
\delta_{j+j', n+1} \cdot \frac{t_1+t_2}{2(n+1)} \log \left[ \frac{(e^{mz}-1)(e^{-mz}-1)}{-z^2}  \right] .
$$
\end{proposition}

\begin{proof}
By $T$-localization, we have
\begin{equation}
\langle \big(m,\Phi(e_j)\big), \big(m,\Phi(e_{j'})\big)\rangle_{g, (0, m)}^{\cA_n\times \bP^1,\circ, \sim}=\sum_{i=1}^{n+1} \left. \langle (m), (m)\rangle_{g, (0, m)}^{\bC^2\times \bP^1,\circ, \sim} \right|_{(t_1, t_2) = (w_i^+, w_i^-)} \Phi(e_j)|_{p_i}\Phi(e_{j'})|_{p_i}.
\end{equation}
By Proposition \ref{rig}, Proposition \ref{cotAn}, and Lemma \ref{add},
\begin{equation}
\left. Z'_{\GW,\beta=0}(\bC^2\times \bP^1)_{(m), (m)}^{\circ, \sim} \right|_{(t_1, t_2) = (w_i^+, w_i^-)} =\frac{w^+_i+w^-_i}{2w^+_iw^-_i}\big(\log (e^{mz}-1)+\log(e^{-mz}-1)-\log z-\log(-z)\big).
\end{equation}
Therefore, in order to compute $Z'_{\GW,\beta=0}(\cA_n\times \bP^1)_{\big(m,\Phi(e_j)\big), \big(m,\Phi(e_{j'})\big)}^{\circ, \sim}$, it suffices to compute
$$
\sum_{i=1}^{n+1}\frac{w^+_i+w^-_i}{w^+_iw^-_i}\Phi(e_j)|_{p_i}\Phi(e_{j'})|_{p_i} = (t_1 + t_2) \sum_{i=1}^{n+1}\frac{\Phi(e_j)|_{p_i}\Phi(e_{j'})|_{p_i}}{w^+_iw^-_i} = (t_1+ t_2) \langle \Phi (e_j) , \Phi (e_{j'}) \rangle,
$$
which is $(t_1 + t_2) \delta_{j+j', n+1}$ since $\Phi$ preserves the Poincar\'e pairing.
\end{proof}

\subsection{3-point functions from rubber}
In this subsection, we will be interested in the 3-point function $Z'_\GW(\cA_n\times \bP^1)_{\vec\mu, \vec\nu, \vec\rho}$ and their relationship with the 2-point rubber invariants, where $\vec\rho=(1,1)^m, (1,1)^{m-2}(2,1),$ or $(1,1)^{m-1}(1,\omega)$. Here $\omega\in H^2(\cA_n)$ is a divisor class. Equivalently, let $\rhobar = (1,0)^m$, $(2,0)(1,0)^{m-2}$, or $(1,k)(1,0)^{m-1}$, $k\neq 0$. The correspondents of $\rhobar$ will be $(1,1)^m, (1,1)^{m-2}(2,1)$, or $(1,1)^{m-1}(1,\Phi(e_k))$ respectively.

\subsubsection{Case $\vec\rho = (1,1)^m$} \label{1,1}

In this case we have the following results.

\begin{proposition} \label{3-pt-10}
\begin{enumerate}[1)]
		\item For $\vec\mu=\vec\nu=(m,1)$, we have
		$$
		\langle\vec{\mu},\vec{\nu}, (1,1)^m \rangle^{\cA_n\times \bP^1,\circ}_{0,(0,m)}=\frac{1}{m(n+1) t_1 t_2} ,
		$$
		and $\langle\vec{\mu},\vec{\nu}, (1,1)^m \rangle^{\cA_n\times \bP^1,\circ}_{g,(\beta,m)}=0$ when $g>0$ or  $\beta\neq 0$.
		
		\item For $\vec\mu=(m,\Phi(e_j)), \vec\nu=(m,\Phi(e_{n+1-j}))$, $1\leq j\leq n$, we have
		$$
		\langle\vec{\mu},\vec{\nu},  (1,1)^m \rangle^{\cA_n\times \bP^1,\circ}_{0,(0,m)}=\frac{1}{m(n+1) },
		$$
		and $\langle\vec{\mu},\vec{\nu}, (1,1)^m \rangle^{\cA_n\times \bP^1,\circ}_{g,(\beta,m)}=0$ when $g>0$ or  $\beta\neq 0$.
		
		\item For $l(\vec\mu)+l(\vec\nu)\geq 3$, we have $\langle\vec{\mu},\vec{\nu}, (1,1)^m \rangle^{\cA_n\times \bP^1,\circ}_{g,(\beta,m)}=0$.
	\end{enumerate}
\end{proposition}

\begin{proof}
	First we should notice that by Lemma 4.2 of \cite{Mau} and by dimensional constraints, the only nontrivial invariant for such $\vec\mu$, $\vec\nu$, $\vec\rho$ is indeed when $\beta=0$, $g=0$, $l(\vec\mu)=l(\vec\nu)=1$. Hence 3) follows.
	
	When $\vec\mu=\vec\nu=(m,1)$, $\langle\vec{\mu},\vec{\nu}, (1,1)^m \rangle^{\cA_n\times \bP^1,\circ}_{0,(0,m)}$ is equal to
	$$
	\frac{1}{m}\sum_{i=1}^{n+1}
	\frac{1}{w_i^- w_i^+} = \frac{1}{m(t_1 + t_2)} \sum_{i=1}^{n+1}
	\left( \frac{1}{w_i^-} + \frac{1}{w_i^+} \right) = \frac{1}{m(t_1 + t_2)}
	\left( \frac{1}{w_1^-} + \frac{1}{w_{n+1}^+} \right)=\frac{1}{m(n+1) t_1 t_2}.$$
	Here we have used the fact that $w_i^- + w_i^+ = t_1 + t_2$ and $w_i^+ = - w_{i+1}^-$ for any $i$. This proves 1).
	
	When $\vec\mu=(m,\Phi(e_j))$, $\vec\nu=(m,\Phi(e_{n+1-j}))$, $1\leq j\leq n$, we have
	$$
	\langle\vec{\mu},\vec{\nu}, (1,1)^m \rangle^{\cA_n\times \bP^1,\circ}_{0,(0,m)} = \frac{1}{m}\sum_{i=1}^{n+1}
	\frac{1}{w_i^- w_i^+}\Phi(e_j)|_{p_i}\Phi(e_{n+1-j})|_{p_i}.
	$$
	which is $\frac{1}{m(n+1)}$ by the computation in the proof of Proposition \ref{3-pt-10}. This proves (2).
\end{proof}

\subsubsection{Case $\vec\rho = (1,1)^{m-2} (2,1)$ or $(1,1)^{m-1} (1, \omega_k)$}

For the $\beta\neq 0$ part, the relation between $Z'_{\GW,\beta\neq0}(\cA_n\times \bP^1)_{\vec\mu, \vec\nu}^{\circ, \sim}$ and $Z'_{\GW,\beta\neq0}(\cA_n\times \bP^1)_{\vec\mu, \vec\nu, \vec\rho}$ is studied in \cite{Mau}:

\begin{proposition}[Proposition 4.3 of \cite{Mau}]

\begin{enumerate}[1)]

\item For $\vec\rho=(1,1)^{m-2}(2,1)$, we have
$$
z^{l(\vec\mu)+l(\vec\nu)-1}Z'_{\GW,\beta\neq0}(\cA_n\times \bP^1)^\circ_{\vec\mu, \vec\nu, \vec\rho} =
\frac{\partial}{\partial z} \left( z^{l(\vec\mu)+l(\vec\nu)-2}Z'_{\GW,\beta\neq0}(\cA_n\times \bP^1)_{\vec\mu, \vec\nu}^{\circ, \sim}   \right).
$$

\item For $\vec\rho=(1,1)^{m-1}(1,\omega_k)$, we have
$$
z^{l(\vec\mu)+l(\vec\nu)}Z'_{\GW,\beta\neq0}(\cA_n\times \bP^1)^\circ_{\vec\mu, \vec\nu, \vec\rho}=
s_k\frac{\partial}{\partial s_k} \left( z^{l(\vec\mu)+l(\vec\nu)-2}Z'_{\GW,\beta\neq0}(\cA_n\times \bP^1)_{\vec\mu, \vec\nu}^{\circ, \sim}   \right).
$$
\end{enumerate}

\end{proposition}

For our purpose, we also need to study the $\beta = 0$ part. We have the following result:

\begin{proposition}\label{rig}
Let $\vec\rho=(1,1)^{m-2}(2,1)$.
\begin{enumerate}
\item For $\vec\mu=(m,\Phi(e_j))$, $\vec\nu=(m,\Phi(e_{n+1-j}))$, $1\leq j\leq n$, we have
\begin{eqnarray*}
&&z^{l(\vec\mu)+l(\vec\nu)-1}Z'_{\GW,\beta=0}(\cA_n\times \bP^1)^\circ_{\vec\mu, \vec\nu, \vec\rho}\\
&=&
\frac{\partial}{\partial z}\big(z^{l(\vec\mu)+l(\vec\nu)-2}Z'_{\GW,\beta=0}(\cA_n\times \bP^1)_{\vec\mu, \vec\nu}^{\circ, \sim}   \big)-\sum_{g\geq 1}(-1)^gz^{2g-1}t_1t_2\langle (1,1)|\tau_1[F]|\rangle^{\cA_n\times \bP^1,\circ}_{g,(0,1)}.
\end{eqnarray*}
\item For $l(\vec\mu)+l(\vec\nu)\geq 3$,
$$z^{l(\vec\mu)+l(\vec\nu)-1}Z'_{\GW,\beta=0}(\cA_n\times \bP^1)^\circ_{\vec\mu, \vec\nu, \vec\rho}=
\frac{\partial}{\partial z}\big(z^{l(\vec\mu)+l(\vec\nu)-2}Z'_{\GW,\beta=0}(\cA_n\times \bP^1)_{\vec\mu, \vec\nu}^{\circ, \sim}   \big).$$
\end{enumerate}
\end{proposition}

\begin{proof}
Since $\beta=0$, by virtual localization, we have
\begin{equation}\label{loc1}
\langle\vec{\mu},\vec{\nu}, \vec{\rho}\rangle^{\cA_n\times \bP^1,\circ}_{g,(0,m)}=\sum_{i=1}^{n+1} \left. \langle \mu, \nu,\rho\rangle_{g, m}^{\bC^2\times \bP^1,\circ} \right|_{(t_1, t_2) = (w_i^+, w_i^-)} \cdot  \prod_{j=1}^{l(\vec{\mu})}\gamma_j|_{p_i}\prod_{k=1}^{l(\vec{\nu})}\gamma'_k|_{p_i},
\end{equation}
where $\vec{\mu}=\{(\mu_1,\gamma_1),\cdots, (\mu_{l(\vec{\mu})},\gamma_{l(\vec{\mu})})\}$, $\vec{\nu}=\{(\nu_1,\gamma'_1),\cdots, (\nu_{l(\vec{\nu})},\gamma'_{l(\vec{\nu})})\}\in\cF_{\cA_n}$ and $\mu,\nu,\rho$ are the underlying unweighted partitions of $\vec{\mu},\vec{\nu},\vec\rho$. Similarly, we also have
\begin{equation}\label{loc2}
\langle\vec{\mu},\vec{\nu}\rangle^{\cA_n\times \bP^1,\circ,\sim}_{g,(0,m)}=\sum_{i=1}^{n+1} \left. \langle \mu, \nu\rangle_{g, m}^{\bC^2\times \bP^1,\circ, \sim} \right|_{(t_1, t_2) = (w_i^+, w_i^-)}  \cdot \prod_{j=1}^{l(\vec{\mu})}\gamma_j|_{p_i}\prod_{k=1}^{l(\vec{\nu})}\gamma'_k|_{p_i},
\end{equation}
and
\begin{equation}\label{loc3}
\langle (1,1)|\tau_1[F]|\rangle^{\cA_n\times \bP^1,\circ}_{g,(0,1)}=\sum_{i=1}^{n+1}\langle \left.  (1)|\tau_1[F]|\rangle^{\bC^2\times \bP^1,\circ}_{g,1} \right|_{(t_1, t_2) = (w_i^+, w_i^-)}  .
\end{equation}
So we only need to compare $\langle \mu, \nu,\rho\rangle_{g, m}^{\bC^2\times \bP^1,\circ}$ and $\langle \mu, \nu\rangle_{g, m}^{\bC^2\times \bP^1,\circ, \sim}$. By rigidification, we have
$$
\langle\mu|\tau_1[F]|\nu\rangle^{\bC^2\times \bP^1,\circ}_{g,m}=\langle\mu|\tau_1[F]|\nu\rangle^{\bC^2\times \bP^1,\circ,\sim}_{g,m}=
(2g-2+l(\mu)+l(\nu))\langle \mu, \nu\rangle_{g, m}^{\bC^2\times \bP^1,\circ, \sim}.
$$
On the other hand, $\langle\mu|\tau_1[F]|\nu\rangle^{\bC^2\times \bP^1,\circ}_{g,m}$ can be computed by the degeneration formula. Let the base $\bP^1$ degenerate into two components, such that the two relative marked points lie on one component and the fiber insertion lies on the other. Degeneration formula implies
$$
\langle\mu|\tau_1[F]|\nu\rangle^{\bC^2\times \bP^1,\circ}_{g,m}
=\sum_{\eta,\Gamma_1,\Gamma_2}
\langle \mu,\nu,\eta\rangle^{\bC^2\times \bP^1,\bullet}_{\Gamma_1}z_{\eta}(t_1t_2)^{l(\eta)}
\langle \eta|\tau_1[F] \rangle^{\bC^2\times \bP^1,\bullet}_{\Gamma_2},
$$
where $z_{\eta}=|\Aut(\eta)|\prod_{i=1}^{l(\eta)}\eta_i$, and the summation is over all domain curve configurations $\Gamma_1$, $\Gamma_2$, such that the glued curve over $\Gamma_1$, $\Gamma_2$ is connected.

The second factor, which is a priori an integral over the moduli of relative stable maps with disconnected domains, can be written as that over a product of moduli spaces with connected domains. Each such moduli space is either of the form
$$\Mbar_{g_i, 1}(\bP^1, \eta^i), \qquad \text{or} \qquad \Mbar_{g_i, 0}(\bP^1, \eta^i),  $$
depending on whether the insertion $\tau_1[F]$ is on the particular connected component or not.

The virtual dimensions of $\Mbar_{g_i, 1}(\bP^1, \eta^i)$ and  $\Mbar_{g_i, 0}(\bP^1, \eta^i)$ are respectively
$$
2g_i-1+m_i+l(\eta^i), \qquad 2g_i-2+m_i+l(\eta^i),
$$
where $m_i, \eta^i$ are the corresponding data associated to the component. The rank of the obstruction bundle $V$ over both moduli spaces is given by
$$
\rk(V) = 2g_i.
$$
Counting dimensions, for $\Mbar_{g_i, 1}(\bP^1, \eta^i)$, there are two possibilities. The first one is that $m_i = 2$, $g_i=0$, $l(\eta^i) =1$, and hence $\eta^i = (2)$. The contribution of this component to the invariant $\langle \eta^i | \tau_1 [F] \rangle_{\Gamma_2}^{\bC^2\times \bP^1, \circ}$ is $\frac{1}{2t_1t_2}$. The second one is that $m_i = 1$, $g_i>0$, $l(\eta^i) =1$, and hence $\eta^i = (1)$. For $\Mbar_{g_i, 0}(\bP^1, \eta^i)$, we must have $m_i = l(\eta^i) = 1$, $g_i=0$, and hence $\eta^i = (1)$.

Therefore, for invariants $\langle \eta|\tau_1[F] \rangle^{\bC^2\times \bP^1,\bullet}_{\Gamma_2}$, relative insertions that could appear in the gluing formula are $\eta = (2,0)(1,0)^{m-2}$, or $\eta =(1,0)^{m}$ . In the first case, $\langle \eta|\tau_1[F] \rangle^{\bC^2\times \bP^1,\bullet}_{\Gamma_2}$ is exactly canceled by the gluing factor $z_{\eta}(t_1t_2)^{l(\eta)}$. In the second case, the factor $\langle \mu,\nu,\eta\rangle^{\bC^2\times \bP^1,\bullet}_{\Gamma_1}$ must be an integral over moduli space of genus 0 stable maps and $l(\mu)=l(\nu)=1$ by Proposition \ref{3-pt-10} and this factor is equal to $\frac{1}{m}$. So when $l(\mu)=l(\nu)=1$ we have
\begin{equation}
2g\langle \mu, \nu\rangle_{g, m}^{\bC^2\times \bP^1,\circ, \sim}=\langle\mu|\tau_1[F]|\nu\rangle^{\bC^2\times \bP^1,\circ}_{g,m}=\langle \mu,\nu,(2)(1)^{m-2} \rangle^{\bC^2\times \bP^1,\circ}_{g} +\langle (1)|\tau_1[F]|\rangle^{\bC^2\times \bP^1,\circ}_{g}.
\end{equation}
When $l(\mu)+l(\nu)\geq 3$
\begin{equation}
(2g-2+l(\mu)+l(\nu))\langle \mu, \nu\rangle_{g, m}^{\bC^2\times \bP^1,\circ, \sim}=\langle\mu|\tau_1[F]|\nu\rangle^{\bC^2\times \bP^1,\circ}_{g,m}=\langle \mu,\nu,(2)(1)^{m-2} \rangle^{\bC^2\times \bP^1,\circ}_{g}.
\end{equation}
By \eqref{loc1}, \eqref{loc2}, \eqref{loc3} and by taking the generating functions, we obtain Proposition \ref{rig}.
\end{proof}

Actually, the $\beta = 0$ part can be reduced to invariants on $\bC^2 \times \bP^1$ and computed directly from results in \cite{Bry-Pan}.

\begin{proposition} \label{cotAn}
For $\vec\mu=(m,\Phi(e_j))$, $\vec\nu=(m,\Phi(e_{n+1-j}))$, $1\leq j\leq n$, $\vec\rho=(1,1)^{m-2}(2,1)$, we have
$$
Z'_{\GW,\beta=0}(\cA_n\times \bP^1)^\circ_{\vec\mu, \vec\nu, \vec\rho}=\frac{t_1+t_2}{2(n+1)} \left( m\frac{\cosh \frac{mz}{2}}{\sinh \frac{mz}{2}}-\frac{\cosh \frac{z}{2}}{\sinh \frac{z}{2}} \right)
$$
\end{proposition}

\begin{proof}
Let $\mu,\nu,\rho$ be the underlying unweighted partitions of $\vec{\mu},\vec{\nu},\vec\rho$. By Theorem 6.5 of \cite{Bry-Pan}, we have
$$
Z'_{\GW}(\bC^2\times \bP^1)^\circ_{\mu, \nu, \rho}=\frac{t_1+t_2}{2t_1t_2} \left(m\frac{\cosh \frac{mz}{2}}{\sinh \frac{mz}{2}}-\frac{\cosh \frac{z}{2}}{\sinh \frac{z}{2}} \right).
$$
By \eqref{loc1},
$$
Z'_{\GW,\beta=0}(\cA_n\times \bP^1)^\circ_{\vec\mu, \vec\nu, \vec\rho}=
\left( m\frac{\cosh \frac{mz}{2}}{\sinh \frac{mz}{2}}-\frac{\cosh \frac{z}{2}}{\sinh \frac{z}{2}} \right) \cdot \sum_{i=1}^{n+1}
\frac{w_i^-+ w_i^+}{2w_i^- w_i^+}\Phi(e_j)|_{p_i}\Phi(e_{j'})|_{p_i}.
$$
By the computation in the proof of Proposition \ref{3-pt-10}, we have $\sum_{i=1}^{n+1}
\frac{w_i^-+ w_i^+}{w_i^- w_i^+}\Phi(e_j)|_{p_i}\Phi(e_{j'})|_{p_i}=\frac{t_1+t_2}{n+1}$. Therefore
$$
Z'_{\GW,\beta=0}(\cA_n\times \bP^1)^\circ_{\vec\mu, \vec\nu, \vec\rho}=\frac{t_1+t_2}{2(n+1)} \left(m\frac{\cosh \frac{mz}{2}}{\sinh \frac{mz}{2}}-\frac{\cosh \frac{z}{2}}{\sinh \frac{z}{2}} \right).
$$
\end{proof}

\begin{lemma}\label{add}
$$
\sum_{g\geq 1}(-1)^gz^{2g-1}\langle (1,1)|\tau_1[F]|\rangle^{\cA_n\times \bP^1,\circ}_{g,(0,1)}=\frac{(t_1+t_2)}{2(n+1)t_1t_2}\left(\frac{\cosh \frac{z}{2}}{\sinh \frac{z}{2}}-\frac{2}{z}\right).
$$
\end{lemma}
\begin{proof}
By Lemma \ref{cot}, we have
$$
\sum_{g\geq 1}(-1)^gz^{2g-1}\langle (1)|\tau_1[F]|\rangle^{\bC^2\times \bP^1,\circ}_{g,1}=\frac{(t_1+t_2)}{2t_1t_2}\left(\frac{\cosh \frac{z}{2}}{\sinh \frac{z}{2}}-\frac{2}{z}\right).
$$
By \eqref{loc3},
$$
\sum_{g\geq 1}(-1)^gz^{2g}\langle (1,1)|\tau_1[F]|\rangle^{\cA_n\times \bP^1,\circ}_{g,(0,1)}=\left(\frac{\cosh \frac{z}{2}}{\sinh \frac{z}{2}}-\frac{2}{z}\right)\sum_{i=1}^{n+1}
\frac{w_i^-+ w_i^+}{2w_i^- w_i^+}.
$$
By the computation in the proof of Proposition \ref{3-pt-10}, we know that $\sum_{i=1}^{n+1}
\frac{w_i^-+ w_i^+}{w_i^- w_i^+}=\frac{(t_1+t_2)}{(n+1)t_1t_2}$. Therefore
$$
\sum_{g\geq 1}(-1)^gz^{2g-1}\langle (1,1)|\tau_1[F]|\rangle^{\cA_n\times \bP^1,\circ}_{g,(0,1)}=\frac{(t_1+t_2)}{2(n+1)t_1t_2}\left(\frac{\cosh \frac{z}{2}}{\sinh \frac{z}{2}}-\frac{2}{z}\right).
$$

\end{proof}

\section{Rubber invariants and Crepant Resolution Conjecture}

Recall that in Section \ref{relative}, we have reduced the relative GW invariants $\langle \mubar,\nubar,\rhobar\rangle^{\cX,\circ}_{g,\gamma}$ for
$$
\rhobar=(1,0)^m, \qquad (2,0)(1,0)^{m-2}, \qquad \textrm{or} \qquad (1,k)(1,0)^{m-1},
$$
to rubber invariants $\langle \mubar,\nubar\rangle^{\cX,\circ,\sim}_{g,\gamma}.$ In this section, using the orbifold Grothendieck--Riemann--Roch calculation in \cite{Tse1}, we will compute the rubber invariants $\langle \mubar,\nubar\rangle^{\cX,\circ,\sim}_{g,\gamma}$, under the main assumption
$$(\dagger): \qquad \delta = 0, \qquad \rk(V) = \vdim +1 >0;$$
in particular, $l(\mubar) = l''(\mubar)$, $l(\nubar) = l''(\nubar)$.

\subsection{Double ramification cycle}

The double ramification cycle $\DR_{g,N}$ in $\Mbar_{g,N}$ is defined as the pushforward of the virtual class under the forgetful map
$$
\Mbar_{g,N-l(\mu)-l(\nu)} (\bP^1, \mu, \nu)^\sim \to \Mbar_{g,N},
$$
where $\mu$, $\nu$ are ordinary partitions of $m$. By definition \ref{vc}, given a sufficiently large twisting choice $\fr$, $\DR_{g,N}$ coincides with the pushforward of the virtual class under the forgetful map
$$
\Mbar^\fr_{g,N-l(\mu)-l(\nu)} (\bP^1, \mu, \nu)^\sim \to \Mbar_{g,N}.
$$
In this subsection we would like to relate our relative invariants to $\DR_{g,N}$.

There is a simple relationship between the obstruction theories. Let $\fr$ be a sufficiently large twisting choice. There is a map
$$
\pi^\fr: \overline\cM_{g, \gamma}^\fr (B \bZ_{n+1} \times \bP^1, \mubar, \nubar)^\sim \to \overline\cM_{g,p}^\fr (\bP^1, \mu, \nu)^\sim,
$$
defined by sending a relative stable map $\cC\to B\bZ_{n+1} \times \bP^1[k] (\fr)$ to the relative coarse moduli space of the composition $\cC\to \bP^1[k](\fr)$.

\begin{lemma} \label{pullback}
$$
\left[ \Mbar^\fr_{g, \gamma}(B\bZ_{n+1}\times \bP^1, \mubar, \nubar)^\sim \right]^\vir = (\pi^\fr)^* \left[ \Mbar^\fr_{g,p} (\bP^1, \mu, \nu)^\sim \right]^\vir.
$$
\end{lemma}

\begin{proof}
This is proved in Lemma \ref{pullback-app}.
\end{proof}

Next we consider the following commutative diagram.
\begin{equation} \label{diag}
\xymatrix{
	\overline\cM^\fr_{g, \gamma} (B \bZ_{n+1} \times \bP^1, \mubar, \nubar)^\sim \ar[r]^-{\rhobar} \ar[d]_{\pi^\fr} & \overline\cM_{g, \mubar \sqcup \nubar \sqcup \gamma} (B \bZ_{n+1}) \ar[d]^{\pi} \\
	\overline\cM_{g,p}^\fr (\bP^1, \mu, \nu)^\sim \ar[r]^-\rho  &  \overline\cM_{g, l(\mu) + l(\nu) + p}
}
\end{equation}
All maps here are proper. Moreover, $\pi_2$ here, like $\pi_1$, is also flat, proper, quasi-finite, and of degree $(n+1)^{2g-1}$. The flatness of $\pi_2$ follows from Corollary 3.0.5 of \cite{ACV}, or alternatively, can be shown in the same manner as in Lemma \ref{Cart-1}.

\begin{lemma}
	$$
	\rhobar_* \left[ \Mbar^\fr_{g, \gamma}(B\bZ_{n+1}\times \bP^1, \mubar, \nubar)^\sim \right]^\vir = \pi^* \DR_{g,l(\mu)+l(\nu)+p}.
	$$
\end{lemma}

\begin{proof}
Consider the following diagram,
$$
\xymatrix{
	\overline\cM^\fr_{g, \gamma} (B \bZ_{n+1} \times \bP^1, \mubar, \nubar)^\sim \ar[r]^-{F} \ar[dr]_{\pi^\fr} & \cN \ar[r]^-G \ar[d]^-{H}  & \overline\cM_{g, \mubar \sqcup \nubar \sqcup \gamma} (B \bZ_{n+1}) \ar[d]^{\pi} \\
	& \overline\cM^\fr_{g,p} (\bP^1, \mu, \nu)^\sim \ar[r]^-\rho  &  \overline\cM_{g, l(\mu) + l(\nu) + p} ,
}
$$
where $\cN := \overline\cM^\fr_{g,p} (\bP^1, \mu, \nu)^\sim \times_{ \overline\cM_{g, l(\mu) + l(\nu) + p} }  \overline\cM_{g, \mubar \sqcup \nubar \sqcup \gamma} (B \bZ_{n+1}) $ is the fiber product. By Lemma \ref{pullback} we have
$$
\rhobar_* \left[ \Mbar^\fr_{g, \gamma}(B\bZ_{n+1}\times \bP^1, \mubar, \nubar)^\sim \right]^\vir =  G_* F_* (\pi^\fr)^* \left[ \Mbar^\fr_{g,p} (\bP^1, \mu, \nu)^\sim \right]^\vir
= G_* F_* F^* H^* \left[ \Mbar^\fr_{g,p} (\bP^1, \mu, \nu)^\sim \right]^\vir.
$$
Note that the map $F$ here is flat, proper, quasi-finite, and of degree $1$. We have $F_* F^* = 1$, and therefore obtain from the above
$$
G_* H^* \left[ \Mbar^\fr_{g,p} (\bP^1, \mu, \nu)^\sim \right]^\vir = \pi^* \rho_* \left[ \Mbar^\fr_{g,p} (\bP^1, \mu, \nu)^\sim \right]^\vir = \pi^* \DR_{g,l(\mu)+l(\nu)+p}.
$$
\end{proof}

\subsection{Pixton's formula}

Now the rubber invariant $\langle \mubar,\nubar\rangle^{\cX,\circ,\sim }_{g,\gamma}$ is equal to
$$
\frac{(-1)^{r_1+r_2-1}(t_1+t_2)}{|\Aut(\mu)||\Aut(\nu)|}
\int_{\Mbar_{g, \mubar\sqcup \nubar\sqcup \gamma}(B\bZ_{n+1})}\pi^*\DR \cdot
\lambda^U_{r_1}\lambda^{U^\vee}_{r_2-1} = \frac{1}{|\Aut(\mu)| |\Aut(\nu)|} \int_{\left[ \Mbar_{g, \mubar\sqcup \nubar\sqcup \gamma}([\bC^2/\bZ_{n+1}]) \right]^\vir}\pi^*\DR.
$$
In order to compute the rubber invariant $\langle \mubar,\nubar\rangle^{\cX,\circ,\sim }_{g,\gamma}$, we first need to study the double ramification cycle $\DR$. A combinatorial expression for $\DR$ is obtained in \cite{JPPZ}, known as Pixton's formula. This formula will be used to study $\langle \mubar,\nubar\rangle^{\cX,\circ,\sim }_{g,\gamma}$ in this paper and we now give a brief description of it.

Let $G_{g,N}$ be the set of all genus $g$ stable graphs with $N$ leaves. To each $\Gamma\in G_{g,N}$, we associate the moduli space
$$
\Mbar_{\Gamma}:=\prod_{v\in V(\Gamma)}\Mbar_{g(v),n(v)}.
$$
Then there is a map
$$
\xi_{\Gamma}:\Mbar_{\Gamma}\to \Mbar_{g,N},
$$
whose image is the closure of the boundary stratum associated with $\Gamma$.

Let $r$ be a positive integer and $\Gamma\in G_{g,N}$. Fix a double ramification datum $A=(a_1,\cdots,a_N)$, where $a_i \in \bZ$ and $\sum_{i=1}^N a_i = 0$. A \emph{weighting mod $r$} of $\Gamma$ is a function
$$
w:H(\Gamma)\to\{1,\cdots,r\},
$$
satisfying
\begin{enumerate}
\item For each $h_i\in L(\Gamma)$ corresponding to the marking $i\in\{1,\cdots,n\}$,
    $$
    w(h_i)=a_i \mod r,
    $$

\item For each $e\in E(\Gamma)$ corresponding to two half-edges $ h,h'\in H(\Gamma)$,
    $$
    w(h)+w(h')=0 \mod r,
    $$

\item For each $v\in V(\Gamma)$,
$$
\sum_{h \textrm{ incident to } v}w(h)=0 \mod r.
$$
\end{enumerate}
We denote by $W_{\Gamma, r}$ the set of all weightings mod $r$ of $\Gamma$, and by $P^{d,r}_g(A)$ the degree $d$ component of the tautological class
\begin{eqnarray*}
&&\sum_{\Gamma\in G_{g,N}}\sum_{w\in W_{\Gamma,r}}\frac{1}{|\Aut{\Gamma}|}\frac{1}{r^{h^1(\Gamma)}}
\xi_{\Gamma*} \left(\prod_{i=1}^{n}\exp(a_i^2\psi_i) \prod_{e=(h,h')\in E(\Gamma)}
\frac{1-\exp(-w(h)w(h')(\psi_h+\psi_{h'}))}{\psi_h+\psi_{h'}}\right)
\end{eqnarray*}
in $R^*(\Mbar_{g,n})$.

Pixton shows that for fixed $g, A,$ and $d$, the class $P^{d,r}_g(A)$ is a polynomial in $r$. Denote by $P^{d}_g(A)$ the constant term of $P^{d,r}_g(A)$. The main result of \cite{JPPZ} is the following theorem:

\begin{theorem}[\cite{JPPZ}] \label{double}
For $g\geq 0$ and double ramification data $A$,
$$
\DR_g(A)=2^{-g}P^{d}_g(A)\in R^g(\Mbar_{g,N}).
$$
\end{theorem}

\subsection{Vanishing property of the Hurwitz--Hodge classes}

\begin{proposition} \label{vanishing}
The Hurwitz--Hodge class $c_{r_1+r_2-1}(\bE_U\oplus \bE_{U^\vee})$ vanishes on the boundary strata of $\Mbar_{g, \mubar \sqcup  \nubar \sqcup  \gamma}(B\bZ_{n+1})$, except those consisting of irreducible singular nodal curves with nontrivial monodromies at nodes.
\end{proposition}

\begin{proof}
We investigate the behavior of Hurwitz--Hodge classes on the boundary. The normalization of a \emph{reducible} twisted nodal curve $C$ in the boundary has at least two connected components. Let $\nu: \widetilde C \to C$ be a partial normalization such that $\widetilde C = C_1\sqcup C_2$ has two connected components. Let $f$ be the number of normalizing nodes. We have
$$\xymatrix{
0 \ar[r] & \cO_C \ar[r] & \nu_* \cO_{\widetilde C} \ar[r] & \bigoplus_{i=1}^f \cO_{p_i} \ar[r] & 0.
}$$
Tensor it with $U \oplus U^\vee$ and consider the long exact sequence (recall that $U$ is the $\bZ_{n+1}$-representation with weight $1$)
$$\xymatrix@R-2pc{
0 \ar[r] & H^0(C, \cO_C\otimes (U\oplus U^\vee)) \ar[r] & \ar[r] H^0(\widetilde C, \cO_{\widetilde C}\otimes (U\oplus U^\vee)) \ar[r] & \bigoplus_{i=1}^f H^0(C, \cO_{p_i} \otimes (U\oplus U^\vee)) \\
 \ar[r] & H^1(C, \cO_C\otimes (U\oplus U^\vee)) \ar[r] & \ar[r] H^1(\widetilde C, \cO_{\widetilde C}\otimes (U\oplus U^\vee)) \ar[r] & 0,
}$$
where the first two terms always vanish and since $p_i = B\bZ_{n+1}$, we have $H^0(C, \cO_{p_i} \otimes (U\oplus U^\vee)) = (\bC^f \otimes (U\oplus U^\vee))^{\bZ_{n+1}}$. Then the sequence becomes
$$\xymatrix{
0 \ar[r] & (\bC^f \otimes (U\oplus U^\vee))^{\bZ_{n+1}}
 \ar[r] & H^1(C, \cO_C\otimes (U\oplus U^\vee)) \ar[r] & H^1(\widetilde C, \cO_{\widetilde C}\otimes (U\oplus U^\vee)) \ar[r] & 0
}. $$

The curve is in the image of
$$
\iota: \Mbar_{g_1, \mubar^1 \sqcup  \nubar^1 \sqcup \gamma^1 \sqcup  \alpha} (B\bZ_{n+1}) \times \Mbar_{g_2, \mubar^2 \sqcup \nubar^2 \sqcup \gamma^2 \sqcup (-\alpha)} (B\bZ_{n+1}) \to \Mbar_{g, \mubar \sqcup \nubar \sqcup \gamma}(B\bZ_{n+1}),
$$
where $g_1+g_2+ f-1=g$, the union of $(\mubar^i, \nubar^i, \gamma^i)$ on the two pieces $i=1,2$ matches the total datum, and $\pm \alpha = \pm (\alpha_1, \cdots, \alpha_f)$ here stand for $f$ markings on either factors with opposite monodromies.

Denote by $\bV:= \bE_U\oplus \bE_{U^\vee}$, $R:=r_1+r_2$ and $\bV_i, R_i$, $i=1,2$ the corresponding bundles and their ranks on the two factors. We have the sequence
$$\xymatrix{
0 \ar[r] & (\cO^{\oplus f}\otimes (U\oplus U^\vee))^{\bZ_{n+1}} \ar[r] & \iota^*\bV \ar[r] & \bV_1 \boxplus \bV_2 \ar[r] & 0.
}$$
Hence,
$$\iota^* c_{R-1}(\bV) = c_{R-1}(\bV_1 \boxplus \bV_2),$$
and the ranks satisfy
\begin{eqnarray*}
R_1+R_2 &=& (2g_1-2 + l(\mubar^1)+ l(\nubar^1) + p_1 + f'') + (2g_2-2 + l(\mubar^2)+ l(\nubar^2) + p_2 + f'') \\
&=& 2(g-f+1) -4 + l(\mu) + l(\nu) + p + 2f'' \\
&=& R-2(f-f'') \\
&\leq& R,
\end{eqnarray*}
where $f''$ is the number of nodal markings with nontrivial monodromies.

If the inequality is strict, then $R_1+R_2\leq R-2$, and $\iota^* c_{R-1}(\bV)$ simply vanishes by dimensional reasons. If the equality holds, which only happens when all nodal markings have nontrivial monodromies, then
$$\iota^* c_{R-1}(\bV) = c_{R_1}(\bV_1) c_{R_2-1}(\bV_2) + c_{R_1-1}(\bV_1) c_{R_2}(\bV_2),$$
which vanishes since the $\bZ_{n+1}$-Mumford relation implies $c_{R_1}(\bV_1)= c_{R_2}(\bV_2)=0$.

We are left with the case when $C$ is irreducible nodal but with some nodes having trivial monodromies. Similarly one can consider the normalization sequence and the rank inequality would be strict. $\iota^* c_{R-1}(\bV)$ still vanishes by dimensional reasons.
\end{proof}

The vanishing result in Proposition \ref{vanishing} will greatly simplify our computation in the next subsection.

\subsection{Computation of rubber invariants and Crepant Resolution Conjecture}
In this subsection, we compute our rubber invariants $\langle \mubar,\nubar\rangle^{\cX,\circ,\sim }_{g,\gamma}$ and prove the corresponding Crepant Resolution Conjecture.

\subsubsection{Pixton's formula expanded in terms of Bernoulli numbers}\label{pixton}
By Theorem \ref{double} and Proposition \ref{vanishing}, in order to compute the rubber invariant
$$
\langle \mubar,\nubar\rangle^{\cX,\circ,\sim }_{g,\gamma}=\frac{(-1)^{r_1+r_2-1}(t_1+t_2)}{|\Aut(\mu)||\Aut(\nu)|}
\int_{\Mbar_{g, \mubar\sqcup \nubar\sqcup \gamma}(B\bZ_{n+1})}\pi^*\DR \cdot
\lambda^U_{r_1}\lambda^{U^\vee}_{r_2-1},
$$
it suffices to consider the restriction of $\pi^*\DR$ to the main stratum, and to the strata of irreducible singular curves with nontrivial monodromies at nodes. Therefore, in the graph sum expression of $P^{d,r}_g(A)$, we only need to consider those graphs with one vertex and $f$ loops, for $0\leq f\leq g$.

Given $\Gamma\in G_{g,l(\mubar)+l(\nubar)+l(\gamma)}$ with $V(\Gamma)=1$ and $E(\Gamma)=f$, we have
\begin{eqnarray*}
&&\prod_{e=(h,h')\in E(\Gamma)}
\frac{1-\exp(-w(h)w(h')(\psi_h+\psi_{h'}))}{\psi_h+\psi_{h'}}\\
&=&-\prod_{e=(h,h')\in E(\Gamma)} \left( (-w(h)(r-w(h)))+\frac{(-w(h)(r-w(h)))^2}{2}(\psi_h+\psi_{h'})
+\cdots \right. \\
&& \left. +\frac{(-w(h)(r-w(h)))^k}{k!}(\psi_h+\psi_{h'})^{k-1} +\cdots \right),
\end{eqnarray*}
where the series must terminate after finite terms, by dimensional reasons.

Recall that the Bernoulli numbers $B_k$ are defined by the following Taylor expansion:
$$
\frac{z }{e^z-1}=\sum_{k=0}^{\infty}B_k \frac{z^k}{k!}.
$$
The only nonzero odd Bernoulli number is $B_1=-\frac{1}{2}$.
The classical Bernoulli's formula implies
$$
1^k+2^k+\cdots+r^k=\frac{1}{k+1}\sum_{i=0}^{k}\frac{(k+1)!}{i!(k+1-i)!}
B_ir^{k+1-i},
$$
where $B_i$ is the $i$-th Bernoulli number. Since $r^{b_1(\Gamma)}=r^f$, in order to pick the constant term of $P^{d,r}_g(A)$, only the term $\frac{w(h)^{2k}}{k!}$ in the factor $\frac{(-w(h)(r-w(h)))^k}{k!}$ survives, because when we sum over $w(h)\in\{1,\cdots,r\}$ the terms in
\begin{equation}\label{Bernoulli}
\sum_{w(h)=1}^rw(h)^{2k}=1^{2k}+2^{2k}+\cdots+r^{2k}=\frac{1}{2k+1}
\sum_{i=0}^{2k}\frac{(2k+1)!}{i!(2k+1-i)!}
B_ir^{2k+1-i},
\end{equation}
is at least $r$-linear. Hence the product over edges would produce an $r^f$ factor. Furthermore, we can only pick the term $\frac{1}{2k+1}\frac{(2k+1)!}{(2k)!(2k+1-2k)!}
B_{2k}r=B_{2k}r$ in the above summation.

Let $[\cdot]_d$ denote the degree $d$ part of a class. Then the relevant part in Pixton's formula expands as
\begin{eqnarray*}
&& 2^{-g}\sum_{f=0}^{g}\frac{(-1)^f}{2^f f!} \cdot (\xi_{\Gamma_f})_* \left( \sum_{M=f}^{g}
\left[ \exp \left( \sum_{i=1}^{l(\mubar)}\mu_i^2\psi_i+
\sum_{j=1}^{l(\nubar)}\nu_j^2\psi_{j+l(\mubar)} \right) \right]_{g-M} \right. \\
&& \left. \cdot \left( \sum_{m_1+\cdots+m_f=M}
\prod_{i=1}^{f}B_{2m_i} \left( \sum_{k_i=0}^{m_i-1}\frac{\psi_{h_i}^{k_i}}{m_i k_i!}
\frac{\psi_{h_i'}^{m_i-1-k_i}}{(m_i-1-k_i)!} \right) \right) \right) \\
&=&\sum_{f=0}^{g}\frac{(-1)^f}{f!} (\xi_{\Gamma_f})_* \left( \sum_{M=f}^{g} 2^{-M}
\cdot \frac{\left( \sum_{i=1}^{l(\mubar)}\mu_i^2\psi_i+
\sum_{j=1}^{l(\nubar)}\nu_j^2\psi_{j+l(\mubar)} \right)^{g-M}}{2^{g-M}(g-M)!} \right. \\
&& \left. \cdot \left( \sum_{m_1+\cdots+m_f=M}
\prod_{i=1}^{f}B_{2m_i} \left( \sum_{k_i=0}^{m_i-1}\frac{\psi_{h_i}^{k_i}}
{(2m_i) k_i!}
\frac{\psi_{h_i'}^{m_i-1-k_i}}{(m_i-1-k_i)!} \right) \right) \right),
\end{eqnarray*}
where $\Gamma_f$ is the unique graph in $G_{g,l(\mubar)+l(\nubar)+l(\gamma)}$ with $|V(\Gamma)|=1$ and $|E(\Gamma)|=f$. In the second line, the appearance of the Bernoulli numbers is caused by the fact that we have picked the term $\frac{1}{2k+1}\frac{(2k+1)!}{(2k)!(2k+1-2k)!}
B_{2k}r=B_{2k}r$ in Equation \eqref{Bernoulli}.

Let us make the following observations which will be used in the next subsubsection. Consider the following diagram:
\begin{equation}
\xymatrix{
	 & \overline\cM_{g, \mubar \sqcup \nubar \sqcup \gamma} (B \bZ_{n+1}) \ar[d]^{\pi} \\
	\overline\cM_{g-f,l(\mubar)+l(\nubar)+l(\gamma)+2f}  \ar[r]^-\xi  &  \overline\cM_{g, l(\mubar)+l(\nubar)+l(\gamma)}.
}
\end{equation}
Here $\pi$ is the forgetful map which forgets the orbifold structure. The map $\xi$ is given by gluing the last $2f$ marked points in pairs. The image of $\xi$ is the closure of the boundary stratum whose dual graph consists of a single vertex and $f$ loops. The map $\pi$ is a finite map of degree $(n+1)^{2g-1}$. One way to see this is that over a generic point $[C]\in \overline\cM_{g, l(\mubar)+l(\nubar)+l(\gamma)}$, the elements in $\pi^{-1}(C)$ are labeled by the choices of monodromies around the $2g$ nontrivial loops on $C$. These monodromies, together with the monodromies given by $\mubar, \nubar,\gamma$,  uniquely determine a principal $\bZ_{n+1}$-bundle over $C$. This implies that $|\pi^{-1}(C)|=(n+1)^{2g}$. On the other hand, each principal $\bZ_{n+1}$-bundle over $C$ has the automorphism group $\bZ_{n+1}$. Therefore, we have $\deg \pi=(n+1)^{2g-1}$.

Let $\alpha_f=(\alpha_{h_1},-\alpha_{h_1},\cdots,\alpha_{h_f},-\alpha_{h_f})$, with $\alpha_{h_1},\cdots,\alpha_{h_f}$ being \emph{nontrivial} elements in $\bZ_{n+1}$. Then there is a forgetful map $\pi':\Mbar_{g-f,\mubar\sqcup \nubar\sqcup \gamma\sqcup \alpha_f}(B \bZ_{n+1})\to \overline\cM_{g-f,l(\mubar)+l(\nubar)+l(\gamma)+2f}$ which forgets the orbifold structure. A point in $\Mbar_{g-f,\mubar\sqcup \nubar\sqcup \gamma\sqcup \alpha_f}(B \bZ_{n+1})$ is represented by the pair $(\tC,\tP)$, where $\tC$ is a twisted curve of genus $g-f$ and $\tP$ is a principal $\bZ_{n+1}$-bundle over $\tC$ with prescribed monodromies $\mubar\sqcup \nubar\sqcup \gamma\sqcup \alpha_f$. Let $x_1,\cdots,x_{2f}$ be the marked points on $\tC$ corresponding to $\alpha_f$. Let $C$ be the curve given by gluing $x_1,\cdots,x_{2f}$ in pairs i.e. by gluing $x_i,x_{i+1}$ together for $i=1,3,\cdots,2f-1$. In order to obtain a point in $\overline\cM_{g, \mubar \sqcup \nubar \sqcup \gamma} (B \bZ_{n+1})$, we need to construct a principal $\bZ_{n+1}$-bundle $P$ over $C$. The bundle $P$ is obtained by gluing the fibers of $\tP$ over $x_i,x_{i+1}$ for $i=1,3,\cdots,2f-1$. For each $i$, there are $(n+1)$ ways to glue the fibers of $\tP$ over $x_i,x_{i+1}$ together. So there are $(n+1)^f$ such choices in total and we choose an arbitrary bijection between these choices and the set $\{1,\cdots,(n+1)^f\}$. For each $j\in\{1,\cdots,(n+1)^f\}$, let $\phi_j:\Mbar_{g-f,\mubar\sqcup \nubar\sqcup \gamma\sqcup \alpha_f}(B \bZ_{n+1})\to \Mbar_{g, \mubar \sqcup \nubar \sqcup \gamma} (B \bZ_{n+1})$ be defined as $\phi_j(\tC,\tP)=(C,P)$, where $P$ is constructed using the data corresponding to $j$. Let $\Mbar_{\alpha_f}=\coprod_{j=1}^{(n+1)^f}\Mbar_{g-f,\mubar\sqcup \nubar\sqcup \gamma\sqcup \alpha_f}(B \bZ_{n+1})$ be the disjoint union of $(n+1)^f$ copies of $\Mbar_{g-f,\mubar\sqcup \nubar\sqcup \gamma\sqcup \alpha_f}$ and let $\phi_{\alpha_f}=\coprod_{j=1}^{(n+1)^f}\phi_j:\Mbar_{\alpha_f}\to \Mbar_{g, \mubar \sqcup \nubar \sqcup \gamma} (B \bZ_{n+1})$. Let $\Mbar=\coprod_{\alpha_f}\Mbar_{\alpha_f}$ and let $\phi=\coprod_{\alpha_f}\phi_{\alpha_f}:\Mbar\to \Mbar_{g, \mubar \sqcup \nubar \sqcup \gamma} (B \bZ_{n+1})$. Then we have the following Cartesian diagram:

\begin{equation}\label{Cart}
\xymatrix{
	\Mbar \ar[r]^-{\phi} \ar[d]_{\pi'} & \overline\cM_{g, \mubar \sqcup \nubar \sqcup \gamma} (B \bZ_{n+1}) \ar[d]^{\pi} \\
	\overline\cM_{g-f,l(\mubar)+l(\nubar)+l(\gamma)+2f}  \ar[r]^-\xi  &  \overline\cM_{g, l(\mubar)+l(\nubar)+l(\gamma)}.
}
\end{equation}
The Cartesian-ness follows from the above argument on $\tC$ and $C$ in family, and $\Mbar$ is a disjoint union because the gluing data here are discrete, and hence have to be locally constant.

\begin{lemma}\label{factor}
Let $a\in A^*(\overline\cM_{g-f,l(\mubar)+l(\nubar)+l(\gamma)+2f})$. Then we have
$$
\int_{\cM_{g, \mubar \sqcup \nubar \sqcup \gamma} (B \bZ_{n+1})}\pi^*\xi_*(a)\lambda^U_{r_1}\lambda^{U^\vee}_{r_2-1}=\sum_{\alpha_f}(n+1)^f \int_{\Mbar_{g-f,\mubar\sqcup \nubar\sqcup \gamma\sqcup \alpha_f}(B \bZ_{n+1})}(\pi')^*(a)\lambda^U_{r_1}\lambda^{U^\vee}_{r_2-1}
$$
\end{lemma}
\begin{proof}
Notice that since $\alpha_{h_1},\cdots,\alpha_{h_f}$ are \emph{nontrivial} elements in $\bZ_{n+1}$, we have $\phi^*(\bE_U)=\bE_U,
\phi^*(\bE_{U^\vee})=\bE_{U^\vee}$. Here we use the same symbol to denote the Hurwitz--Hodge bundles on $\Mbar$ and on $\Mbar_{g, \mubar \sqcup \nubar \sqcup \gamma} (B \bZ_{n+1})$. By the diagram (\ref{Cart}), we have
$$
\int_{\Mbar_{g, \mubar \sqcup \nubar \sqcup \gamma} (B \bZ_{n+1})}\pi^*\xi_*(a)\lambda^U_{r_1}\lambda^{U^\vee}_{r_2-1}=
\int_{\Mbar_{g, \mubar \sqcup \nubar \sqcup \gamma} (B \bZ_{n+1})}\phi_*(\pi')^*(a)\lambda^U_{r_1}\lambda^{U^\vee}_{r_2-1}.
$$
By projection formula, we have
$$
\int_{\Mbar_{g, \mubar \sqcup \nubar \sqcup \gamma} (B \bZ_{n+1})}\phi_*(\pi')^*(a)\lambda^U_{r_1}\lambda^{U^\vee}_{r_2-1}=
\int_{\Mbar}(\pi')^*(a)\lambda^U_{r_1}\lambda^{U^\vee}_{r_2-1}.
$$
Since $\Mbar =\coprod_{\alpha_f} \Mbar_{\alpha_f}$, $\Mbar_{\alpha_f}=\coprod_{j=1}^{(n+1)^f}\Mbar_{g-f,\mubar\sqcup \nubar\sqcup \gamma\sqcup \alpha_f}(B \bZ_{n+1})$, we have
$$
\int_{\Mbar}(\pi')^*(a)\lambda^U_{r_1}\lambda^{U^\vee}_{r_2-1}=\sum_{\alpha_f}(n+1)^f \int_{\Mbar_{g-f,\mubar\sqcup \nubar\sqcup \gamma\sqcup \alpha_f}(B \bZ_{n+1})}(\pi')^*(a)\lambda^U_{r_1}\lambda^{U^\vee}_{r_2-1}.
$$
\end{proof}

\subsubsection{Grothendieck--Riemann--Roch calculation}
The double ramification cycle has been expressed in terms of $\psi$ classes. Now we compute our rubber invariant $\langle \mubar,\nubar\rangle^{\cX,\circ,\sim }_{g,\gamma}$ using the Grothendieck--Riemann--Roch theorem.

At the end of Section 2 in \cite{Zhou}, J. Zhou obtains the following expression for descendent GW invariants of $[\bC^2/\bZ_{n+1}]$, using Grothendieck--Riemann--Roch:

\begin{eqnarray} \label{jianzhou}
 \left\langle \prod_{j=1}^N \tau_{k_j}(a_j) \right\rangle_g^{[\bC^2/\bZ_{n+1}]}
  &=& \frac{(t_1+t_2) (n+1)^{2g-1}(-1)^{r_1}}{2} \sum_{I\sqcup J= [N]} \frac{B_{r_1+r_2} \left( \frac{c(I)}{n+1} \right)}{r_1+r_2} (-1)^{|J|} \cdot \frac{1}{4^g \prod_{j=1}^N (2k_j+1)!!} \nonumber \\
  && + \delta_{N,2} \frac{(t_1+t_2) (n+1)^{2g-2} (-1)^g}{2} \sum_{c=0}^n \frac{B_{2g}\left(\frac{c}{n+1} \right)}{2g} \frac{1}{4^{g-1} \prod_{k_j>0} (2k_j-1)!!} ,
\end{eqnarray}
where $a_j\in \bZ_{n+1}$, $c(I) = - \sum_{i\in I} a_i$. Moreover, we assume that all $a_j$ occurring here are nonzero, $\sum a_j=0$, and $\sum k_j=g$, which hold in our case.

Applying the formula above, the rubber invariant can be written as
$$
\langle \mubar,\nubar\rangle^{\cX,\circ,\sim }_{g,\gamma} = I_{1,g} + I_{2,g},
$$
where $I_{1,g}$ and $I_{2,g}$ are contributed by the two terms in (\ref{jianzhou}) respectively, and $I_{2,g}$ is only nonzero when $l(\mubar) = l(\nubar) = 1$ and $\gamma = \emptyset$.

Let's compute them separately. $|\Aut(\mu)||\Aut(\nu)| \cdot I_{1,g}$ is the corresponding term in
$$
(-1)^{r_1+r_2-1}(t_1+t_2)
\int_{\Mbar_{g, \mubar\sqcup \nubar\sqcup \gamma}(B\bZ_{n+1})} \pi^*\DR \cdot
\lambda^U_{r_1}\lambda^{U^\vee}_{r_2-1} = \int_{\left[ \Mbar_{g, \mubar\sqcup \nubar\sqcup \gamma}([\bC^2/\bZ_{n+1}]) \right]^\vir} \pi^*\DR ,
$$
which is
\begin{eqnarray*}
&&\sum_{f=0}^{g}\frac{(-1)^f(n+1)^{f}}{f!}\sum_{\alpha_f}
\int_{\left[ \Mbar_{g-f,\mubar\sqcup \nubar\sqcup \gamma\sqcup \alpha_f}([\bC^2/\bZ_{n+1}]) \right]^\vir}
\sum_{M=f}^{g}2^{-M} \cdot \frac{\left( \sum_{i=1}^{l(\mubar)}\mu_i^2\psibar_i+
\sum_{j=1}^{l(\nubar)}\nu_j^2\psibar_{j+l(\mubar)} \right)^{g-M}}{2^{g-M}(g-M)!} \\
&& \cdot \left( \sum_{m_1+\cdots+m_f=M}
\prod_{i=1}^{f}B_{2m_i} \left( \sum_{k_i=0}^{m_i-1}\frac{\psibar_{h_i}^{k_i}}
{(2m_i) k_i!} \frac{\psibar_{h_i'}^{m_i-1-k_i}}{(m_i-1-k_i)!} \right) \right) \\
&=&\sum_{f=0}^{g}\frac{(-1)^f(n+1)^{f}}{f!}\sum_{\alpha_f}
\int_{\left[ \Mbar_{g-f, \mubar\sqcup \nubar\sqcup \gamma\sqcup \alpha_f}([\bC^2/\bZ_{n+1}]) \right]^\vir}
\sum_{M=f}^{g} 2^{-M} \\
&& \cdot \sum_{\sum_{i}k_i+\sum_{j}l_j=g-M}
\prod_{i=1}^{l(\mubar)}\frac{\mu_i^{2k_i}\psibar_i^{k_i}}{2^{k_i}k_i!}
\prod_{j=1}^{l(\nubar)}\frac{\nu_j^{2l_i}
\psibar_{j+l(\mubar)}^{l_j}}{2^{l_j}l_j!} \cdot \left( \sum_{m_1+\cdots+m_f=M}
\prod_{i=1}^{f}B_{2m_i} \left( \sum_{k_i=0}^{m_i-1}\frac{\psibar_{h_i}^{k_i}}
{(2m_i)
k_i!}
\frac{\psibar_{h_i'}^{m_i-1-k_i}}{(m_i-1-k_i)!} \right) \right) \\
&=&\sum_{f=0}^{g}\frac{(-1)^f(n+1)^{f}}{f!} \sum_{\alpha_f} \sum_{M=f}^{g} 2^{-M} \frac{(t_1+t_2) (n+1)^{2g-2f-1}(-1)^{r_1}}{2} \\
&& \cdot \sum_{I\sqcup J= \mubar\sqcup \nubar\sqcup \gamma\sqcup \alpha_f} \frac{B_{r_1+r_2} \left( \frac{c(I)}{n+1} \right)}{r_1+r_2} (-1)^{|J|} \sum_{\sum_{i}k_i+\sum_{j}l_j=g-M} \prod_{i=1}^{l(\mubar)}\frac{(\frac{\mu_i}{2})^{2k_i}}{(2k_i+1)!}
\prod_{j=1}^{l(\nubar)}\frac{(\frac{\nu_i}{2})^{2l_i}
}{(2l_j+1)!} \\
&&\cdot \frac{1}{4^{M-f}} \sum_{m_1+\cdots+m_f=M}
\prod_{i=1}^{f} \left( \sum_{k_i=0}^{m_i-1}\frac{B_{2m_i}}{(2m_i)(2k_i+1)!!k_i!}
\frac{1}{(2(m_i-1-k_i)+1)!!(m_i-1-k_i)!} \right) \\
&=&\sum_{f=0}^{g}\frac{(-1)^f(n+1)^{f}}{f!}\sum_{\alpha_f}
\sum_{M=f}^{g} 2^{-M} \frac{(t_1+t_2) (n+1)^{2g-2f-1}(-1)^{r_1}}{2} \\
&& \cdot \sum_{I\sqcup J= \mubar\sqcup \nubar\sqcup \gamma\sqcup \alpha_f} \frac{B_{r_1+r_2} \left( \frac{c(I)}{n+1} \right)}{r_1+r_2} (-1)^{|J|}  \sum_{\sum_{i}k_i+\sum_{j}l_j=g-M}
 \prod_{i=1}^{l(\mubar)}\frac{(\frac{\mu_i}{2})^{2k_i}}{(2k_i+1)!}
\prod_{j=1}^{l(\nubar)}\frac{(\frac{\nu_i}{2})^{2l_i}
}{(2l_j+1)!} \\
&&\cdot \frac{1}{2^{M-f}} \sum_{m_1+\cdots+m_f=M}
\prod_{i=1}^{f} \left( \sum_{k_i=0}^{m_i-1}\frac{B_{2m_i}}{(2m_i)(2k_i+1)!}
\frac{1}{(2(m_i-1-k_i)+1)!} \right)\\
&=&\sum_{f=0}^{g}\frac{(-1)^f(n+1)^{f}}{f!}\sum_{\alpha_f}
\sum_{M=f}^{g}\frac{(t_1+t_2) (n+1)^{2g-2f-1}(-1)^{r_1}}{2} \cdot \sum_{I\sqcup J= \mubar\sqcup \nubar\sqcup \gamma\sqcup \alpha_f} \frac{B_{r_1+r_2} \left( \frac{c(I)}{n+1} \right)}{r_1+r_2} (-1)^{|J|} \\
&& \cdot [z^{2g-2M}] \left( \prod_{i=1}^{l(\mubar)}\cS(\mu_i z)\prod_{j=1}^{l(\nubar)}\cS(\nu_j z) \right) \cdot \left( \sum_{m_1+\cdots+m_f=M}
\prod_{i=1}^{f}\frac{B_{2m_i}}{(2m_i)(2m_i)!} \right).
\end{eqnarray*}
Here the summation $\sum_{\alpha_f}$ is over all tuples $\alpha_f=(\alpha_{h_1},-\alpha_{h_1},\cdots,\alpha_{h_f},-\alpha_{h_f})$ with $\alpha_{h_1},\cdots,\alpha_{h_f}$ being \emph{nontrivial} elements in $\bZ_{n+1}$, corresponding to monodromies around the nodes; in the last equality we used the fact that $$\sum_{k_i=0}^{m_i-1}\frac{(2m_i)!}{(2k_i+1)!}
\frac{1}{(2(m_i-1-k_i)+1)!}=2^{2m_i-1};$$
the function $\cS(z)$ is defined as
$$\cS(z):= \frac{\sinh(z/2)}{z/2}.$$
The appearance of the factor $(n+1)^{f}$ in the second equality follows from Lemma \ref{factor}.

By the trick in Section 3.1 of \cite{Zhou}, we can rewrite the term involving Bernoulli numbers as follows.
\begin{eqnarray*}
  \sum_{I\sqcup J= [N]} \frac{B_{r_1+r_2} \left( \frac{c(I)}{n+1} \right)}{r_1+r_2} (-1)^{|J|}
  &=& \frac{1}{n+1} \sum_{c=0}^n \sum_{l=0}^n \zeta^{lc}\zeta^{l\sum_{i\in I} a_i} \sum_{I\sqcup J= [N]} \frac{B_{r_1+r_2} \left( \frac{c}{n+1} \right)}{r_1+r_2} (-1)^{|J|} \\
  &=& \frac{1}{n+1} \sum_{c=0}^n \sum_{l=0}^n \zeta^{lc} \frac{B_{r_1+r_2} \left( \frac{c}{n+1} \right)}{r_1+r_2} \sum_{I\sqcup J= [N]} \zeta^{l\sum_{i\in I} a_i}(-1)^{|J|} \\
  &=& \frac{1}{n+1} \sum_{l=1}^n \sum_{c=0}^n \zeta^{lc} \frac{B_{r_1+r_2} \left( \frac{c}{n+1} \right)}{r_1+r_2} \prod_{j=1}^N (\zeta^{l a_j}-1). \\
\end{eqnarray*}
Therefore, let $a$, $b$ be the tuples of markings determined by $\mubar$, $\nubar$ respectively, and we have
\begin{eqnarray*}
I_{1,g} &=& \frac{1}{|\Aut(\mu)||\Aut(\nu)|} \sum_{f=0}^{g}\frac{(-1)^f(n+1)^{f}}{f!}\sum_{\alpha_f}
\sum_{M=f}^{g}\frac{(t_1+t_2) (n+1)^{2g-2f-2}(-1)^{r_1}}{2} \\
&& \cdot \sum_{l=1}^n \sum_{c=0}^n \zeta^{lc} \frac{B_{r_1+r_2} \left( \frac{c}{n+1} \right)}{r_1+r_2} \cdot \prod_{i=1}^{l(\mubar)} (\zeta^{l a_i}-1)\prod_{j=1}^{l(\nubar)} (\zeta^{l b_j}-1)\prod_{k=1}^{l(\gamma)} (\zeta^{l \gamma_k}-1)\prod_{i=1}^{f} (\zeta^{l \alpha_{h_i}}-1)
(\zeta^{-l \alpha_{h_i}}-1)\\
&&\cdot [z^{2g-2M}] \left(\prod_{i=1}^{l(\mubar)}\cS(\mu_i z)\prod_{j=1}^{l(\nubar)}\cS(\nu_j z) \right) \left( \sum_{m_1+\cdots+m_f=M}
\prod_{i=1}^{f}\frac{B_{2m_i}}{(2m_i)(2m_i)!} \right).
\end{eqnarray*}
Notice that since $l\neq 0$, $$\sum_{\alpha_{h_i}=1}^{n}(\zeta^{l \alpha_{h_i}}-1)(\zeta^{-l \alpha_{h_i}}-1)=\sum_{\alpha_{h_i}=1}^{n}(2-\zeta^{l \alpha_{h_i}}-\zeta^{-l \alpha_{h_i}})=2(n+1).$$
Hence, one can eliminate the sum $\sum_{\alpha_f}$ and obtain
\begin{eqnarray*}
I_{1,g} &=& \frac{1}{|\Aut(\mu)||\Aut(\nu)|} \sum_{f=0}^{g}\frac{(-1)^f2^f}{f!}
\sum_{M=f}^{g}\frac{(t_1+t_2) (n+1)^{2g-2}(-1)^{r_1}}{2} \\
&&\cdot\sum_{l=1}^n \sum_{c=0}^n \zeta^{lc} \frac{B_{r_1+r_2} \left( \frac{c}{n+1} \right)}{r_1+r_2} \cdot \prod_{i=1}^{l(\mubar)} (\zeta^{l a_i}-1)\prod_{j=1}^{l(\nubar)} (\zeta^{l b_j}-1)\prod_{k=1}^{l(\gamma)} (\zeta^{l \gamma_k}-1)\\
&&\cdot [z^{2g-2M}] \left( \prod_{i=1}^{l(\mubar)}\cS(\mu_i z)\prod_{j=1}^{l(\nubar)}\cS(\nu_j z) \right) \left( \sum_{m_1+\cdots+m_f=M}
\prod_{i=1}^{f}\frac{B_{2m_i}}{(2m_i)(2m_i)!} \right).
\end{eqnarray*}

Now let's compute $I_{2,g}$. Let $l(\mubar) = l(\nubar) = 1$ and $\gamma = \emptyset$. First we notice that second term of \eqref{jianzhou} only contribute to the case when $f=0$. So the corresponding contribution in $\langle \mubar,\nubar\rangle^{\cX,\circ,\sim }_{g,\emptyset}$ is
$$
\int_{\left[ \Mbar_{g,\mubar\sqcup \nubar }([\bC^2/\bZ_{n+1}]) \right]^\vir}
\sum_{M=0}^{g}2^{-M} \cdot \frac{\left( \mu_1^2\psibar_1+
	\nu_1^2\psibar_2 \right)^{g-M}}{2^{g-M}(g-M)!}.
$$
Moreover, by dimensional constraints, only the $M=0$ term contributes. Therefore
$$
I_{2,g} = \int_{\left[ \Mbar_{g,\mubar\sqcup \nubar}([\bC^2/\bZ_{n+1}]) \right]^\vir}
	 \frac{\left( \mu_1^2\psibar_1+
	 	\nu_1^2\psibar_2 \right)^{g}}{2^{g}g!} =	\int_{\left[ \Mbar_{g,\mubar\sqcup \nubar }([\bC^2/\bZ_{n+1}]) \right]^\vir}2^{-g}m^{2g}\cdot\sum_{k_1+l_1=g}\frac{\psibar_1^{k_1}}{k_1!}\frac{\psibar_2^{l_1}}{l_1!},
$$
where $g\geq 1$. So we have
\begin{eqnarray*}
	I_{2,g}	
	&=&2^{-g}m^{2g}\frac{(t_1+t_2)(n+1)^{2g-2}(-1)^g}{2}\sum_{c=0}^{n}\frac{B_{2g}\left(\frac{c}{n+1}\right)}{2g}\sum_{k_1+l_1=g}\frac{1}{4^{g-1}(2k_1-1)!!k_1!(2l_1-1)!!l_1!}\\
	&=&m^{2g}\frac{(t_1+t_2)(n+1)^{2g-2}(-1)^g}{2}\sum_{c=0}^{n}\frac{B_{2g}\left(\frac{c}{n+1}\right)}{2g}\sum_{k_1+l_1=g}\frac{1}{4^{g-1}(2k_1)!(2l_1)!}\\
		&=&m^{2g}\frac{(t_1+t_2)(n+1)^{2g-2}(-1)^g}{2}\sum_{c=0}^{n}\frac{B_{2g}\left(\frac{c}{n+1}\right)}{2g(2g)!}\frac{1}{4^{g-1}}2^{2g-1}\\
		&=&m^{2g}(t_1+t_2)(n+1)^{2g-2}(-1)^g\sum_{c=0}^{n}\frac{B_{2g}\left(\frac{c}{n+1}\right)}{2g(2g)!} .
	\end{eqnarray*}

\subsubsection{Generating functions}

\begin{definition}
Define the generating function for rubber invariants as
$$
Z(x,z)_{\mubar,\nubar}^{\circ, \sim}:=\sum_{g\geq 0, \gamma}(-1)^gz^{2g}\frac{x_\gamma}{\gamma!}\langle \mubar,\nubar\rangle^{\cX,\circ,\sim }_{g,\gamma},
$$
where $\gamma=(\gamma_1,\cdots,\gamma_p)$ with $0\neq\gamma_i\in\bZ_{n+1}$, $x_\gamma=x_{\gamma_1}\cdots x_{\gamma_p}$, and we use the more intuitive notation $\gamma!$ to denote $|\Aut(\gamma)|$. The factor $\gamma!$ appears because we would like to count those extra marked points as \emph{unordered}. Moreover, the summation is over all rubber invariants satisfying the $(\dagger)$ assumption at the beginning of Section 5.
\end{definition}

In order for the moduli space $\Mbar_{g, a\sqcup b\sqcup \gamma}([\bC^2/\bZ_{n+1}])$ to be nonempty, we must have
$$
\sum_{i=1}^{l(\mubar)}a_i+\sum_{j=1}^{l(\nubar)}b_j
+\sum_{k=1}^{l(\gamma)}\gamma_k=0 \,\ \mod (n+1).
$$

Recall that the Bernoulli polynomials $B_k(t)$, and Bernoulli numbers $B_k$ are defined by the following Taylor expansion:
$$
\frac{z e^{tz}}{e^z-1}=\sum_{k=0}^{\infty}B_k(t) \frac{z^k}{k!}, \qquad B_k:= B_k(0).
$$
The only nonzero odd Bernoulli number is $B_1=-\frac{1}{2}$. Define the generating function
\begin{eqnarray*}
F(z) \ :=\  \sum_{k=1}^{\infty}B_{2k}\frac{z^{2k}}{(2k)!} \ =\  \frac{z}{e^z-1}+\frac{z}{2}-1.
\end{eqnarray*}

First we assume that at least one of $l(\mubar)$ and $l(\nubar)$ is greater or equal to 2. Then the second term of \eqref{jianzhou} does not contribute to $Z(x,z)_{\mubar,\nubar}^{\circ, \sim}$. So with the observations above, the generating function $Z(x,z)_{\mubar,\nubar}^{\circ, \sim}$ can be expressed as
\begin{eqnarray*}
&&|\Aut(\mu)||\Aut(\nu)| \cdot Z(x,z)_{\mubar,\nubar}^{\circ, \sim}\\
&=&\sum_{g\geq 0}\sum_{\gamma}(-1)^gz^{2g}\frac{x^\gamma}{\gamma!}
\cdot \frac{1}{n+1} \sum_{b=0}^{n}\prod_{i=1}^{l(\mubar)}\zeta^{ba_i}
\prod_{j=1}^{l(\nubar)}\zeta^{bb_j}\prod_{k=1}^{l(\gamma)}\zeta^{b\gamma_k}
\sum_{f=0}^{g}\frac{(-1)^f2^f}{f!} \\
&& \cdot \sum_{M=f}^{g}\frac{(t_1+t_2) (n+1)^{2g-2}(-1)^{r_1}}{2} \sum_{l=1}^n \sum_{c=0}^n \zeta^{lc} \frac{B_{r_1+r_2} \left( \frac{c}{n+1} \right)}{r_1+r_2} \prod_{i=1}^{l(\mubar)} (\zeta^{l a_i}-1)\prod_{j=1}^{l(\nubar)} (\zeta^{l b_j}-1)\prod_{k=1}^{l(\gamma)} (\zeta^{l \gamma_k}-1)\\
&&\cdot [z^{2g-2M}] \left( \prod_{i=1}^{l(\mubar)}\cS(\mu_i z)\prod_{j=1}^{l(\nubar)}\cS(\nu_j z) \right) \left( \sum_{m_1+\cdots+m_f=M}
\prod_{i=1}^{f}\frac{B_{2m_i}}{(2m_i)(2m_i)!} \right) \\
&=&-\sum_{g\geq 0}z^{2g}
\sum_{b=0}^{n}\sum_{l=1}^n\prod_{i=1}^{l(\mubar)}\zeta^{ba_i+\frac{a_i}{2}}
\prod_{j=1}^{l(\nubar)}\zeta^{bb_j+\frac{b_j}{2}}
\cdot \sum_{p\geq0}\frac{1}{p!} \left( \sum_{a=1}^n\zeta^{ba+\frac{a}{2}} (\zeta^{l a}-1)x_a \right)^p
\sum_{f=0}^{g}\frac{(-1)^f2^f}{f!}\\
&&\sum_{M=f}^{g}\frac{(t_1+t_2) (n+1)^{2g-3}}{2}
\cdot \sum_{c=0}^n \zeta^{lc} \frac{B_{r_1+r_2} \left( \frac{c}{n+1} \right)}{r_1+r_2} \prod_{i=1}^{l(\mubar)} (\zeta^{l a_i}-1)\prod_{j=1}^{l(\nubar)} (\zeta^{l b_j}-1)\\
&&\cdot [z^{2g-2M}] \left( \prod_{i=1}^{l(\mubar)}\cS(\mu_i z)\prod_{j=1}^{l(\nubar)}\cS(\nu_j z) \right) \cdot [z^{2M}] \left( \int \frac{F(z)}{z} \right)^f,
\\
\end{eqnarray*}
where we've fixed a certain square root $\zeta^{\frac{1}{2}} = e^{\frac{\pi i}{n+1}}$ of $\zeta$, and by $\int \frac{F(z)}{z}$ we always mean the power series obtained by termwise integration with constant term $0$; we also used the formula for the rank $r_1$ given in Section \ref{obs bundle}.

When $l(\mubar)=l(\nubar)=1$, then $I_{2,g}$ also contributes to $Z(x,z)_{\mubar,\nubar}^{\circ, \sim}$. Its contribution is
\begin{eqnarray*}
\sum_{g\geq 1}(-1)^g z^{2g}I_{2,g}&=&\sum_{g\geq 1}m^{2g}(t_1+t_2)(n+1)^{2g-2}\sum_{c=0}^{n}\frac{B_{2g}\left(\frac{c}{n+1}\right)}{2g(2g)!}z^{2g}\\
&=&\frac{(t_1+t_2)}{(n+1)^2}\sum_{g\geq 1}\sum_{c=0}^{n}\frac{B_{2g}\left(\frac{c}{n+1}\right)}{2g(2g)!}\big((n+1)mz\big)^{2g} .
\end{eqnarray*}
Notice that for any formal variable $u$,
$$
\frac{1}{1-e^u}=\frac{\sum_{c=0}^{n}e^{cu}}{1-e^{(n+1)u}}=-\sum_{c=0}^{n}\sum_{k=0}^{\infty}\frac{B_{k}\left(\frac{c}{n+1}\right)}{k!}\big((n+1)u\big)^{k-1}.
$$
Let
$$
H(z)=(n+1)\int\left(\frac{1}{e^z-1}-\frac{1}{z}\right)dz=(n+1)\left(\log\frac{e^z-1}{e^z}-\log z\right),
$$
where we let the integration constant be zero. Then we have
\begin{eqnarray}\label{genI2}
\sum_{g\geq 1}(-1)^g z^{2g}I_{2,g} &=& \frac{(t_1+t_2)}{(n+1)^2}\cdot\frac{H(mz)+H(-mz)}{2} \nonumber \\
&=&\frac{(t_1+t_2)}{2(n+1)}\left(\log (e^{mz}-1)+\log(e^{-mz}-1)-\log (mz)-\log (-mz)\right).
	\end{eqnarray}

We will also need the generating function encoding rubber invariants for all $\mubar$, $\nubar$.

\begin{definition}\label{CV1}
For any $j\in\{1,\cdots,n\}$, we introduce a formal variable $y_j$ and define the following change of variables:
$$
y_j \ =\ \frac{2\sqrt{-1}}{n+1}\sum_{a=1}^n\sin\frac{a\pi}{n+1}\cdot\zeta^{ja}
x_a \ =\  \frac{1}{n+1} \sum_{a=1}^n (\zeta^{a/2}- \zeta^{-a/2}) \zeta^{ja} x_a.
$$
For any $1\leq s\leq t\leq n$, define
$$
y_{s\to t}=y_s+\cdots+y_t.
$$
\end{definition}

By Lemma 3.3 of \cite{Zhou}, we have the following lemma:
\begin{lemma}\label{change1}
For $0\leq b\leq n$, $1\leq l\leq n$, we have
$$
\sum_{a=1}^n\zeta^{ba+\frac{a}{2}} (\zeta^{l a}-1)x_a=\left\{\begin{array}{ll}(n+1)y_{b+1\to b+l}, & \qquad b+l< n+1,\\
-(n+1)y_{b+l-n\to b}, & \qquad b+l\geq n+1,\end{array} \right.
$$
\end{lemma}

\begin{definition}\label{CV2}
For any integer $d>0$, $i\in\{1,2\}$ and any $a\in\{1,\cdots,n\}$, we introduce two formal variables $p^i_{d,a}$, $\hat{p}^i_{d,a}$ and define the following change of variables:
$$
\hat{p}^i_{d,j}=\frac{1}{n+1} \sum_{a=1}^n (\zeta^{a/2}- \zeta^{-a/2}) \zeta^{ja} p^i_{d,a}.
$$
For any $1\leq s\leq t\leq n$, define
$$
\hat{p}^i_{d,s\to t}=\hat{p}^i_{d,s}+\cdots+\hat{p}^i_{d,t}.
$$
For any $\bZ_{n+1}$-weighted partition $\mubar=\{(\mu_1,a_1),\cdots,(\mu_{l(\mubar)},a_{l(\mubar)})\}$, define
$$
p^i_{\mubar}=p^i_{\mu_1,a_1}\cdots p^i_{\mu_{l(\mubar)},a_{l(\mubar)}}, \qquad \hat{p}^i_{\mubar}=\hat{p}^i_{\mu_1,a_1}\cdots \hat{p}^i_{\mu_{l(\mubar)},a_{l(\mubar)}}.
$$
Given any ordinary partition $\mu=\{\mu_1,\cdots,\mu_{l(\mu)}\}$, define
$$
\hat{p}^i_{\mu,s\to t}=(\hat{p}^i_{\mu_1,s}+\cdots+\hat{p}^i_{\mu_1,t})\cdots
(\hat{p}^i_{\mu_{l(\mu)},s}+\cdots+\hat{p}^i_{\mu_{l(\mu)},t}).
$$
\end{definition}
Similar to Lemma \ref{change1}, we have the following lemma:

\begin{lemma}\label{change2}
For $d>0$, $i\in\{1,2\}$, $0\leq b\leq n$, $1\leq l\leq n$, we have
$$
\sum_{a=1}^n\zeta^{ba+\frac{a}{2}} (\zeta^{l a}-1)p^i_{d,a}=\left\{\begin{array}{ll}(n+1)\hat{p}^i_{d,b+1\to b+l}, & \qquad b+l< n+1,\\
-(n+1)\hat{p}^i_{d,b+l-n\to b}, & \qquad b+l\geq n+1,\end{array} \right.
$$
\end{lemma}

\begin{definition}
Define the generating function $Z(x,z,p^1,p^2)^{\circ, \sim}$ as
$$
Z(x,z,p^1,p^2)^{\circ, \sim}: =\sum_{|\mubar|=|\nubar|=m} Z(x,z)_{\mubar,\nubar}^{\circ, \sim}
p^1_{\mubar}p^2_{\nubar}.
$$
\end{definition}

Applying Lemma \ref{change1} and Lemma \ref{change2} to the generating function $Z(x,z,p^1,p^2)^{\circ,\sim}$, we have
\begin{eqnarray*}
&&Z(x,z,p^1,p^2)^{\circ,\sim}-\sum_{a=1}^{n}\sum_{g\geq 1}(-1)^g z^{2g}I_{2,g}p^1_{m,a}p^2_{m,n+1-a}\\
&=&-\sum_{|\mubar|=|\nubar|=m} \frac{p^1_{\mubar}p^2_{\nubar}}{|\Aut(\mu)||\Aut(\nu)|} \sum_{g\geq 0}z^{2g}
\sum_{b=0}^{n} \sum_{\substack{1\leq l\leq n \\ b+l<n+1}}\prod_{i=1}^{l(\mubar)}\zeta^{ba_i+\frac{a_i}{2}}
\prod_{j=1}^{l(\nubar)}\zeta^{bb_j+\frac{b_j}{2}} \\
&& \sum_{p\geq0}\frac{1}{p!}
((n+1)y_{b+1\to b+l})^p \sum_{f=0}^{g}\frac{(-1)^f2^f}{f!} \sum_{M=f}^{g}\frac{(t_1+t_2) (n+1)^{2g-3}}{2}
\cdot \sum_{c=0}^n \zeta^{lc} \frac{B_{r_1+r_2} \left( \frac{c}{n+1} \right)}{r_1+r_2} \\
&& \cdot \prod_{i=1}^{l(\mubar)} (\zeta^{l a_i}-1)\prod_{j=1}^{l(\nubar)} (\zeta^{l b_j}-1) \cdot [z^{2g-2M}] \left( \prod_{i=1}^{l(\mubar)}\cS(\mu_i z)\prod_{j=1}^{l(\nubar)}\cS(\nu_j z) \right) \cdot [z^{2M}] \left( \int \frac{F(z)}{z} \right)^f
\\
&&-\sum_{|\mubar|=|\nubar|=m} \frac{p^1_{\mubar}p^2_{\nubar}}{|\Aut(\mu)||\Aut(\nu)|} \sum_{g\geq 0}z^{2g}
\sum_{b=0}^{n} \sum_{\substack{1\leq l\leq n \\ b+l\geq n+1}} \prod_{i=1}^{l(\mubar)}\zeta^{ba_i+\frac{a_i}{2}}
\prod_{j=1}^{l(\nubar)}\zeta^{bb_j+\frac{b_j}{2}} \\
&& \cdot \sum_{p\geq0}\frac{1}{p!}
(-(n+1)y_{b+l-n\to b})^p
\sum_{f=0}^{g}\frac{(-1)^f2^f}{f!} \sum_{M=f}^{g}\frac{(t_1+t_2) (n+1)^{2g-3}}{2}
\cdot \sum_{c=0}^n \zeta^{lc} \frac{B_{r_1+r_2} \left( \frac{c}{n+1} \right)}{r_1+r_2} \\
&& \cdot \prod_{i=1}^{l(\mubar)} (\zeta^{l a_i}-1)\prod_{j=1}^{l(\nubar)} (\zeta^{l b_j}-1) \cdot [z^{2g-2M}] \left( \prod_{i=1}^{l(\mubar)}\cS(\mu_i z)\prod_{j=1}^{l(\nubar)}\cS(\nu_j z) \right) \cdot [z^{2M}] \left( \int \frac{F(z)}{z} \right)^f \\
&=&-\sum_{|\mu|=|\nu|=m} \sum_{g\geq 0}z^{2g}
\sum_{b=0}^{n}\sum_{\substack{1\leq l\leq n \\ b+l<n+1}} \frac{(n+1)^{l(\mu)+l(\nu)}
\hat{p}^1_{\mu,b+1\to b+l}\hat{p}^2_{\nu,b+1\to b+l}}{|\Aut(\mu)||\Aut(\nu)|} \\
&& \cdot \sum_{p\geq0}\frac{1}{p!}
((n+1)y_{b+1\to b+l})^p
\sum_{f=0}^{g}\frac{(-1)^f2^f}{f!} \sum_{M=f}^{g}\frac{(t_1+t_2) (n+1)^{2g-3}}{2} \sum_{c=0}^n \zeta^{lc} \frac{B_{r_1+r_2} \left( \frac{c}{n+1} \right)}{r_1+r_2} \\
&&\cdot [z^{2g-2M}] \left( \prod_{i=1}^{l(\mubar)}\cS(\mu_i z)\prod_{j=1}^{l(\nubar)}\cS(\nu_j z) \right) \cdot [z^{2M}] \left( \int \frac{F(z)}{z} \right)^f
\\
&&-\sum_{|\mu|=|\nu|=m}\sum_{g\geq 0}z^{2g}
\sum_{b=0}^{n} \sum_{\substack{1\leq l\leq n \\ b+l\geq n+1}}
\frac{(-(n+1))^{l(\mu)+l(\nu)}
\hat{p}^1_{\mu,b+l-n\to b}\hat{p}^2_{\nu,b+l-n\to b}}{|\Aut(\mu)| |\Aut(\nu)|} \\
&& \cdot \sum_{p\geq0}\frac{1}{p!}
(-(n+1)y_{b+l-n\to b})^p
\sum_{f=0}^{g}\frac{(-1)^f2^f}{f!} \sum_{M=f}^{g}\frac{(t_1+t_2) (n+1)^{2g-3}}{2} \sum_{c=0}^n \zeta^{lc} \frac{B_{r_1+r_2} \left( \frac{c}{n+1} \right)}{r_1+r_2} \\
&&\cdot [z^{2g-2M}] \left( \prod_{i=1}^{l(\mubar)}\cS(\mu_i z)\prod_{j=1}^{l(\nubar)}\cS(\nu_j z) \right) \cdot [z^{2M}] \left( \int \frac{F(z)}{z} \right)^f\\
&=&-2\sum_{|\mu|=|\nu|=m}\sum_{g\geq 0}z^{2g}
\sum_{1\leq s\leq t\leq n}
\frac{\hat{p}^1_{\mu,s\to t}\hat{p}^2_{\nu,s\to t}}{|\Aut(\mu)| |\Aut(\nu)|}
\cdot \sum_{p\geq0}\frac{1}{p!}
(y_{s\to t})^p
\sum_{f=0}^{g}\frac{(-1)^f2^f}{f!}\\
&&\sum_{M=f}^{g}\frac{(t_1+t_2) (n+1)^{2g-3+l(\mu)+l(\nu)+p}}{2}
\cdot \sum_{c=0}^n \zeta^{c(t-s+1)} \frac{B_{r_1+r_2} \left( \frac{c}{n+1} \right)}{r_1+r_2} \\
&&\cdot [z^{2g-2M}] \left( \prod_{i=1}^{l(\mubar)}\cS(\mu_i z)\prod_{j=1}^{l(\nubar)}\cS(\nu_j z) \right) \cdot [z^{2M}] \left( \int \frac{F(z)}{z} \right)^f,
\end{eqnarray*}
where in the last equality, we have used the fact
$$
B_m(1-x)=(-1)^mB_m(x).
$$

\subsection{Crepant resolution for rubber invariants}

In this subsection, we compare the computations above with the results (Proposition 3.6 and 4.3) in \cite{Mau}. Recall the generating functions for the relative GW theory of $\cA_n\times \bP^1$ and the rubber theory are defined in \cite{Mau} as
\begin{eqnarray*}
Z'_\GW(\cA_n\times \bP^1)_{\vec\mu, \vec\nu}^{\circ, \sim} &:=& \sum_{g, \beta} z^{2g} s_1^{\beta\cdot\omega_1} \cdots s_n^{\beta\cdot\omega_n} \langle \vec\mu, \vec\nu \rangle_{g, (\beta, m)}^{\circ, \sim}, \\
Z'_\GW(\cA_n\times \bP^1)_{\vec\mu^1, \cdots, \vec\mu^r} &:=& \sum_{\chi, \beta} z^{-\chi} s_1^{\beta\cdot \omega_1} \cdots s_n^{\beta\cdot \omega_n} \langle \vec\mu^1, \cdots, \vec\mu^r \rangle^{\cA_n\times \bP^1, \bullet}_{\chi, (\beta, m)},
\end{eqnarray*}
where $\vec\mu^i\in \cF_{\cA_n}$.

Recall the explicit isomorphism between the two Fock spaces. If we identify $\bZ_{n+1}$ with the orbifold cohomology $H^*_\orb([\bC^2/\bZ_{n+1}])$, the isomorphism is given by
$$\Phi: H^*_\orb([\bC^2/\bZ_{n+1}]) \cong H^*(\cA_n),$$
\begin{equation*}
e_0 \mapsto 1, \qquad e_i \mapsto \frac{\zeta^{i/2} - \zeta^{-i/2}}{n+1} \sum_{j=1}^n \zeta^{ij} \omega_j, \qquad 1\leq i\leq n,
\end{equation*}
where $\omega_1, \cdots, \omega_n \in H^2(\cA_n, \bQ)$ is the dual basis to the exceptional curves in $\cA_n$.

\begin{lemma}\label{I2}
	Let $\mubar=(m,k),\nubar=(m,-k),k\neq 0$ and let $\vec\mu, \vec\nu\in \cF_{\cA_n}$ be their correspondents under the isomorphism above. Then
	$$
	\sum_{g\geq 1}(-1)^g z^{2g}I_{2,g}=Z'_{\GW,\beta=0}(\cA_n\times \bP^1)^{\circ, \sim}_{\vec\mu, \vec\nu}
	$$
\end{lemma}
\begin{proof}
This follows from Equation \eqref{genI2} and Proposition \ref{Anrubber}.
\end{proof}

\begin{theorem}[Rubber GW crepant resolution] \label{thm:rubber}
Given $\mubar, \nubar\in \cF_{[\bC^2/\bZ_{n+1}]}$, let $\vec\mu, \vec\nu\in \cF_{\cA_n}$ be their correspondents under the isomorphism above. Then under the change of variables
  $$s_j = \zeta \exp \left( \frac{1}{n+1} \sum_{a=1}^n (\zeta^{a/2}- \zeta^{-a/2}) \zeta^{ja} x_a \right), \qquad 1\leq j\leq n.$$
we have:
\begin{enumerate}[1)]
\item When $l(\mubar)+l(\nubar)=2$,
  $$Z'_\GW(\cA_n\times \bP^1)^{\circ, \sim}_{\vec\mu, \vec\nu}=Z'_\GW([\bC^2/\bZ_{n+1}]\times \bP^1)^{\circ, \sim}_{\mubar, \nubar};$$
\item When $l(\mubar)+l(\nubar)\geq 3$,
  $$Z'_{\GW,\beta\neq 0}(\cA_n\times \bP^1)^{\circ, \sim}_{\vec\mu, \vec\nu}=Z'_\GW([\bC^2/\bZ_{n+1}]\times \bP^1)^{\circ, \sim}_{\mubar, \nubar}.$$
 \end{enumerate}
\end{theorem}

\begin{proof}
We make the following observation. For any $1\leq l\leq n$ and any formal variable $u$,
$$
1+\sum_{d=1}^\infty(\zeta^le^u)^d=\frac{1}{1-\zeta^le^u}=
-\frac{\sum_{c=0}^n\zeta^{cl}e^{cu}}{e^{(n+1)u}-1}=-\sum_{c=0}^n\zeta^{cl}
\sum_{k=1}^\infty\frac{B_k(\frac{c}{n+1})}{k!}((n+1)u)^{k-1}.
$$
Taking derivatives on both sides for $2g-3+l(\mu)+l(\nu)$ times,
$$
\sum_{d=1}^\infty d^{2g-3+l(\mu)+l(\nu)}(\zeta^le^u)^d=
-\sum_{c=0}^n\zeta^{cl}\sum_{p=0}^\infty \frac{B_{r_1+r_2} \left( \frac{c}{n+1} \right)}{(r_1+r_2)p!}(n+1)^{2g-3+l(\mu)+l(\nu)+p}u^p,
$$
where we used the fact $r_1+r_2=2g-2+l(\mu)+l(\nu)+p$. Substitute this identity into the generating function,
\begin{eqnarray*}
&&Z(x,z,p^1,p^2)^{\circ, \sim}-\sum_{a=1}^{n}\sum_{g\geq 1}(-1)^g z^{2g}I_{2,g}p^1_{m,a}p^2_{m,n+1-a} \\
&=&(t_1+t_2)\sum_{|\mu|=|\nu|=m}\sum_{g\geq 0}z^{2g}
\sum_{1\leq s\leq t\leq n}
\frac{\hat{p}^1_{\mu,s\to t}\hat{p}^2_{\nu,s\to t}}{|\Aut(\mu)| |\Aut(\nu)|} \sum_{d=1}^\infty
d^{2g-3+l(\mu)+l(\nu)}(\zeta^{t-s+1}\exp(y_{s\to t}))^d\\
&&\cdot\sum_{f=0}^{g}\frac{(-1)^f2^f}{f!}
\sum_{M=f}^{g} [z^{2g-2M}] \left( \prod_{i=1}^{l(\mubar)}\cS(\mu_i z)\prod_{j=1}^{l(\nubar)}\cS(\nu_j z) \right) \cdot [z^{2M}] \left( \int \frac{F(z)}{z} \right)^f.
\end{eqnarray*}

A key observation here is that the power series $\int \frac{F(z)}{z}$ starts from the quadratic term (we let the integration constant be zero). Thus we can rewrite the sum $\sum_{M=f}^{g}$ as $\sum_{M=0}^{g}$ and the result does not change. Therefore,
\begin{eqnarray*}
&&Z(x,z,p^1,p^2)^{\circ, \sim}-\sum_{a=1}^{n}\sum_{g\geq 1}(-1)^g z^{2g}I_{2,g}p^1_{m,a}p^2_{m,n+1-a} \\
&=&(t_1+t_2)\sum_{|\mu|=|\nu|=m}
\sum_{1\leq s\leq t\leq n}
\frac{\hat{p}^1_{\mu,s\to t}\hat{p}^2_{\nu,s\to t}}{|\Aut(\mu)| |\Aut(\nu)|} \sum_{d=1}^\infty
d^{-3+l(\mu)+l(\nu)} \\
&& \cdot (\zeta^{t-s+1}\exp(y_{s\to t}))^d \cdot \prod_{i=1}^{l(\mu)}\cS(d\mu_i z) \prod_{j=1}^{l(\nu)}\cS(d\nu_j z) \cdot \exp\left( -2\int \frac{F(dz)}{dz} \right)
\end{eqnarray*}
By Lemma \ref{exp} below, we obtain the following formula:
\begin{eqnarray*}
&&Z(x,z,p^1,p^2)^{\circ, \sim}-\sum_{a=1}^{n}\sum_{g\geq 1}(-1)^g z^{2g}I_{2,g}p^1_{m,a}p^2_{m,n+1-a} \\
&=& \frac{(t_1+t_2)}{|\Aut(\mu)| |\Aut(\nu)|} \sum_{|\mu|=|\nu|=m}
\sum_{1\leq s\leq t\leq n}
\hat{p}^1_{\mu,s\to t}\hat{p}^2_{\nu,s\to t}\\
&& \sum_{d=1}^\infty \frac{d^{l(\mu)+l(\nu)-3}(\zeta^{t-s+1}\exp(y_{s\to t}))^d
\cdot(\prod_{i=1}^{l(\mu)}\cS(d\mu_i z)\prod_{j=1}^{l(\nu)}\cS(d\nu_j z))}{\cS(dz)^2}
\end{eqnarray*}
This coincides with the formula Proposition 3.6 in \cite{Mau}, where the parameters $s_j$ are related to $y_j$ by
$$s_j = \zeta e^{y_j}, \qquad 1\leq j\leq n.$$
Then Theorem \ref{thm:rubber} follows from Lemma \ref{I2}.
\end{proof}

The following lemma computes the exponential term.
\begin{lemma} \label{exp}
  $$\exp\left( \int \frac{F(z)}{z} \right) = \cS(z).$$
\end{lemma}

\begin{proof}
  Both sides take the value $1$ at $z=0$. Thus it makes sense to take logarithms and it suffices to prove
  $$\int \frac{F(z)}{z} = \log \cS(z),$$
  which is clear by checking the derivatives of both sides match with each other.
\end{proof}

\subsection{Crepant Resolution Conjecture for $3$-point functions}

Recall that we have the following rigidification results from Section 3, in the case $\rk(V) = \vdim +1 >0$.
$$\langle \mubar,\nubar,(2,0)(1,0)^{m-2} \rangle^{\cX,\circ}_{g,\gamma} + \sum_{k=1}^n \langle \mubar,\nubar,(1,k)(1,0)^{m-1} \rangle^{\cX,\circ}_{g,\gamma \backslash (k)} = (2g-2+p+l(\mubar)+l(\nubar))\langle \mubar,\nubar\rangle^{\cX,\circ,\sim}_{g,\gamma}.$$
$$\langle \mubar,\nubar,(1,k)(1,0)^{m-1} \rangle^{\cX,\circ}_{g,\gamma} = \langle \mubar, \nubar \rangle^{\cX, \circ, \sim}_{g, \gamma\sqcup (k)}.$$

We can obtain the following result for the \emph{disconnected} theory, by the same argument as in Proposition 4.4 of \cite{Mau}. The partition function for $[\bC^2/\bZ_{n+1}]\times \bP^1$ is defined as
$$Z'_\GW([\bC^2/\bZ_{n+1}]\times \bP^1)_{\mubar^1, \cdots, \mubar^r} := \sum_{\chi, \gamma} z^{-\chi} \frac{x^\gamma}{\gamma !} \langle \mubar^1, \cdots, \mubar^r \rangle^{[\bC^2/\bZ_{n+1}]\times \bP^1, \bullet}_{\chi, \gamma}.$$

\begin{lemma}\label{3ptsdeg0}
	Let $\mubar=(m,k),\nubar=(m,-k),k\neq 0, \rhobar=(2,0)(1,0)^{m-2}$ and let $\vec\mu, \vec\nu, \vec\rho\in \cF_{\cA_n}$ be their correspondents under the isomorphism above. Then
	$$
	Z'_{\GW,\beta=0}(\cA_n\times \bP^1)^\circ_{\vec\mu, \vec\nu, \vec\rho} =z^{-1}\frac{\partial}{\partial z}\left(	\sum_{g\geq 1}(-1)^g z^{2g}I_{2,g}   \right)-z^{-1}\sum_{g\geq 1}(-1)^g z^{2g-1}t_1t_2\langle (1,0)|\tau_1[F]|\rangle^{\cX,\circ}_{g,\emptyset}
	$$
\end{lemma}
\begin{proof}
This follows from Proposition \ref{cotAn}, Lemma \ref{cot}, and \eqref{genI2}.
\end{proof}

\begin{theorem}[GW crepant resolution] \label{GW-CRC}
Given $\mubar, \nubar, \rhobar\in \cF_{[\bC^2/\bZ_{n+1}]}$, with
$$\rhobar = (1,0)^m, \qquad (2,0)(1,0)^{m-2}, \qquad \textrm{or} \qquad (1,k)(1,0)^{m-1},$$
let $\vec\mu, \vec\nu, \vec\rho\in \cF_{\cA_n}$ be their correspondents. Then under the change of variables
  $$s_j = \zeta \exp \left( \frac{1}{n+1} \sum_{a=1}^n (\zeta^{a/2}- \zeta^{-a/2}) \zeta^{ja} x_a \right), \qquad 1\leq j\leq n,$$
we have:
\begin{enumerate}[1)]
\item When $l(\mubar)+ l(\nubar)=2$, and $\rhobar = (1,0)^m$ or $(2,0)(1,0)^{m-2}$,
$$Z'_\GW(\cA_n\times \bP^1)_{\vec\mu, \vec\nu, \vec\rho} = Z'_\GW([\bC^2/\bZ_{n+1}]\times \bP^1)_{\mubar, \nubar, \rhobar};$$
\item When $l(\mubar)+ l(\nubar)\geq 3$, or $\rhobar =(1,k)(1,0)^{m-1}$,
$$Z'_{\GW,\beta\neq 0}(\cA_n\times \bP^1)_{\vec\mu, \vec\nu, \vec\rho} = Z'_\GW([\bC^2/\bZ_{n+1}]\times \bP^1)_{\mubar, \nubar, \rhobar}.$$
\end{enumerate}
\end{theorem}

\begin{proof}
The case $\rhobar = (1,0)^m$ is a direct corollary of results in Section \ref{1,0} and \ref{1,1}. We now concentrate on $\rhobar = (2,0)(1,0)^{m-2}$ or $(1,k)(1,0)^{m-1}$.

Recall that in Section 3 we have classified all invariants in cases: when $\delta = 1$, invariants reduce to the smooth case \cite{Bry-Pan} and can be matched directly; when $\delta = 0$ and $\rk(V) > 0$, which we call the $(\dagger)$ condition,  invariants are all linear in $(t_1 + t_2)$; when $\delta=0$ and $\rk(V) = \vdim = 0$, invariants are constant in $(t_1 + t_2)$. We try to match the latter two parts separately. The rigidification results can be rewritten as equations
$$Z'_\GW([\bC^2/\bZ_{n+1}]\times \bP^1)_{\mubar, \nubar, (1,k)(1,0)^{m-1}}^\circ = z^{-2} \frac{\partial}{\partial x_k} Z'_\GW([\bC^2/\bZ_{n+1}]\times \bP^1)^{\circ, \sim}_{\mubar, \nubar},$$
and for $l(\mubar)+ l(\nubar)=2$
\begin{eqnarray*}
&&z^{l(\mu) +l(\nu)-1}Z'_\GW([\bC^2/\bZ_{n+1}]\times \bP^1)_{\mubar, \nubar, (2,0)(1,0)^{m-2}}^\circ \\
&=&
\frac{\partial}{\partial z}
\left( z^{l(\mu) + l(\nu) -2} Z'_\GW([\bC^2/\bZ_{n+1}]\times \bP^1)^{\circ, \sim}_{\mubar, \nubar} \right)-\sum_{g\geq 1}(-1)^g z^{2g-1}t_1t_2\langle (1,0)|\tau_1[F]|\rangle^{\cX,\circ}_{g,\emptyset},
\end{eqnarray*}
and for $l(\mubar)+ l(\nubar)\geq 3$
$$
z^{l(\mu) +l(\nu)-1}Z'_\GW([\bC^2/\bZ_{n+1}]\times \bP^1)_{\mubar, \nubar, (2,0)(1,0)^{m-2}}^\circ=\frac{\partial}{\partial z}
\left( z^{l(\mu) + l(\nu) -2} Z'_\GW([\bC^2/\bZ_{n+1}]\times \bP^1)^{\circ, \sim}_{\mubar, \nubar} \right).
$$
and the disconnected version is also true. On the other hand, by the change of variables
$$\frac{\partial}{\partial x_k} = \frac{\zeta^{k/2} - \zeta^{-k/2}}{n+1} \sum_{j=1}^n \zeta^{jk} s_k \frac{\partial}{\partial s_k}.$$
Compared with Proposition 4.4 of \cite{Mau}, we conclude by Theorem \ref{thm:rubber} and by Lemma \ref{3ptsdeg0} that the $(t_1 + t_2)$-linear terms of $Z'_\GW(\cA_n\times \bP^1)_{\vec\mu, \vec\nu, \vec\rho}$ and $Z'_\GW([\bC^2/\bZ_{n+1}] \times \bP^1)_{\mubar, \nubar, \rhobar}$ coincide.

It suffices to match the constant terms, which are contributed by invariants satisfying $\vdim = \rk (V) = 0$. As discussed in Part b) of Section \ref{rubber-section}, the followings are the only three possibilities for this to give nontrivial connected invariants: $\mubar = (m_1, 0)(m_2, k)$, $\nubar = (m, -k)$; $\mubar = (m_1, k)(m_2, -k)$, $\nubar = (m,0)$; $\mubar = (m,0)$, $\nubar = (m,-k)$. Here $m_1 + m_2 = m$, $k\neq 0$, and $\mubar$, $\nubar$ could be switched, and $\rhobar = (2,0)(1,0)^{m-2}$ for the first two and $\rhobar = (1,k)(1,0)^{m-1}$ for the third possibility. In all these cases we have $g=0$ and
$$\langle \mubar, \nubar, \rhobar \rangle_{0, \gamma = \emptyset}^{\cX, \circ} = \langle \mu, \nu, \rho \rangle_0^{\bP^1, \circ} \cdot (e_k, e_{-k})_{[\bC^2 / \bZ_{n+1}]} ,$$
where $\langle \ \rangle^{\bP^1}$ is the relative GW invariant for $\bP^1$, and $(-,-)_{[\bC^2 / \bZ_{n+1}]}$ is the orbifold Poincar\'e pairing.

On the other hand, consider the corresponding $\vec\mu$, $\vec\nu$ for those cases, and the theory $\langle \vec\mu, \vec\nu, \vec\rho \rangle_{g, \beta}^{\cA_n\times \bP^1, \circ}$, where $\beta \in H_2(\cA_n, \bZ)$. When $\beta\neq 0$, Proposition 4.3 of \cite{Mau} implies that these invariants always vanish. When $\beta = 0$, it is easy to see, by dimensional constraints, that we must have $g=0$. Then we have
$$\langle \vec\mu, \vec\nu, \vec\rho \rangle_{0, \beta=0}^{\cA_n\times \bP^1, \circ} = \langle \mu, \nu, \rho \rangle_0^{\bP^1, \circ} \cdot (\Phi(e_k), \Phi(e_{-k}) )_{\cA_n},$$
where $\Phi$ is the isomorphism (\ref{correspondence}), and $(-, -)_{\cA_n}$ is the Poincar\'e pairing on $\cA_n$. Finally, one can observe that the isomorphism $\Phi$ actually preserves the Poincar\'e pairing, and therefore 3-point functions in the three exceptional cases also match.
\end{proof}

\begin{remark}
The statements of the GW crepant resolution in Theorem \ref{GW-CRC} fall into two cases 1) and 2), where in Case 2), the $\beta = 0$ part of the partition function for $\cA_n \times \bP^1$ does not match with any GW invariants of $[\bC^1 /\bZ_{n+1}] \times \bP^1$. The reason for this discrepancy is the following.

First, by dimension counting, one can see that $Z'_{GW, \beta = 0} (\cA_n \times \bP^1)_{\vec\mu, \vec\nu, \vec\rho}$ is only possibly nonzero for the $g=0$ terms. Similar analysis as above shows that the invariant $\langle \vec\mu, \vec\nu, \vec\rho \rangle_{0, \beta=0}^{\cA_n\times \bP^1, \circ}$ reduces to
$$
\langle \mu, \nu, \rho \rangle_0^{\bP^1, \circ} \cdot \langle \Phi (e_{i_1}), \cdots, \Phi(e_{l(\mu) + l(\nu) + l(\rho)}) \rangle^{\cA_n},
$$
where the $\Phi (e_{i_l}) \in H^*(\cA_n)$ are all cohomology classes carried by $\vec\mu$, $\vec\nu$, $\vec\rho$, and $\langle - \rangle^{\cA_n}$ is the classical $n$-point function on $\cA_n$. This $n$-point function is not zero in general in Case 2).
\end{remark}

\subsection{Proof of Lemma \ref{cot}} \label{pflemma}

In this subsection, we prove Lemma \ref{cot}. Recall that $\langle (1,0)|\tau_1[F]|\rangle^{\cX,\circ}_{g,\emptyset}$ is an integral over the moduli space $\Mbar_{g, (\one)}(\cY, (1,0))$. The integral is nonzero only if $\delta=1$ which means that the monodromies over the $2g$ nontrivial loops on the domain curve are all trivial. Let $\pi:\Mbar_{g, (\one)}(\cY, (1,0))\to \Mbar_{g, 1}(\bP^1, (1))$ be the forgetful map which forgets the orbifold structure and let $\Mbar\subset \Mbar_{g, (\one)}(\cY, (1,0))$ be the locus where the monodromies over the $2g$ nontrivial loops on the domain curve are all trivial. Then the degree of $\pi|_{\Mbar}:\Mbar\to \Mbar_{g, 1}(\bP^1, (1))$ is equal to $\frac{1}{n+1}$. Therefore we have
$$
\langle (1,0)|\tau_1[F]|\rangle^{\bC^2\times\bP^1,\circ}_{g,\emptyset}=\frac{1}{n+1}\langle (1)|\tau_1[F]|\rangle^{\bC^2\times\bP^1,\circ}_{g}.
$$
Let the base $\bP^1$ degenerate into two components, such that the relative marked point and the fiber insertion lies on the same component. Degeneration formula implies
$$
\langle (1)|\tau_1[F]|\rangle^{\bC^2\times\bP^1,\circ}_{g}=\sum_{g_1+g_2=g}\langle (1)|\tau_1[F]|(1)\rangle^{\bC^2\times\bP^1,\circ}_{g_1}t_1t_2\langle (1)|\rangle^{\bC^2\times\bP^1,\circ}_{g_2}.
$$
By dimensional constraint, it is easy to see that $g_2=0$ and the degeneration formula reduces to
$$
\langle (1)|\tau_1[F]|\rangle^{\bC^2\times\bP^1,\circ}_{g}=\langle (1)|\tau_1[F]|(1)\rangle^{\bC^2\times\bP^1,\circ}_{g}.
$$
By rigidification, we have
$$
\langle (1)|\tau_1[F]|(1)\rangle^{\bC^2\times\bP^1,\circ}_{g}=\langle (1)|\tau_1[\one]|(1)\rangle^{\bC^2\times\bP^1,\circ,\sim}_{g}=2g\langle (1),(1)\rangle^{\bC^2\times\bP^1,\circ,\sim}_{g}.
$$
By the definition of the double ramification cycle, we have
$$
\langle (1),(1)\rangle^{\bC^2\times\bP^1,\circ,\sim}_{g}=(-1)^{2g-1}\frac{(t_1+t_2)}{t_1t_2}\int_{\Mbar_{g,2}}\DR\lambda_g\lambda_{g-1}
=\int_{[\Mbar_{g,2}(\bC^2)]^\vir}(\pi')^*\DR,
$$
where $\pi':\Mbar_{g,2}(\bC^2)\to\Mbar_{g,2}$ is the forgetful map. By the discussion of the double ramification cycle in Section \ref{pixton}, we have
$$
\int_{[\Mbar_{g,2}(\bC^2)]^\vir}(\pi')^*\DR\\
=\int_{[\Mbar_{g,2}(\bC^2)]^\vir}\frac{(\psi_1+\psi_2)^g}{2^gg!} =\int_{[\Mbar_{g,2}(\bC^2)]^\vir}2^{-g}\sum_{k+l=g}\frac{\psi_1^k}{k!}\frac{\psi_2^l}{l!}.
$$
By Grothendieck-Riemann-Roch computations or by \eqref{jianzhou}, we have
\begin{eqnarray*}
\int_{[\Mbar_{g,2}(\bC^2)]^\vir}2^{-g}\sum_{k+l=g}\frac{\psi_1^k}{k!}\frac{\psi_2^l}{l!}
&=&(-1)^{g}\frac{(t_1+t_2)}{t_1t_2}2^{-g} \Big( \frac{1}{2}\sum_{I\sqcup J=[2]}\frac{B_{2g}}{2g}(-1)^{|J|}\sum_{k+l=g}
\frac{1}{4^g(2k+1)!!(2l+1)!!k!l!}\\
&&+ \frac{1}{2}\frac{B_{2g}}{2g}\sum_{k+l=g}
\frac{1}{4^{g-1}(2k-1)!!(2l-1)!!k!l!}  \Big)\\
&=&(-1)^{g}\frac{(t_1+t_2)}{t_1t_2}\Big(\frac{1}{2}\sum_{I\sqcup J=[2]}\frac{B_{2g}}{2g}(-1)^{|J|}\sum_{k+l=g}
\frac{1}{4^g(2k+1)!(2l+1)!}\\
&&+ \frac{1}{2}\frac{B_{2g}}{2g}\sum_{k+l=g}
\frac{1}{4^{g-1}(2k)!(2l)!}  \Big).
\end{eqnarray*}
Notice that the first term in the bracket vanishes. Therefore
$$
\int_{[\Mbar_{g,2}(\bC^2)]^\vir}2^{-g}\sum_{k+l=g}\frac{\psi_1^k}{k!}\frac{\psi_2^l}{l!} = (-1)^{g}\frac{(t_1+t_2)}{t_1t_2}\frac{B_{2g}}{2g(2g)!}.
$$

Notice that for any formal variable $u$, $\frac{1}{e^u-1}=\sum_{k=0}^{\infty}\frac{B_{k}}{k!}u^{k-1}$. Let
$$
G(z) := \int\left(\frac{1}{e^z-1}-\frac{1}{z}\right)dz=\log\frac{e^z-1}{e^z}-\log z,
$$
where we let the integration constant be zero. Then
\begin{eqnarray*}
\sum_{g\geq 1}(-1)^g z^{2g}\langle (1),(1)\rangle^{\bC^2\times\bP^1,\circ,\sim}_{g} &=& \frac{(t_1+t_2)}{t_1t_2}\cdot\frac{G(z)+G(-z)}{2}  \\
	&=& \frac{(t_1+t_2)}{2t_1t_2}\left(\log (e^{z}-1)+\log(e^{-z}-1)-\log z-\log (-z)\right).
\end{eqnarray*}
Therefore
\begin{eqnarray*}
\sum_{g\geq 1}(-1)^g z^{2g-1}\langle (1,0)|\tau_1[F]|\rangle^{\cX,\circ}_{g,\emptyset} &=& \frac{1}{n+1}
\frac{\partial}{\partial z} \left( \sum_{g\geq 1}(-1)^g z^{2g}\langle (1),(1)\rangle^{\bC^2\times\bP^1,\circ,\sim}_{g}\right) \\
&=& \frac{(t_1+t_2)}{2(n+1)t_1t_2}\left(\frac{\cosh \frac{z}{2}}{\sinh \frac{z}{2}}-\frac{2}{z}\right).
\end{eqnarray*}

\section{Orbifold quantum cohomology of symmetric products}

As a generalization of \cite{Bry-Gra}, there is another theory in connection with our picture, the orbifold quantum cohomology of the symmetric products: $\Sym ([\bC^2/\bZ_{n+1}])$ and $\Sym (\cA_n)$.

Let $X$ be a scheme, and $\Sym^m (X):= [X^m/S_m]$ be its $m$-th symmetric product. We would like to consider the orbifold GW theory of $\Sym^m (X)$. Let $f: \cC\rightarrow \Sym^m (X)$ be a stable map. Following K. Costello \cite{Cos}, $f$ is equivalent to certain \'etale cover $\cC' \to \cC$ of degree $m$, together with a stable map from $\cC'$ to $X$.

Connected components of the inertia stack $I \Sym^m (X)$ are indexed by partitions $\lambda$ of $m$. For a partition $\lambda$, which corresponds to a conjugacy class of $S_m$, the associated component is
\begin{equation} \label{twisted-comp}
\left[ (X^m)^\lambda \middle/ \Aut(\lambda) \right] = X^{l(\lambda)} \times B \Aut (\lambda),
\end{equation}
where $(X^m)^\lambda$ is the fixed loci in $X^m$ under the action of elements in the conjugacy class, $\Aut(\lambda)$ is the stabilizer.

The state space for $\Sym^m (X)$ is its orbifold cohomology $H^*_\orb(\Sym^m (X))$, which is the cohomology of $I \Sym^m (X)$, with some degree shift. It is a classical result that the super graded vector space $\bigoplus_{m\geq 0} H^*_\orb(\Sym^m (X))$ can be realized as an irreducible highest weight representation of the super-Heisenberg algebra associated with $H^*(X)$. For a reference, see for example Section 5.2 in \cite{Ade-Lei-Rua}. As a result, $H^*_\orb(\Sym^m (X))$ has a basis indexed by $H^*(X)$-weighted partitions of $m$.

In the case when $X$ is a surface, there is an isomorphism
$$H^*_\orb(\Sym^m (X)) \cong H^*(\Hilb^m (X)),$$
which respects the gradings and Poincar\'e pairings. Moreover, if $X$ has trivial canonical bundle, it also preserves the (orbifold) cup product.

\subsection{$\Sym ([\bC^2/\bZ_{n+1}])$}

The description above for $I\Sym^m X$ where $X$ is a scheme does not apply to orbifolds. However, for our special target $[\bC^2 / \bZ_{n+1}]$, the symmetric product is easy to describe combinatorially. Precisely, the $m$-th symmetric product of $[\bC^2 / \bZ_{n+1}]$ is defined as
$$\Sym^m ([\bC^2/\bZ_{n+1}]) \ :=\  [[\bC^2/\bZ_{n+1}]^m/S_m] \ =\  [\bC^{2m} / (\bZ_{n+1} \wr S_m)],$$
where $\bZ_{n+1} \wr S_m := (\bZ_{n+1})^m \rtimes S_m$ is the wreath product.

Denote $G:= \bZ_{n+1} \wr S_m$. Stable maps into $[\bC^{2m}/G]$ are the same as those mapping into the origin $BG$. Evaluation maps land in the orbifold cohomology of $BG$, which is indexed by conjugacy classes of $G$, or in other words, indexed by $\bZ_{n+1}$-weighted partitions of $m$.

The age of a component indexed by such a partition $\lambdabar$ is $m-l'(\lambdabar)$, where $l'$ denotes the number of parts with trivial decoration. Components with age $1$ are exactly
$$\rhobar = (2,0)(1,0)^{m-2}, \qquad \textrm{or} \qquad (1,k)(1,0)^{m-1}, \qquad k\neq 0\in \bZ_{n+1}.$$

We aim to compute the 2-point functions, whose moduli space is
$$\Mbar_{0, (\mubar, \nubar; b, \gamma)} (\Sym^m ([\bC^2/\bZ_{n+1}])).$$
Here the datum $(\mubar, \nubar; b, \gamma)$ specifies the monodromies around marked points on the domain curve, in which $\mubar, \nubar$ are treated as marked points with insertions from the target, and $b, \gamma$ are extra marked points treated as the analog of ``degree class". More precisely, $\mubar$ and $\nubar$ are two $\bZ_{n+1}$-weighted partitions, indicating monodromies at two of the marked points; $\gamma=(\gamma_1, \cdots, \gamma_{l(\gamma)}) \in \bZ_{n+1}$ records extra marked points with monodromies $(1,\gamma_i)(1,0)^{m-1}$, $\gamma_i\neq 0$, $i=1,\cdots,l(\gamma)$; $b\geq 0$ is an integer, recording the number of extra \emph{unordered} marked points with monodromies $(2,0)(1,0)^{m-1}$. The reason why we treat these two types of extra markings differently is that we would like to count only those coming from the $\Sym$ operation as unordered.

In other words, let $\bar {D}_{\mubar}, \bar {D}_{\nubar}, \bar D_0, \cdots, \bar D_n$ be the connected components in the rigidified inertial stack $\bar I \Sym^m ([\bC^2 / \bZ_{n+1}])$ associated with partitions $\mubar,\nubar,(2,0)(1,0)^{m-2}$, and $(1,k)(1,0)^{m-1}, k=1,\cdots , n$ respectively. We define
$$\Mbar_{0, (\mubar, \nubar; b, \gamma)} (\Sym^m ([\bC^2/\bZ_{n+1}])) := [\Mbar / S_b],$$
where $\Mbar$ is the fiber product in the following Cartesian diagram
$$\xymatrix{
\Mbar \ar[r] \ar[d] & \bar {D}_{\mubar} \times \bar {D}_{\nubar} \times \bar D_0 \times \cdots \bar D_n \ar[d] \\
\Mbar_{0, 2 +b +l(\gamma)} (\Sym^m ([\bC^2/\bZ_{n+1}])) \ar[r] & \left( \bar I \Sym^m ([\bC^2/\bZ_{n+1}]) \right)^{ 2 +b +l(\gamma) }.
}$$

The 2-point functions are defined as
$$\langle \mubar, \nubar \rangle_{\Sym^m([\bC^2/\bZ_{n+1}])}:= \sum_{b, \gamma} z^b \frac{x^\gamma}{\gamma!} \int_{\left[ \Mbar_{0, (\mubar, \nubar; b, \gamma)} (\Sym^m ([\bC^2/\bZ_{n+1}])) \right]^\vir} 1,$$
where the integration is defined by $T$-localization. Similarly, we can define $r$-point functions and the following equations directly follow from the definition
$$\langle \mubar, \nubar, (2,0)(1,0)^{m-2} \rangle_{\Sym^m ([\bC^2 / \bZ_{n+1}])} = \frac{\partial}{\partial z} \langle \mubar, \nubar \rangle_{\Sym^m ([\bC^2 / \bZ_{n+1}])},$$
$$\langle \mubar, \nubar, (1,k)(1,0)^{m-1} \rangle_{\Sym^m ([\bC^2 / \bZ_{n+1}])} =  \frac{\partial}{\partial x_k} \langle \mubar, \nubar \rangle_{\Sym^m ([\bC^2 / \bZ_{n+1}])}.$$

We now apply Costello's construction (e.g. Lemma 2.2.1 in \cite{Cos}). Let $f: \cC \to \Sym^m ([\bC^2/\bZ_{n+1}])$ be a stable map, representing a geometric point in the moduli space. By definition this is equivalent to a principal $S_m$-bundle $\cP \to \cC$, together with an $S_m$-equivariant map $\cP \to [\bC^2 / \bZ_{n+1}]^m$, which is representable since $f$ is. Taking $\cC':= \cP \times_{S_m} \{1, \cdots, m\}$, which is possibly disconnected, we obtain a diagram
\begin{equation} \label{point-U_2}
\xymatrix{
\cC' \ar[d]_\pi \ar[r]^-{f'} & [\bC^2 / \bZ_{n+1}] \\
\cC,
}
\end{equation}
where $\pi$ is an \'etale covering. Note that $f'$ is representable, because the quotient by $S_m$ is free. The moduli space $\Mbar_{0, (\mubar, \nubar; b, \gamma)} (\Sym^m ([\bC^2/\bZ_{n+1}]))$ is then isomorphic to the moduli space of such \'etale coverings.

Let $\underline\pi: \underline{\cC'} \to \underline\cC$ be the induced map between coarse moduli spaces, which is a branched covering. The ramification profile is completely determined by the topological datum $(\mubar, \nubar; b, \gamma)$. For example, suppose $0, \infty \in \underline\cC$ are images of the marked points in $\cC$ associated with $\mubar, \nubar$. Then $\underline\pi^{-1}(0)$ consists of $l(\mu)$ points, with ramification degrees $\mu_1, \cdots, \mu_{l(\mu)}$, and monodromies given by decorations of $\mubar$. Similar for $\infty \in \underline\cC$. Moreover, there are $b$ branched points on $\underline\cC$ over which the ramification profiles are specified by $(2,0)(1,0)^{m-2}$, and $\gamma_k$ points on $\underline\cC$ over which the ramification profiles are specified by $(1,k)(1,0)^{m-1}$. Finally, the genus $g = g(\cC')$ can be computed via Riemann--Hurwitz:
$$
b = 2g-2 + l(\mu) + l(\nu).
$$

Replacing $[\bC^2 / \bZ_{n+1}]$ with $B\bZ_{n+1}$, there is a similar description of $\Mbar_{0, (\mubar, \nubar; b, \gamma)} ( \Sym^m (B\bZ_{n+1}) )$, where objects are diagrams as above with $f'$ mapping into $B\bZ_{n+1}$. Consider the Hurwitz--Hodge bundle $V$ associated with the $\bZ_{n+1}$-representation $\bC^2$, whose fibers are $H^1 (\cC, \pi_* (\cO_{\cC'} \otimes \bC^2))$. In the case $\rk (V) >0$, this is the obstruction bundle, and by the same argument as in previous sections,
\begin{eqnarray*}
	\int_{\left[ \Mbar_{0, (\mubar, \nubar; b, \gamma)} (\Sym^m ([\bC^2/\bZ_{n+1}])) \right]^\vir} 1 &=& \int_{\left[ \Mbar_{0, (\mubar, \nubar; b, \gamma)} (\Sym^m ( B\bZ_{n+1} )) \right]^\vir} e_T(V) \\
	&=& (t_1 + t_2) \int_{\left[ \Mbar_{0, (\mubar, \nubar; b, \gamma)} (\Sym^m ( B\bZ_{n+1} )) \right]^\vir} c_{\rk(V) - 1}(V).
\end{eqnarray*}

We now make a key observation: over certain open substacks of the moduli space, the perfect obstruction theories on $\Mbar_{0, (\mubar, \nubar; b, \gamma)} (\Sym^m ( B\bZ_{n+1} ))$ and $\Mbar^\bullet_{g, \gamma} ( B\bZ_{n+1} \times \bP^1, \mubar, \nubar)^\sim$ are the same, and the Chern class of obstruction bundles vanish on the complements of those open substacks.

More precisely, this means the following. Let $\cU_1 \subset \Mbar^\bullet_{g, \gamma} ( B\bZ_{n+1} \times \bP^1, \mubar, \nubar)^\sim$ be the open substack of relative stable maps with \emph{smooth domains and no contracted components}. Let $\cU_2 \subset \Mbar_{0, (\mubar, \nubar; b, \gamma)} (\Sym^m ( B\bZ_{n+1} ))$ be the open substack where the $\cC'$ as in diagram (\ref{point-U_2}) are \emph{smooth}.

\begin{lemma} \label{equiv-open}
	There is an equivalence of stacks
	$$
	\xymatrix{
		\Psi: \cU_2 \ar[r]^-\sim & \left[ \cU_1 \middle/ \Aut (\mu) \times \Aut(\nu) \right],
	}
	$$
	preserving the obstruction theories and obstruction bundles.
\end{lemma}

\begin{proof}
	For a scheme $S$, consider an object $(f', \pi)$ of $\cU_2$ over $S$, represented by a diagram as (\ref{point-U_2}), i.e. $\cC'$ and $\cC$ are families of twisted curves over $S$, and $f': \cC' \to S\times [\bC^2 / \bZ_{n+1}]$ is a family of prestable maps. Since the fibers of $\cC'$ are required to have irreducible connected components, the fibers of $\cC$ has to be irreducible, whose coarse moduli space has to be $\bP^1$. Let $\overline\cC$ be the coarse moduli space. The markings $\mubar$ and $\nubar$ induce sections of $\overline\cC$ over $S$. Let $\fM_{0,2}^\circ \subset \fM_{0,2}$ be the open substack consisting curves with irreducible domains. Then $\overline\cC$ defines a map $S \to \fM_{0,2}^\circ$.
	
	Let $\fT^\sim$ the classifying stack of expanded pairs with respect to the pair $(\bP^1,  \{0\} , \{\infty\})$, with non-rigid target, and let $\fY^\sim$ be the universal family.  By the construction of Graber--Vakil \cite{GV}, the open substack $\fM_{0,2}^{ss} \subset \fM_{0,2}$ of semistable curves, which contains $\fM_{0,2}^\circ$, is actually isomorphic to $\fT^\sim$. Moreover, for our special target $(\bP^1,  \{0\} , \{\infty\})$, the universal families over $\fM_{0,2}^{ss}$ and $\fT^\sim$ are also isomorphic. Therefore, the map $S\to \fM_{0,2}^\circ$ gives a family of relative stable maps to the non-rigid target
	$$
	\xymatrix{
	\overline\cC \ar[dr] \ar[r] & \cY_S \ar[d] \ar[r] & \fY^\sim \ar[d] \\
	& S \ar[r] &  \fT^\sim .
}
	$$
	
	Let $\bar\pi: \cC' \to \overline\cC$ be the composition of $\pi$ with the coarse moduli. Together with $f'$, this defines a map $(f', \bar\pi): \cC' \to  B\bZ_{n+1} \times \cY_S$, which is representable since $f'$ is, and stable since $\pi$ is \'etale. We therefore define it to be the image of the object $(f, \pi)$ under $\Psi$. The image has to be considered up to the action of $\Aut(\mu) \times \Aut(\nu)$ since there is no ordering for marked points that come from the ramification points of the covering.
	
	The inverse can be constructed easily. The families of expanded pairs with respect to the pair $(B\bZ_{n+1} \times \bP^1, B\bZ_{n+1} \times \{0\}, B\bZ_{n+1} \times \{\infty\})$ are exactly those obtained from $\cY_S \to S$ by base change over $B\bZ_{n+1}$. Then, given an object $\bar f': \cC' \to B\bZ_{n+1} \times \cY_S$ in $\cU_1(
	S)$, we have the family of branched covering $\cC' \to \cY_S$. Applying the root construction, one can always add orbifold structures on the branch and ramification divisors, to make the modified map \'etale, hence forming an object in $\cU_2$. One can check that this is the inverse to $\Psi$.
	
The compatibility of obstruction theories follows from the fact that they are both unobstructed over the chosen open stratum, and hence both stacks are smooth. The compatibility of obstruction bundles follows from the fact that $\pi: \cC' \to \cC$ is finite \'etale, and hence  $H^1 (\cC, \pi_* (\cO_{\cC'} \otimes \bC^2) ) = H^1 (\cC', \cO_{\cC'} \otimes \bC^2)$.
\end{proof}

Therefore, we now have the following diagram:
$$\xymatrix{
\cU_1 \ar[d]_-i \ar@{^{(}->}[r]^-j & \Mbar^\bullet_{g, \gamma} (B\bZ_{n+1}\times \bP^1, \mubar, \nubar)^\sim \ar[d]^-q \\
\Mbar_{0, (\mubar, \nubar; b, \gamma)} (\Sym^m ( B\bZ_{n+1} )) \ar[r]^-p  & \Mbar^\bullet_{g, \mubar\sqcup \nubar\sqcup \gamma} (B\bZ_{n+1}),
}$$
where $j$ is an open embedding and $i$ is \'etale, whose image is the open substack $\cU_2$.

\begin{lemma} \label{Sym-rubber}
Let $\cU_3 \subset \Mbar^\bullet_{g, \mubar\sqcup \nubar\sqcup \gamma} (B\bZ_{n+1})$ be the open substack of curves whose connected components are irreducible. Then the class
$$p_* \left[ \Mbar_{0, (\mubar, \nubar; b, \gamma)} (\Sym^m ( B\bZ_{n+1} )) \right]^\vir- \frac{q_* \left[ \Mbar^\bullet_{g, \gamma} (B\bZ_{n+1}\times \bP^1, \mubar, \nubar)^\sim \right]^\vir}{|\Aut(\mu)| |\Aut(\nu)|} $$
is supported on the complement $\Mbar^\bullet_{g, \mubar\sqcup \nubar\sqcup \gamma} (B\bZ_{n+1}) \backslash \cU_3$.
\end{lemma}

\begin{proof}
For simplicity, we denote $\Mbar_1 := \Mbar^\bullet_{g, \gamma} (B\bZ_{n+1}\times \bP^1, \mubar, \nubar)^\sim$ and $\Mbar_2 = \Mbar_{0, (\mubar, \nubar; b, \gamma)} (\Sym^m ( B\bZ_{n+1} )) $. Let $\tilde j : \tilde\cU_1 \subset \Mbar_1$ be the open substack of relative stable maps with \emph{no contracted irreducible components}. Then $\cU_1 \subset \tilde\cU_1$ dense open, and $q^{-1} (\cU_3) \subset \tilde\cU_1$.

Consider the excision exact sequence $A_* (\Mbar_1 \backslash \tilde\cU_1) \to A_* (\Mbar_1) \xrightarrow{\tilde j^*} A_* (\tilde\cU_1) \to 0$.  Since obstruction theories are compatible, and $\tilde\cU_1$ is unobstructed, we have $\tilde j^* [\Mbar_1]^\vir = [\tilde\cU_1]^\vir = [\tilde\cU_1]$. On the other hand, $\cU_1 \subset \tilde\cU_1$ is dense; so they have the same closure in $\Mbar_1$, which we denote by $\bar\cU_1$. Hence $[\Mbar_1]^\vir - [\bar\cU_1] \in A_* (\Mbar_1 \backslash \tilde\cU_1)$. Moreover, $q$ maps $\Mbar_1 \backslash \tilde\cU_1$ into the complement of $\cU_3$, we have $q_* [ \Mbar_1]^\vir - q_* [\bar\cU_1]$ is supported on $\Mbar_{g, \mubar\sqcup \nubar\sqcup \gamma} (B\bZ_{n+1}) \backslash \cU_3$. The same is true for $p_* [\Mbar_2]^\vir - p_* [ \bar\cU_2 ]$, since actually $\bar\cU_2 = \Mbar_2$ and $\Mbar_2$ is unobstructed. Finally, looking at generic points, by Lemma \ref{equiv-open}, we have $q_* [\bar\cU_1] - |\Aut(\mu)| |\Aut(\nu)| \cdot p_* [\bar\cU_2]$ is supported on the complement of $\cU_3$. Hence the lemma holds.
\end{proof}

Recall that the vanishing result Proposition \ref{vanishing} states that $c_{r_1+r_2-1} (V_1\oplus V_2)$, as a class on $\Mbar^\bullet_{g, \mubar \sqcup \nubar \sqcup \gamma} (B\bZ_{n+1})$, vanishes on the complement of $\cU_3$. Hence Lemma \ref{Sym-rubber} establish an identity between 2-point functions of $\Sym^m ([\bC^2 / \bZ_{n+1}])$ and rubber GW invariants of $[\bC^2 / \bZ_{n+1}] \times \bP^1$. Passing back to the disconnected theory, we obtain the following GW/Sym correspondence.

\begin{theorem}[GW/Sym correspondence] \label{GW-Sym}
  Given $\mubar, \nubar, \rhobar\in \cF_{[\bC^2/\bZ_{n+1}]}$, with
  $$\rhobar = (1,0)^m, \qquad (2,0)(1,0)^{m-2}, \qquad \textrm{or} \qquad (1,k)(1,0)^{m-1},$$
  where $k\neq 0$, we have
  $$z^{l(\mu)+l(\nu)+l(\rho)-m} Z'_\GW([\bC^2/\bZ_{n+1}]\times \bP^1)_{\mubar, \nubar, \rhobar} =  \langle \mubar, \nubar, \rhobar \rangle_{\Sym^m ([\bC^2/\bZ_{n+1}])},$$
  where the right hand side is the 3-point genus-zero orbifold GW invariants of $\Sym^m ([\bC^2/\bZ_{n+1}])$.
\end{theorem}

\begin{proof}
As explained above, in case $\rk(V) = r_1 + r_2 > 0$, 2-point invariants $\langle \mubar, \nubar \rangle_{\Sym^m([\bC^2/\bZ_{n+1}])}$ are exactly the same as as 2-point rubber invariants for $[\bC^2 / \bZ_{n+1}] \times \bP^1$. In other words,
$$\langle \mubar, \nubar \rangle_{\Sym^m([\bC^2/\bZ_{n+1}])} \qquad \text{and} \qquad z^{l(\mu) + l(\nu) - 2} Z'_\GW ([\bC^2 / \bZ_{n+1}] \times \bP^1 )_{\mubar, \nubar}^\sim$$
have the same $(t_1+t_2)$-linear parts.

For the constant terms, i.e. the $\rk(V) = \vdim = 0$ case, the 3-point functions on $\Sym^m([\bC^2/\bZ_{n+1}])$ reduce to the ordinary orbifold cup product. One can check the identity by direct computation.
\end{proof}

\subsection{$\Sym (\cA_n)$}

As mentioned before, the orbifold cohomology of $\Sym^m (\cA_n)$ are parameterized by $H^*(\cA_n)$-weighted partitions of $m$. An insertion of partition $\vec\lambda$ at a marked point $x\in \cC$ indicates that the evaluation map lands into the twisted sector associated with $\lambda$. The age of such a component is $m-l(\lambda)$. Hence a basis of $H^2_\orb(\Sym^m (\cA_n))$ can be taken as the following,
$$(2,1)(1,1)^{m-2}, \qquad (1,\omega_i)(1,1)^{m-1}, \qquad 1\leq i\leq n,$$
where $\omega_1, \cdots, \omega_n \in H^2(\cA_n, \bQ)$ is the dual basis to the exceptional curves in $\cA_n$.

The 3-point function of this theory has been computed by Cheong--Gholampour \cite{Che-Gho}. By definition the $r$-point function is
$$\langle \vec\mu^1, \cdots, \vec\mu^r \rangle_{\Sym^m(\cA_n), (b, \beta)}:= \int_{\left[ \Mbar_{0,r} (\Sym^m(\cA_n), (b, \beta)) \right]^\vir} \prod_{i=1}^r \ev_i^* \vec\mu^i,$$
$$\langle \vec\mu^1, \cdots, \vec\mu^r \rangle_{\Sym^m(\cA_n)}:= \sum_{b, \beta} z^b  s_1^{(\beta, \omega_1)} \cdots s_n^{(\beta, \omega_n)} \langle \vec\mu^1, \cdots, \vec\mu^r \rangle_{\Sym^m(\cA_n), (b, \beta)},$$
where the integer $b>0$ here stands for the number of \emph{unordered} extra marked points with monodromies $(2,1)(1,1)^{m-2}$.

\begin{theorem}[GW/Sym correspondence for $\cA_n$, Theorem 0.3 in \cite{Che-Gho}]
  Given $\vec\mu, \vec\nu, \vec\rho\in \cF_{\cA_n}$, with
  $$\vec\rho = (1,1)^m, \qquad (2,1)(1,1)^{m-2}, \qquad \textrm{or} \qquad (1,\omega_k)(1,1)^{m-1},$$
  we have
  $$z^{l(\mu)+l(\nu)+l(\rho)-m} Z'_\GW(\cA_n\times \bP^1)_{\vec\mu, \vec\nu, \vec\rho} =  \langle \vec\mu, \vec\nu, \vec\rho \rangle_{\Sym^m (\cA_n)},$$
  where the right hand side is the 3-point genus-zero orbifold GW invariants of $\Sym^m (\cA_n)$.
\end{theorem}

\section{GW/DT/Hilb/Sym correspondence}

\subsection{Relative DT theory and GW/DT correspondence of $[\bC^2/\bZ_{n+1}]\times \bP^1$}

Recall the setting of Section \ref{relative-GW}. Let $\cX = [\bC^2/\bZ_{n+1}]\times \bP^1$ be the target and $z_1, \cdots, z_r$ be $r$ points on $\cY = B\bZ_{n+1}\times \bP^1$. Consider the DT theory for $\cX$, relative to the fibers $[\bC^2/\bZ_{n+1}]\times \{z_i\}$, $i=1, \cdots, r$. For the detailed definition of relative DT theory we refer the readers to \cite{Zhou3,Zhou2}. The moduli space of relative DT theory is the relative Hilbert stack
$$\Hilb^{m,\varepsilon}(\cX[k], \vec\mu^1, \cdots, \vec\mu^r),$$
parameterizing 1-dimensional compactly supported closed substacks $\cZ$ in the modified ``bubbled" target $\cX[k]$, for all $k\geq 0$. The stacky curves $\cZ$ are required to satisfy some transversality and stability conditions.

Here $m\geq 0$ is an integer, and $\varepsilon = (\varepsilon_0, \cdots, \varepsilon_n) \in \bZ^{n+1}$. The pair $(m,\varepsilon)$ indicates the topological datum
$$[\cO_\cZ] = m [\cO_\cY\otimes \bC_\reg] + \sum_{j=0}^n \varepsilon_j [\cO_\pt\otimes \bC_{\zeta^j}] \in K(\cX),$$
where $\bC_{\zeta^j}$ denotes the 1-dimensional $\bZ_{n+1}$-representation with weight $j$, and $\bC_\reg$ denotes the regular representation.

$\mubar^i$'s are relative insertions specifying restrictions of $\cZ$ to boundary divisors, living in the Fock space
$$
H^*(\Hilb^m([\bC^2/\bZ_{n+1}])) \cong H^*(\Hilb^m(\cA_n)) = \cF_{\cA_n}.
$$
Note that this space is the same as the Fock space for $\cA_n$, but different from $\cF_{[\bC^2 / \bZ_{n+1}]}$.

The relative DT invariants are defined via $T$-localization with respect to the 2-dimensional torus $T$ acting on the fiber. One can form the generating function
$$\langle \vec\mu^1, \cdots, \vec\mu^r \rangle_m := \sum_{\varepsilon\geq 0} q_0^{\varepsilon_0} \cdots q_n^{\varepsilon_n} \langle \vec\mu^1, \cdots, \vec\mu^r \rangle_{m, \varepsilon},$$
and the corresponding \emph{reduced} DT $r$-point function
$$Z'_\DT([\bC^2/\bZ_{n+1}]\times \bP^1)_{\vec\mu^1, \cdots, \vec\mu^r} := \frac{\langle \vec\mu^1, \cdots, \vec\mu^r \rangle_m}{\langle \vec\mu^1, \cdots, \vec\mu^r \rangle_0},$$
where we omit the number $m$ which is always fixed and implicit in the formula.

In \cite{Zhou2} the following theorem is proved, indicating a close connection between relative reduced DT theory of $\cX$ and the quantum cohomology of $\Hilb^m([\bC^2/\bZ_{n+1}])$.

\begin{theorem}[DT/Hilb correspondence, Theorem 1.1 of \cite{Zhou2}]
  Given $\vec\mu, \vec\nu, \vec\rho \in \cF_{\cA_n}$, with
$$\vec\rho = (1,1)^m, \qquad (2,1)(1,1)^{m-2}, \qquad \textrm{or} \qquad (1,\omega_k)(1,1)^{m-1},$$
we have $T$-equivariantly,
  $$Z'_\DT([\bC^2/\bZ_{n+1}]\times \bP^1)_{\vec\mu, \vec\nu, \vec\rho} = \langle \vec\mu, \vec\nu, \vec\rho \rangle_{\Hilb^m([\bC^2/\bZ_{n+1}])},$$
  where the right hand side is the 3-point genus-zero GW invariants of $\Hilb^m([\bC^2/\bZ_{n+1}])$.
\end{theorem}

In this DT/Hilb correspondence, there is no change of variables or analytic continuation. The Fock spaces on both sides are the same, and the parameters $q$ are identical. Using this result, one can prove a crepant resolution correspondence between the relative DT theories of $\cX= [\bC^2/\bZ_{n+1}]\times \bP^1$ and $\cA_n\times \bP^1$.

Recall that there is an explicit isomorphism between cohomology groups
$$\Phi: H^*_\orb([\bC^2/\bZ_{n+1}]) \cong H^*(\cA_n),$$
$$e_0 \mapsto 1, \qquad e_i \mapsto \frac{\zeta^{i/2} - \zeta^{-i/2}}{n+1} \sum_{j=1}^n \zeta^{ij} \omega_j, \qquad 1\leq i\leq n,$$
where $\omega_1, \cdots, \omega_n \in H^2(\cA_n, \bQ)$ is the dual basis to the exceptional curves in $\cA_n$. Under this isomorphism, we can explicitly identify the Fock spaces $\cF_{[\bC^2/\bZ_{n+1}]} \cong \cF_{\cA_n}$.

For a curve $Z\subset \cA_n\times \bP^1$, its topological data are specified by the pair $(\chi, (\beta, m))$, where $\chi = \chi(\cO_Z)\in \bZ$ and $\beta\in H_2(\cA_n, \bZ)$ such that $m[\bP^1] + \beta = [Z]\in H_2(\cA_n\times \bP^1, \bZ)$. The generating function for the relative DT theory of $\cA_n\times \bP^1$ is defined in \cite{Mau-Ob2} as
$$Z_\DT(\cA_n\times \bP^1)_{\vec\mu^1, \cdots, \vec\mu^r} := \sum_{\chi, \beta} Q^\chi s_1^{(\beta, \omega_1)} \cdots s_n^{(\beta, \omega_n)} \langle \vec\mu^1, \cdots, \vec\mu^r \rangle^{\cA_n\times \bP^1}_{\chi, (\beta, m)},$$
and the \emph{reduced} partition function $Z'_\DT$ is defined by quotient out the degree $0$ contribution.

In \cite{Zhou2} the DT crepant resolution correspondence was stated for the $T^\pm$-equvariant theories, where $T^\pm \subset T$ is the anti-diagonal torus. To match it with our results here, we need to consider a slight variant of that. The $\beta = 0$ part of the DT partition function for $\cA_n \times \bP^1$ corresponds to the $\varepsilon_0 = \cdots = \varepsilon_n$ part of the DT partition function for $[\bC^2 / \bZ_{n+1}] \times \bP^1$.

\begin{definition}
	We call
	$$
	Z'_{\DT, \varepsilon-\irreg} (\cA_n\times \bP^1)_{\vec\mu, \vec\nu, \vec\rho} := Z'_\DT(\cA_n\times \bP^1)_{\vec\mu, \vec\nu, \vec\rho} - Z'_{\DT, \varepsilon_0 = \cdots = \varepsilon_n}(\cA_n\times \bP^1)_{\vec\mu, \vec\nu, \vec\rho}
	$$
	the \emph{$\varepsilon$-irregular} part of the partition function, which under the crepant resolution, matches with the $\beta \neq 0$ part for $\cA_n \times \bP^1$.
\end{definition}

\begin{theorem}[DT crepant resolution, Variant of Theorem 1.2 in \cite{Zhou2}] \label{DT-CRC}
  Given $\vec\mu, \vec\nu, \vec\rho\in \cF_{\cA_n}$, with
  $$
  \vec\rho = (1,1)^m, \qquad (2,1)(1,1)^{m-2}, \qquad \textrm{or} \qquad (1,\omega_k)(1,1)^{m-1},
  $$
 we have $T$-equivariantly, under the change of variables
 $$Q = q_0 q_1 \cdots q_n, \qquad s_i = q_i, \qquad i\geq 1,$$
  \begin{enumerate}[1)]
 	\item when $l(\vec\mu)+ l(\vec\nu)=2$,  and $\vec\rho = (1,1)^m$ or $(2,1)(1,1)^{m-2}$,
 	$$
 	Q^{-m} Z'_\DT(\cA_n\times \bP^1)_{\vec\mu, \vec\nu, \vec\rho} = Z'_\DT([\bC^2/\bZ_{n+1}]\times \bP^1)_{\vec\mu, \vec\nu, \vec\rho};
 	$$
 	\item when $l(\vec\mu)+ l(\vec\nu)\geq 3$, or $\vec\rho =(1,\omega_k)(1,1)^{m-1}$,
 	$$
 	Q^{-m} Z'_{\DT, \beta\neq 0}(\cA_n\times \bP^1)_{\vec\mu, \vec\nu, \vec\rho} = Z'_{\DT, \varepsilon-\irreg} ([\bC^2/\bZ_{n+1}]\times \bP^1)_{\vec\mu, \vec\nu, \vec\rho}.
 	$$
 \end{enumerate}
\end{theorem}

\begin{remark}
Analogous to the GW theories, the $\beta = 0$ or $\varepsilon$-regular part of the DT partition can also be reduced to $Z'_{\DT} (\bC^2 \times \bP^1)_{\mu, \nu, \rho}$ and the classical $n$-point functions of the state space, i.e., $T$-equivariant cohomologies of $\Hilb^m (\cA_n)$ and $\Hilb^m ([\bC^2 / \bZ_{n+1}])$. The two cohomologies are isomorphic as rings if one reduces to the $T^\pm$-equivariant theory (which is the approach taken in \cite{Zhou2}), but not in general. Here in Theorem \ref{DT-CRC}, by considering Case 1) and 2), we also avoid those exceptional unmatched classical $n$-point functions.
\end{remark}

The GW/DT correspondence for $\cA_n\times \bP^1$ was proved in \cite{Mau-Ob2}.

\begin{theorem}[GW/DT correspondence for $\cA_n\times \bP^1$, Theorem 1.1 in \cite{Mau-Ob2}] \label{GW-DT-An}
  Given $\vec\mu, \vec\nu, \vec\rho \in \cF_{\cA_n}$, with
  $$\vec\rho = (1,1)^m, \qquad (2,1)(1,1)^{m-2}, \qquad \textrm{or} \qquad (1,\omega_k)(1,1)^{m-1},$$
  we have
  $$(-iz)^{l(\mu)+l(\nu)+l(\rho)-m} Z'_\GW(\cA_n\times \bP^1)_{\vec\mu, \vec\nu, \vec\rho} = (-Q)^{-m} Z'_\DT(\cA_n\times \bP^1)_{\vec\mu, \vec\nu, \vec\rho},$$
  under the change of variables $Q=-e^{iz}$.
\end{theorem}

Combining Theorem \ref{GW-CRC}, Theorem \ref{DT-CRC} and Theorem \ref{GW-DT-An}, we obtain the following main theorem.

\begin{theorem}[GW/DT correspondence] \label{GW-DT}
  Given $\mubar, \nubar, \rhobar\in \cF_{[\bC^2/\bZ_{n+1}]}$, with
  $$
  \rhobar = (1,0)^m, \qquad (2,0)(1,0)^{m-2}, \qquad \textrm{or} \qquad (1,k)(1,0)^{m-1},
  $$
 let $\vec\mu$, $\vec\nu$, $\vec\rho$ be their correspondents in $\cF_{\cA_n}$. Then
  under the change of variables
  $$
  Q =q_0 q_1 \cdots q_n = -e^{iz}, \qquad q_j = \zeta \exp \left( \frac{1}{n+1} \sum_{a=1}^n (\zeta^{a/2}- \zeta^{-a/2}) \zeta^{ja} x_a \right), \qquad 1\leq j\leq n,
  $$
  we have
 \begin{enumerate}[1)]
 	\item when $l(\mubar)+ l(\nubar)=2$,  and $\rhobar = (1,0)^m$ or $(2,0)(1,0)^{m-2}$,
 	$$
 	(-iz)^{l(\mu)+l(\nu)+l(\rho)-m} Z'_\GW([\bC^2/\bZ_{n+1}]\times \bP^1)_{\mubar, \nubar, \rhobar} = (-1)^m Z'_\DT([\bC^2/\bZ_{n+1}]\times\bP^1)_{\vec\mu, \vec\nu, \vec\rho};
 	$$
 	\item when $l(\mubar)+ l(\nubar)\geq 3$, or $\rhobar =(1,k)(1,0)^{m-1}$,
 	$$
 	(-iz)^{l(\mu)+l(\nu)+l(\rho)-m} Z'_\GW([\bC^2/\bZ_{n+1}]\times \bP^1)_{\mubar, \nubar, \rhobar} = (-1)^m Z'_{\DT, \varepsilon-irreg} ([\bC^2/\bZ_{n+1}]\times\bP^1)_{\vec\mu, \vec\nu, \vec\rho}.
 	$$
 \end{enumerate}
\end{theorem}

The formula for change of variables here is exactly the same as the GW/DT correspondence for CY local $\bZ_{n+1}$-gerby curves, proved in \cite{Ross-Zong}. In a future work, we will see that the relative GW/DT theory for $[\bC^2/\bZ_{n+1}]\times \bP^1$, and more generally, for nontrivial local gerby curves, is closely related to the GW/DT topological vertex in the CY case. The former is also known as the \emph{capped} vertex.

\subsection{Equivalences of theories}

As a summary of all existing results, we obtain the following diagram, indicating relationships between various theories.

\begin{figure}[h]
\begin{center}
\psfrag{A}{$\QH(\Sym(\cA_n))$}
\psfrag{B}{$\GW(\cA_n\times \bP^1)$}
\psfrag{C}{$\DT(\cA_n\times \bP^1)$}
\psfrag{D}{$\QH(\Hilb(\cA_n))$}
\psfrag{E}{$\QH(\Sym([\bC^2/\bZ_{n+1}]))$}\psfrag{F}{$\GW([\bC^2/\bZ_{n+1}]\times \bP^1)$}
\psfrag{G}{$\DT([\bC^2/\bZ_{n+1}]\times \bP^1)$}
\psfrag{H}{$\QH(\Hilb([\bC^2/\bZ_{n+1}]))$}
\hspace*{-2.8cm}\includegraphics[scale=0.5]{cube2.eps}
\end{center}
\end{figure}


A few remarks on the equivalences:

\begin{enumerate}[1)]
  \setlength{\parskip}{1ex}

  \item Each line here means a equivalence of theories, in the sense that the 3-point functions of the two theories connected by the line are equal, provided that one insertion is the identity or a divisor. The equality here is possibly up to change of variables and analytic continuation.

  \item $45^\circ$ lines indicate GW/Sym and DT/Hilb correspondences. All these correspondences are \emph{without} change of variables, and are proved via some identification of parts of the moduli's. In other words, these are the more geometric correspondences, and the holomorphic symplectic structures on $\cA_n$ and $[\bC^2/\bZ_{n+1}]$ play a crucial role.

  \item Horizontal lines in the front face are GW/DT correspondences. The change of variables involves exponential maps.

  \item Vertical lines in the front face are crepant resolution correspondences for relative GW and DT theories.

  \item Lines in the back face indicate the crepant resolution/transformation correspondence for the (partial) Hilbert--Chow morphisms
      $$\xymatrix{
      \Hilb(\cA_n) \ar[r]^-\cong & \Hilb([\bC^2/\bZ_{n+1}]) \ar[r] & \Sym(\cA_n) \ar[r] & \Sym([\bC^2/\bZ_{n+1}]),
      }$$
      where $\Hilb(\cA_n)$ and $\Hilb([\bC^2/\bZ_{n+1}])$ are mutually symplectic flops, related by wall-crossings.
\end{enumerate}

\subsection{Generation conjecture} \label{sec-gen-conj}

Consider the Fock spaces $\cF_{\cA_n}\cong \cF_{[\bC^2/\bZ_{n+1}]}$. Let $D_0, \cdots, D_n$ be a basis of the divisors, for example, the obvious ones we have taken in previous sections. 3-point functions of each theory described above define a product structure on $\cF$, and our previous results state that the operators $M_{D_i}$ of multiplication by divisors of each theory are equivalent.

There is a further step one can make to extend the results to general $r$-point functions. By the degeneration formula, it suffices to know all 3-point functions, with arbitrary insertions, instead of only divisors. In other words, we need to know the multiplication operator $M_\gamma$ for any class $\gamma\in \cF$. As in \cite{Mau-Ob2}, we make the following conjecture.

\begin{conjecture}
  The joint eigenspaces for the operators $M_{D_i}$, $0\leq i\leq n$ are 1-dimensional for all $m>0$.
\end{conjecture}

Under this conjecture, the ring $\cF$ can be generated by divisors $D_i$, and our results extend.

\begin{corollary*}
  Under the above conjecture, all theories are equivalent in the sense that all $r$-point correlation functions are equal for arbitrary $n$.
\end{corollary*}

\appendix

\section{Moduli stack of twisted curves with labelled nodes and markings}

In this appendix, in order to compare obstruction theories, we introduce the following auxillary stacks. Let $n$, $g$, $N$ and $M$ be positive integers.

\begin{definition}
We define a stack $\fM_{g,N,M}^\Delta$ as follows. For any scheme $S$, an object of the category $\fM_{g,N,M}^\Delta (S)$ consists of a family of prestable twisted curves $\cC$ over $S$, such that the connected components of its markings and nodes are labelled by one of the following four types:
\begin{itemize}

\item \emph{non-relative} and \emph{relative} markings, the number of which we denote by $N$ and $M$ respectively;

\item \emph{non-special} and \emph{special} nodes.

\end{itemize}
As usual, markings are numbered and nodes are not. An automorphism of an object is an automorphism of the twisted curve, which preserves the markings and the labels. Define $\fM^\Delta_{g,N,M} (B\bZ_{n+1})$ to be the stack of twisted curves, labelled as above, together with a (not necessarily representable) morphism to $B\bZ_{n+1}$.
\end{definition}

Let $\fM_{g, N+M}$ (resp. $\fM_{g, N+M}(B\bZ_{n+1})$) be the smooth Artin stack of $(N+M)$-pointed prestable twisted curves (resp. together with a map to $B\bZ_{n+1}$). The stack $\fM_{g, N, M}^\Delta$ (resp. $\fM^\Delta_{g,N,M} (B\bZ_{n+1})$) can be interpreted as a stack with labels, in the sense of \cite{Cos, ACFW} 
\footnote{In \cite{Cos}, the labels are put on irreducible components of domain curves to remember the degrees. The construction in Section 7.1 of \cite{ACFW} is more similar to our case, where labels are added at nodes and markings.} 
constructed from $\fM_{g, N+M}$ (resp. $\fM_{g,N+M} (B\bZ_{n+1})$). There is a forgetful map $\fM^\Delta_{g,N,M} \to \fM_{g, N+M}$ (resp. $\fM^\Delta_{g,N,M} (B\bZ_{n+1}) \to \fM_{g, N+M} (B\bZ_{n+1})$), which is \'etale. There is also an \'etale map $F: \fM_{g, N, M}^\Delta (B\bZ_{n+1}) \to \fM^\Delta_{g,N, M}$ forgetting the morphism to $B\bZ_{n+1}$. So these stacks are also smooth Artin stacks, with the same deformation theory.

Now let's define two ``coarsification" maps for $\fM_{g, N, M}^\Delta$. 

\begin{enumerate}[1)]
	
\setlength{\parskip}{1ex}

\item{\bf Coarsify non-special nodes and non-relative markings.} We define a map
$$
\pi :  \fM_{g, N, M}^\Delta \to \fM_{g,N,M}^\Delta
$$
as follows. Let $\cC \to S$ be an object of $\fM_{g,N,M}^\Delta (S)$, and $\bar C \to S$ be its coarse moduli space. Let $\cZ \subset \cC$ be the locus of all \emph{non-special nodes and non-relative markings}, and $\bar Z$ be its coarse moduli space. Take an \'etale neighborhood $\bar U$ of $\bar Z$, and let $\cU := \bar U \times_{\bar C} \cC$. By possibly shrinking $\bar U$, we may assume that $\cU \backslash \cZ \to \bar U \backslash \bar Z$ is an isomorphism. Let $C$ be the twisted curve obtained by gluing $\bar U$ with $\cC \backslash \cZ$. We define the image of $\cC$ under $\pi$ to be $C \to S$. 

\item {\bf Coarsify and forget relative markings.} We define a map
$$
\rho: \fM_{g, N, M}^\Delta \to \fM^\Delta_{g, N, 0} .
$$
Let $\cC \to S$ be an object of $\fM_{g,N,M}^\Delta (S)$. We apply the similar procedure to all the relative markings, and then forget those relative markings. 

\end{enumerate}

\begin{lemma} \label{Cart-1}
The diagram
$$
\xymatrix{
\fM_{g, N, M}^\Delta \ar[d]_-\pi \ar[r]^-\rho & \fM^\Delta_{g, N, 0} \ar[d]^-{\pi } \\
\fM_{g, N, M}^\Delta \ar[r]^\rho & \fM_{g,N, 0}^\Delta
}
$$
is commutative and Cartesian. Moreover, the map $\pi$ is quasi-finite and flat.
\end{lemma}

\begin{proof}
The Cartesian-ness of the diagram follows from the fact that one can recover the twisted curve $\cC \to S$ from its images under $\pi$ and the horizontal map together. It is clear that $\pi$ in the second column is quasi-finite.

For the flatness, let $\fM \subset \fM_{g,N,0}^\Delta$ be the open and closed substack where stabilizer groups along all non-special nodes and non-relative markings are trivial. We note that $\pi$ factors through and surjectively maps onto $\fM$. Choose a smooth covering $U \to \fM$, and form the Cartesian diagram
$$
\xymatrix{
	V \ar[r] \ar[d]_{\pi_U} &  \fM^\Delta_{g, N, 0}   \ar[d]^{\pi} \\
	U \ar[r]  & \fM .
}
$$
It follows that both $U$ is a smooth scheme, and $V$ is a smooth DM stack; $\pi_U$ is surjective, and admits 0-dimensional fibers. Passing to an \'etale covering $\tilde V \to V$, we obtain a map $\tilde \pi_U: \tilde V \to U$ between smooth schemes, which is surjective and admits 0-dimensional fibers. By Proposition 6.1.5 in \cite{EGA}, $\tilde \pi_U$, and hence $\pi_U$, is flat. The flatness of $\pi$ then follows from Lemma 06PZ in \cite{St}.
\end{proof}

Now let $X$ be a smooth projective scheme, and $D$ be a smooth divisor in $X$. Given a twisting choice $\fr$, let $\Mbar^\fr_{g, N, M, \Gamma} (X, D)$ be the moduli of \emph{transversal} relative stable maps into $\fr$-twisted expanded pairs of $(X,D)$, where $N$, $M$ are the numbers of relative and non-relative markings, and $\Gamma$ denotes other topological data, such as degree. For convenience, we will omit $\Gamma$ in the notation.

We also consider the pair $(B\bZ_{n+1} \times X, B\bZ_{n+1} \times D)$. There are forgetful maps
$$
\overline\cM_{g, N, M}^\fr (B\bZ_{n+1} \times X, B\bZ_{n+1} \times D) \to \fM_{g, N,M}^{\Delta} (B\bZ_{n+1}), \qquad \overline\cM_{g,N,M}^\fr (X, D) \to \fM_{g, N,M }^{\Delta} ,
$$
defined by taking the domains (and also the morphism to $B\bZ_{n+1}$ for the first map), where markings and nodes are labelled as follows:
\begin{itemize}

\setlength{\parskip}{1ex}	

\item relative and non-relative markings are still labelled as they are;

\item nodes mapping into the nodes of the $\fr$-twisted expanded pairs are labelled as \emph{special} nodes;

\item nodes mapping into the smooth locus of the $\fr$-twisted expanded pairs are labelled as \emph{non-special} nodes.
\end{itemize}

These two maps are well-defined, since the type of a node or marking is unchanged in a continuous family of relative stable maps. 

\begin{lemma} \label{Cart-2}
Let $\pi^\fr:\overline\cM_{g, N, M}^\fr (B \bZ_{n+1} \times X, B\bZ_{n+1} \times D) \to \overline\cM_{g,N, M}^\fr (X, D)$ be the map sending a $\fr$-twisted relative stable map to its relative coarse moduli space with respect to $\fr$-twisted expanded pairs of $(X, D)$. The following diagram is commutative and Cartesian:
$$
	\xymatrix{
		\overline\cM_{g, N, M}^\fr (B \bZ_{n+1} \times X, B\bZ_{n+1} \times D) \ar[r] \ar[d]_{\pi^\fr} & \fM_{g, N, M }^{\Delta} \ar[d]^{\pi \circ F} (B\bZ_{n+1}) \\
		\overline\cM_{g,N, M}^\fr (X, D) \ar[r]  & \fM_{g, N, M }^{\Delta} .
	}
$$
\end{lemma}

\begin{proof}
Let $\cX \to S$ be a family of $\fr$-twisted expanded pairs over $S$, and $\tilde f: \cC \to B\bZ_{n+1} \times \cX$ be a transversal relative stable map. Let $f: C \to \cX$ be its image under $\pi^\fr$. We look at the local behavior of $\tilde f$ and $f$ near all the markings and nodes.
\begin{itemize}
\setlength{\parskip}{1ex}
\item Near a special node, $\tilde f$ is of the form
$$
\left[ \Spec \frac{\cO_S [u,v]}{(uv-s)} \middle/ \bZ_a \right] \to B\bZ_{n+1} \times \left[ \Spec \frac{\cO_{S \times D} [x,y]}{(xy-t)} \middle/ \bZ_r \right] ,
$$
for $t\in \cO_D$, $s\in \cO_S$. By transversality, we must have $a|r$, and the map is $(x,y) \mapsto (u, v)$. In particular, the group homomorphism $\bZ_a \to \bZ_r$ is injective. Hence the coarsification $\cC \to C$ is an isomorphism near a special node. The case is similar for a relative marking.

\item Near a non-special node, $\tilde f$ is of the form
$$
\left[ \Spec \frac{\cO_S [u,v]}{(uv-s)} \middle/ \bZ_a \right] \to B\bZ_{n+1} \times U
$$
where $U$ is some smooth scheme. Then $\cC \to C$ is a is the twisting along a node (in the sense of Section 1.4 in \cite{Abr-Fan}) of order $a$ near a non-special node. The case is similar for a non-relative marking.
\end{itemize}
Summarizing all the cases, we see that under the map $\pi^\fr$, the twisted curve $\cC$ is coarsified to $C$ in exactly the same way as we define $\pi \circ F$, i.e. all non-special nodes and non-relative markings are coarsified. Hence the diagram is commutative. The Cartesian-ness is easy to check.
\end{proof}

We now compare the obstruction theories. Let $\fT$ be the Artin stack parametrizing all $\fr$-twisted expanded pairs of $(X, D)$, $\fX$ is the universal family over $\fT$, and $\fC$ is the universal curve over $\Mbar^\fr_{g, N,M} (X, D)$. Let $\fD \subset \fX$ be the universal relative divisor, and $\Sigma \subset \fC$ be the divisor of \emph{non-relative markings}\footnote{In \cite{Abr-Fan} this is denoted by $\Sigma'$, and the $\Sigma$ there is $\Sigma + f^{-1} \fD$ in our notation.}. Consider the universal diagram (and also for $(B\bZ_{n+1} \times X, B\bZ_{n+1} \times D)$)
$$
\xymatrix{
	\fC \ar[r]^-f \ar[d]_-p & \fX \ar[d] & \widetilde\fC \ar[r]^-{\tilde f} \ar[d]_-{\tilde p} & B\bZ_{n+1} \times \fX \ar[d] \\
	\Mbar^\fr_{g, N, M} (X, D) \ar[r] & \fT & \overline\cM_{g, N, M}^\fr (B \bZ_{n+1} \times X, B\bZ_{n+1} \times D) \ar[r] & \fT  .
}
$$
In \cite{Abr-Fan}, the relative perfect obstruction theory on $\Mbar_{g, N,M}^\fr (X,D)$ over $\fT$ is given by the complex
$$
\bE^\bullet := Rp_* \left( [f^* L_{\fX / \fT} \to L_{\fC / \Mbar^\fr_{g, N, M} (X, D)} (\Sigma) ] \otimes \omega_{\fC / \Mbar^\fr_{g, N, M} (X, D)} \right)
$$
where $\omega$ denotes the dualizing line bundle, and $[f^* L_{\fX / \fT} \to L_{\fC / \Mbar^\fr_{g, N, M} (X, D)} (\Sigma) ]$ has degrees concentrated in $[-1,0]$. Similarly, there is an obstruction theory for $(B\bZ_{n+1} \times X, B\bZ_{n+1} \times D)$, which we denote by $\widetilde\bE^\bullet$.

\begin{remark}
	Only $\Sigma$ but not the full divisor of markings $\Sigma + f^{-1} \fD$ is used in the obstruction. The reason is that the relative divisor $f^{-1} \fD$ is completely determined by the relative stable map $f$ and the divisor $\fD$ in the target. Hence relative markings are not free to deform in the moduli of domain curves.
\end{remark}

\begin{lemma} \label{pullback-app}
	$$
	\left[ \Mbar^\fr_{g, N, M}(B\bZ_{n+1}\times X, B\bZ_{n+1} \times D) \right]^\vir = (\pi^\fr)^* \left[ \Mbar^\fr_{g,N, M} (X, D) \right]^\vir.
	$$
\end{lemma}

\begin{proof}
We have the diagram
$$
\xymatrix{
	\overline\cM_{g, N, M}^\fr (B \bZ_{n+1} \times X, B\bZ_{n+1} \times D) \ar[r] \ar[d]_{\pi_1^\fr} &  \fM_{g, N, M}^\Delta (B\bZ_{n+1}) \ar[d]_-F  & \\
	\Nbar \ar[r]^-{G_1} \ar[d]_{\pi^\fr_2} & \fM_{g, N, M}^\Delta \ar[r]^-\rho \ar[d]_\pi & \fM_{g, N, 0}^\Delta \ar[d]^-\pi \\
	\overline\cM_{g,N, M}^\fr (X, D) \ar[r]^-{G_2}  &  \fM_{g, N, M}^\Delta \ar[r]^-\rho & \fM_{g,N, 0}^\Delta ,
}
$$
where $F$ is the map forgetting the morphism to $B\bZ_{n+1}$, and $\Nbar$ is defined to be the fiber product $\overline\cM_{g,N, M}^\fr (X, D) \times_{\fM_{g, N, M}^\Delta, \pi} \fM_{g, N, M}^\Delta$, which parameterizes diagrams of the form $C' \xrightarrow{u} C \xrightarrow{f} X[k](\fr)$, where $u: C' \to C$ is a coarsification of twisted curves at all non-special nodes and non-relative markings, and $f$ is a relative stable map. Lemma \ref{Cart-1} and \ref{Cart-2} implies that all three squares in the diagram are Cartesian. All vertical maps are flat, and $F$ is \'etale. 

Consider the following alternative perfect obstruction theories on $\overline\cM_{g,N, M}^\fr (X, D)$:
$$
\bE^\bullet_\rel := Rp_* \left( f^* L_{\fX / \fT} \otimes \omega_{\fC / \Mbar_{g,N, M}^\fr (X, D)} \right) [1].
$$
The relation between $\bE^\bullet_\rel$ and  $\bE^\bullet$ is 
$$
\xymatrix{
\bE^\bullet \ar[r] & \bE^\bullet_\rel \ar[r] & Rp_* \left( L_{\fC / \Mbar_{g, N, M}^\fr (X, D)} (\Sigma) \otimes \omega_{\fC / \Mbar^\fr_{g, N, M} (X, D)} \right) [1] \ar[r]^-{[1]} & 
}
$$
where the third object is actually $G_2^* \rho^* L_{\fM_{g,N, 0}^\Delta} [1]$, which is responsible for the deformation theory of the domain curves. Therefore, $\bE^\bullet_\rel$ is a relative perfect obstruction theory over the base $\fM_{g,N, 0}^\Delta \times \fT$, and defines the same virtual classes defined as $\bE^\bullet$ (for a proof, see Proposition A.1.(1) in \cite{Bry-Leu}). 

The universal diagram for $\Nbar$ is the left of the following
$$
\xymatrix{
\fC' \ar[dr]^-u \ar@/_1pc/[ddr]_{p'} \ar@/^1pc/[drr]^{f'} & & \\
& \Nbar \times_{\Mbar^\fr_{g,N,M} (X,D)} \fC \ar[d]^-{\pr_1} \ar[r]_-{f \circ \pr_2} & \fX \ar[d]  &&  \Nbar \times_{\Mbar^\fr_{g,N,M} (X,D)} \fC \ar[r]^-{\pr_2} \ar[d]_-{\pr_1} & \fC \ar[d]^-p \\
& \Nbar \ar[r] & \fT, && \Nbar \ar[r]^-{\pi_2^\fr} & \Mbar^\fr_{g,N,M} (X,D) . 
}
$$
Proposition 7.2 of \cite{Beh-Fan} implies that $(\pi_2^\fr)^* \bE_\rel^\bullet$ is a relative perfect obstruction theory on $\Nbar$, over $\fM_{g,N, 0}^\Delta \times \fT^\fr$, and it defines the virtual class 
\begin{equation} \label{N-1}
[\Nbar]^\vir = (\pi_2^\fr)^* [\overline\cM_{g,N, M}^\fr (X, D)]^\vir . 
\end{equation}
The following computation shows that $(\pi_2^\fr)^* \bE_\rel^\bullet$ agrees with the relative obstruction theory on $\Nbar$:
\begin{eqnarray*}
R (p')_* \left( (f')^* L_{\fX / \fT} \otimes \omega_{\fC' / \Nbar} \right) [1] &=& R(\pr_1)_* \circ u_* \left( u^* \pr_2^* f^*  L_{\fX / \fT} \otimes \omega_{\fC' / \Nbar} \right) [1] \\
&=& R(\pr_1)_*  \left(  \pr_2^* f^*  L_{\fX / \fT} \otimes u_* \omega_{\fC' / \Nbar} \right) [1] \\
&=& R(\pr_1)_*  \pr_2^*  \left( f^*  L_{\fX / \fT} \otimes \omega_{\fC / \Mbar_{g,N, M}^\fr (X, D)} \right) [1] \\
&=& (\pi_2^\fr)^* Rp_* \left( f^* L_{\fX / \fT} \otimes \omega_{\fC / \Mbar_{g,N, M}^\fr (X, D)} \right) [1]
\end{eqnarray*}
where we used $u_* \omega_{\fC' / \Nbar} = \omega_{\pr_1} = \pr_2^* \omega_{\fC / \Mbar_{g,N, M}^\fr (X, D)}$, since $u$ is the coarsification map at all non-special nodes and non-relative markings. In other words, $[\Nbar]^\vir$ agrees with the virtual class defined by the natural relative obstruction theory on $\Nbar$. 

We then apply the previous trick conversely for $\Nbar$, and replace it by 
$$
\bF^\bullet := R (p')_* \left( [ (f')^* L_{\fX / \Nbar} \to L_{\fC' / \Nbar} (\Sigma') ] \otimes \omega_{\fC' / \Nbar} \right) , 
$$ 
which fits in the distinguished triangle $\bF^\bullet  \to (\pi_2^\fr)^* \bE^\bullet_\rel \to G_1^* \rho^* L_{\fM_{g,N, 0}^\Delta} [1] \xrightarrow{[1]}$. It is a relative perfect obstruction theory on $\Nbar$ over $\fT^r$, which defines the same virtual class $[\Nbar]^\vir$. 
	
Now one can check directly that $(\pi_1^\fr)^* \bF^\bullet \cong \widetilde\bE^\bullet$, since the universal curve and universal target for $\overline\cM_{g, N, M}^\fr (B \bZ_{n+1} \times X, B\bZ_{n+1} \times D)$ are simply the pullback along $\pi_1^\fr$ of those for $\Nbar$. One can then consider the triple
$$
\xymatrix{
\overline\cM_{g, N, M}^\fr (B \bZ_{n+1} \times X, B\bZ_{n+1} \times D) \ar[dr] \ar[rr]^-{\pi_1^\fr} && \Nbar \ar[dl] \\
& \fT , 
}
$$
and obtain 
\begin{equation} \label{N-2}
[\overline\cM_{g, N, M}^\fr (B \bZ_{n+1} \times X, B\bZ_{n+1} \times D)]^\vir = (\pi_1^\fr)^! [\Nbar]^\vir = (\pi_1^\fr)^* [\Nbar]^\vir,
\end{equation}
where $(\pi_1^\fr)^!$ is the virtual pullback in the sense of \cite{Man}, which coincides with $(\pi_1^\fr)^*$ in this case. The lemma then follows from (\ref{N-1}) and (\ref{N-2}). 
\end{proof}

\bibliographystyle{plain}
\bibliography{reference}

\end{document}